\documentclass[12pt,a4paper]{article}

\usepackage[english]{babel}

\usepackage{mathptmx} 
\usepackage[scaled=0.90]{helvet} 
\usepackage{courier} 
\normalfont

\usepackage[T1]{fontenc}
\usepackage{amssymb,latexsym,amsmath,amsthm}
\usepackage{makeidx}
\usepackage{a4}
\usepackage{dsfont}
\usepackage[usenames,dvipsnames]{color}
\usepackage{comment}

 \theoremstyle{plain}
 \newtheorem{Thm}{Theorem}[section]
 \newtheorem{Cor}[Thm]{Corollary}
 \newtheorem{Lemma}[Thm]{Lemma}
 \newtheorem{Prop}[Thm]{Proposition}
 \newtheorem{Conj}[Thm]{Conjecture}
 \newtheorem{Prelim}[Thm]{Preliminary}

 \theoremstyle{definition}
 \newtheorem{Rem}[Thm]{Remark}
 
 \newtheorem{Defi}[Thm]{Definition}

 \newenvironment{thm}[0]
    {\begin{Thm}\noindent}%
    {\end{Thm}}
 \newenvironment{defi}[0]
    {\begin{Defi}\noindent\rm}%
    {\end{Defi}}
 \newenvironment{cor}[0]
    {\begin{Cor}\noindent}%
    {\end{Cor}}
 \newenvironment{lemma}[0]
    {\begin{Lemma}\noindent}%
    {\end{Lemma}}
 \newenvironment{prop}[0]
    {\begin{Prop}\noindent}%
    {\end{Prop}}
 \newenvironment{rem}[0]
    {\begin{Rem}\noindent\rm}%
    {\end{Rem}}
    {\end{Conj}}
    {\end{Prelim}}
 \numberwithin{Thm}{section}
 \numberwithin{equation}{section}

\def\naam#1{\label{#1}}

\def\bib#1{\cite{#1}}

\def\medno{\medbreak\noindent}
\def\text#1{\;\;\;\;{\rm \hbox{#1}}\;\;\;\;}
\def\qquad{\quad\quad}

\def\itema{\item[{\rm (a)}]}
\def\itemb{\item[{\rm (b)}]}
\def\itemc{\item[{\rm (c)}]}
\def\itemd{\item[{\rm (d)}]}
\def\iteme{\item[{\rm (e)}]}
\def\itemf{\item[{\rm (f)}]}

\def\msy#1{{\mathbb #1}}
\def\C{{\msy C}}
\def\N{{\msy N}}

\def\R{{\msy R}}

\def\ga{\alpha}
\def\gb{\beta}
\def\gd{\delta}
\def\geps{\varepsilon}
\def\gf{\varphi}

\def\gl{\lambda}

\def\gs{\sigma}

\def\gD{\Delta}

\def\gS{\Sigma}

\def\fa{{\mathfrak a}}
\def\fb{{\mathfrak b}}
\def\fg{{\mathfrak g}}
\def\fh{{\mathfrak h}}

\def\fk{{\mathfrak k}}
\def\fl{{\mathfrak l}}
\def\fm{{\mathfrak m}}
\def\fn{{\mathfrak n}}

\def\fp{{\mathfrak p}}
\def\fq{{\mathfrak q}}
\def\fs{{\mathfrak s}}

\def\to{\rightarrow}
\def\Re{\mathrm{Re}\,}

\def\inp#1#2{\langle#1\,,\,#2\rangle}

\def\Ad{\mathrm{Ad}}
\def\End{\mathrm{End}}

\def\Hom{\mathrm{Hom}}

\def\ad{\mathrm{ad}}
\def\after{\,{\scriptstyle\circ}\,}

\def\pr{\mathrm{pr}}

\def\tr{\mathrm{tr}\,}

\def\iq{{\mathrm q}}
\def\iC{{\scriptscriptstyle \C}}

\def\cA{{\mathcal A}}

\def\cC{{\mathcal C}}
\def\cD{{\mathcal D}}

\def\cH{{\mathcal H}}

\def\cM{{\mathcal M}}

\def\cP{{\mathcal P}}

\def\cW{{\mathcal W}}

\def\cZ{{\mathcal Z}}

\def\col{\,:\,}
\def\faq{\fa_\iq}
\def\faqd{\fa_\iq^*}
\def\faqdc{\fa_{\iq\iC}^*}
\def\fadc{\fa_{\iC}^*}
\def\fad{\fa^*}

\def\Aq{A_\iq}

\def\ev{\mathrm{ev}}
\def\Ind{\mathrm{Ind}}

\def\Cartan{\theta}

\def\Vtau{V_\tau}

\def\Aqp{A_\iq^+}

\def\GL{\mathrm{GL}}
\def\embeds{\hookrightarrow}

\def\reg{\mathrm{reg}}
\def\dotvar{\, \cdot\,}

\def\faPdc{\fa_{P\iC}^*}

\def\bs{\backslash}
\def\bp{{}^\backprime}

\def\1{\mathbf{1}}

\def\parabsgsA{\cP_\gs(A)}
\def\parabsgsAq{\cP_\gs(\Aq)}
\def\rmn{{\rm n}}
\def\HP{{H_P}}
\def\rmi{{\rm i}}

\def\rhoQh{\rho_{Q\!\ih}}
\def\rhoPh{\rho_{P\!\ih}}
\def\rhoP{\rho_P}
\def\one{{\mathds 1}}
\def\fu{{\rm fu}}
\def\dM{\widehat M}

\def\dMzerofu{ \widehat M_{0\fu}}
\def\Aq{{A_\iq}}

\def\parabs{\cP}
\def\eps{\varepsilon}
\def\hatotimes{\widehat\otimes}
\def\HC{{\scriptstyle{\rm HC}}}

\def\rmI{{\rm I}}

\def\Aqp{A_\iq^+}
\def\cAMtwo{{\mathcal A}_{M,2}}
\def\bp{{}^\backprime}
\def\Ah{A_{\rm h}}
\def\hide#1{\mbox{}}
\def\ssLG{{\scriptscriptstyle L \bs G}}
\def\ssLH{{\scriptscriptstyle L \bs H}}
\def\ssHG{{\scriptscriptstyle H \bs G}}

\def\ih{{\,\mathrm{h}}}
\def\fah{\fa_{\ih}}
\def\parabsA{\cP(A)}
\def\fahd{{\fah^*}}

\def\prq{{\rm pr}_{\rm q}}
\def\holset{\Omega}
\def\faqreg{\fa_\iq^{\rm reg}}
\def\fap{\fa^+}
\def\faqp{\fa_\iq^+}
\def\parabsaboveQ{\cP_\gs(A, Q)}
\def\faqdp{\fa_{\iq}^{*+}}
\def\sPi{{\scriptscriptstyle \Pi}}
\def\convergeset{\Upsilon}
\def\pirep{\pi}
\def\Cminf{C^{-\infty}}
\def\hatUpsilon{\widehat\Upsilon}
\def\fH{{\mathfrak H}}
\def\inj{i^\#}
\def\injP{i^\#_{\gl}}
\def\proj{p^\#}
\def\projP{p^\#_{\gl}}
\begin{document}
\title{Normalizations of Eisenstein integrals for reductive symmetric spaces}
\date{June 30, 2016}
\author{Erik P. van den Ban, Job J. Kuit \footnote{Supported by the Danish National Research Foundation through the Centre for Symmetry and Deformation (DNRF92).}}
\maketitle
\begin{abstract}
We construct minimal Eisenstein integrals for a reductive symmetric space $G/H$ as matrix
coefficients of the minimal principal series of $G.$ The Eisenstein integrals thus obtained
include those from the $\gs$-minimal principal series.
In addition, we obtain related Eisenstein integrals, but with different normalizations.
Specialized to the case of the group, this wider class includes Harish-Chandra's
 minimal Eisenstein integrals.
\end{abstract}
\tableofcontents

\section*{Introduction}
\addcontentsline{toc}{section}{Introduction}
Eisenstein integrals play a fundamental role in harmonic analysis on reductive symmetric spaces
of the form $X = G/H;$
here $G$ is assumed to be a real reductive group of the Harish-Chandra class, and
$H$ an (essentially connected) open subgroup of the group $G^\gs$ of an involution $\gs$
of $G.$
The notion of Eisenstein integral goes back to Harish-Chandra, who used it
to describe the contribution of generalized principal series to the Plancherel decomposition
of a real reductive group $\bp G.$ In this setting an Eisenstein integral is essentially
a matrix coefficient of an induced representation of the form $\Ind_{\bp P}^{\bp G} (\bp \omega),$
with $\bp P$ a proper parabolic subgroup of $\bp G$ and $\bp \omega$ a suitable representation
of $\bp P.$

For general symmetric spaces $G/H,$ the notion of Eisenstein integral was introduced
in \cite{Banps2} for minimal $\gs$-parabolic subgroups of $G,$ i.e., minimal parabolic subgroups
of $G$ with  the property that $\gs(P) = \bar P.$ The notion was later generalized to arbitrary $\gs$-parabolic subgroups in \cite{BrylDel}, \cite{DelEis} and found application in the Plancherel theorem for $G/H,$ see
\cite{DelPlanch} and \cite{BSPlanch}.
In this setting of reductive symmetric spaces, the Eisenstein integrals appear essentially as matrix coefficients of $K$-finite matrix coefficients with
$H$-fixed distribution vectors.

A group $\bp G$ of the Harish-Chandra class may be viewed as a homogeneous space for the left times right action of $G = \bp G \times \bp G$ on $\bp G,$ and  is thus realized as the symmetric space $G/H$ with $H$ the diagonal in $G.$ The definition of Eisenstein integral for the symmetric space $G/H$ yields a matrix coefficient on $\bp G$ which is closely related to Harish-Chandra's Eisenstein integral, but not equal to it.
The two obtained types of Eisenstein integrals differ by a normalization which can be described in terms of intertwining operators, see \bib{BSmulti} for details.
In the present paper we develop a notion of minimal Eisenstein integrals for reductive symmetric spaces, which cover both the existing notion  for symmetric spaces and Harish-Chandra's notion for the group.

An even stronger motivation for the present article lies in the application of its results to a theory of cusp forms for symmetric spaces, initiated by M.~Flensted-Jensen. In   \cite{BK} we use our results on Eisenstein integrals  to generalize the results of \cite{FJcusp1} and 
\cite{FJcusp2} to reductive symmetric spaces of $\gs$-split rank one (i.e., $\dim \faq =1$).

We will now explain our results in more detail.
 Let $\Cartan$ be a Cartan
 involution of $G$ commuting with $\gs$ and let $K$ be the associated
 maximal compact subgroup of $G.$ Let
$$
\fg = \fk \oplus \fp = \fh \oplus \fq
$$
be the eigenspace decompositions into the $\pm 1$-eigenspaces for the infinitesimal involutions $\Cartan$ and $\gs,$ respectively. Furthermore,
let $\faq$ be a maximal abelian subspace of $\fp \cap\fq$ and $\fa$ a maximal abelian subspace of $\fp$ containing $\faq.$ We put $\Aq:= \exp \faq$ and $A: = \exp \fa.$

For the description of the minimal
$\gs$-principal series one needs the (finite) set of minimal $\gs$-parabolic subgroups of $G$ containing $\Aq;$ this set is denoted by $\parabsgsAq.$ In the case of the group $\bp G$ one may take $A = \bp A \times \bp A,$ with $\bp \fa$ maximal
abelian in  $\bp \fp.$ Then $\cP_\gs(A)$ consists of all parabolic subgroups of the form $\bp  P \times \bp \bar  P,$ with $\bp P$ a minimal parabolic subgroup from $\bp G$ containing $\bp A.$ To obtain Harish-Chandra's Eisenstein itegral one would need to also consider minimal parabolic
subgroups of the form $\bp P \times \bp P.$

Our goal is then to define Eisenstein integrals by means of suitable $H$-fixed distribution vectors for all minimal parabolics of $G$ containing $A.$ The (finite) set of these is denoted by $\cP(A).$
For the case of the group one has $\cP_\gs(\Aq) \subsetneq \cP(A),$ but
for general symmetric spaces $G/H,$ the parabolic subgroups from $\parabsgsAq$ will in general not be minimal.

A parabolic subgroup $P \in \cP(A)$ is called $\fq$-extreme if it
is contained in a parabolic subgroup  $P_0$ from $\parabsgsAq,$ see Section
\ref{s: notation and preliminaries} for details. For such a parabolic, each representation $\Ind_{P_0}^G(\xi\otimes \gl\otimes 1)$ of the $\gs$-principal series can be embedded in the representation
$\Ind_P^G(\xi_M\otimes (\gl - \rhoPh)\otimes 1)$ of the minimal principal series, through induction  in  stages.
Here $\xi$ is a finite dimensional unitary representation of the Langlands component $M_0:=M_{P_0},$ and $\xi_M$ denotes the restriction of $\xi$ to $M := M_P.$
Furthermore, $\gl \in \faqdc$ and $\rhoPh := \rho_P - \rho_{P_0}.$
This is discussed in Section \ref{s: comparison of principal series representations}.

In Section \ref{Section H-fixed distribution vectors, q-extreme case}
the $H$-fixed generalized vectors of the first of these induced representations  are shown to allow a natural
realization in the latter. To describe it, one needs to parametrize the open $H$-orbits on $G/P_0.$
We will avoid this complication in the introduction,
and work under the simplifying assumption that $HP_0$ is the single open orbit.
This condition is always fulfilled in the case of the group; in general the open orbits are given by $PvH,$ for $v$ in a finite set $\cW \simeq W(\faq)/W_{K\cap H}(\faq)).$

Let $C^{-\infty}(P_0:\xi:\gl )$ denote the space of generalized vectors for the induced representations
$\Ind_{P_0}^G(\xi \otimes \gl \otimes 1).$ The $H$-fixed elements in this space needed
for the definition of the Eisenstein integral are parametrized
by $V(\xi) = \cH_\xi^{M_0\cap H}.$ Given $\eta \in V(\xi),$ one has a family
$$
j(P_0: \xi: \gl : \eta) \in C^{-\infty}(P_0: \xi:\gl)^H,    \qquad (\gl \in \faqdc),
$$
defined in \cite{Banps1}. In a suitable sense  it depends meromorphically on $\gl \in \faqdc.$ This family has image
$j_H(P:\xi_M:\gl:\eta)$ in $C^{-\infty}(P:\xi_M: \gl - \rhoPh)^H.$ By definition
the latter defines a continuous conjugate linear functional on the space
$C^\infty(P:\xi_M: -\bar \gl + \rhoPh).$ 
In (\ref{e: defi J H P}) we show
that for $\gl $ in a suitable region $\Omega_P \subseteq \faqdc$
this functional is  given by an absolutely convergent integral
\[
\tag{{\rm I}.1}
\label{e: defi j P by averaging}
\inp{j_H(P:\xi_M:\gl:\eta)}{f}  = \int_{H_P \bs H} \bar f_{\eta,\omega},
\]
for  $f \in C^\infty(P: \xi_M: -\bar \gl+\rhoPh).$ Here  $H_P:= H \cap P,$ and  $\bar f_{\eta,\omega}$
is a natural interpretation of the function $\inp{\eta}{f}|_H \in C^\infty(H)$ as a density on the quotient manifold $H_P \bs H.$

To extend formula  (\ref{e: defi j P by averaging})
to the setting of a parabolic subgroup $Q \in \cP(A)$ which is not $\fq$-extreme,
two problems need to be solved. First of all a suitable domain $\Omega_Q$ for the convergence needs
to be determined. Next, the resulting family $j_H(Q:\xi_M: \gl)$ needs to be
extended meromorphically in the parameter $\gl \in \faqdc.$

In the present paper both these problems are solved
by using a suitable partial ordering $\succeq$ on $\cP(A)$ whose
maximal elements are the $\fq$-extremal parabolic subgroups, see
 Section \ref{s: minimal parabolics} for details.
 Let $P \in \cP_\gs(A)$ be such that $P \succeq Q.$
Then the  definition of the ordering guarantees that $ H_P \subseteq H_Q$ and that
the fiber $H_P\bs H_Q$ of the natural fiber bundle
$H_P\bs H \to H_Q \bs H$ is diffeomorphic to $N_Q \cap \bar N_P$ in a natural way, see Section \ref{s: an important fibration}.
We use a general Fubini type theorem for densities on fiber bundles, discussed in the appendix of this paper, to decompose the integral
(\ref{e: defi j P by averaging}) in terms of a fiber integral over $N_Q \cap \bar N_P$ followed by an integral over the base manifold $H_Q\bs H,$ see Thm.
\ref{t: fiber integral J}.
The first of these integrals turns out to be the integral
for the standard intertwining operator
$$
A(Q:P: \xi_M: -\bar \gl +\rhoPh): C^\infty(P: \xi_M: -\bar \gl + \rhoPh ) \to C^\infty(Q: \xi_M: -\bar \gl + \rhoPh ),
$$
whereas the second integral turns out to be (\ref{e: defi j P by averaging}) with $Q$ in place of $P.$ According to Theorem \ref{t: convergence j P and intertwiner} this results in
the formula
\[
\label{e: j P j Q and intertwiner}
\tag{{\rm I}.2}
j_H(P:\xi_M:\gl : \eta) = j_H(Q: \xi_M: \gl : \eta) \after A(Q:P: \xi_M: - \bar  \gl + \rhoPh),
\]
with convergent integrals for $\gl \in \Omega_P.$ Convergence of the integral for
$j_H(Q: \xi_M: \gl : \eta) $ is thus obtained through Fubini's theorem, as a consequence
of the known convergence of the other two integrals. Furthermore, since the appearing standard intertwining operator has an inverse which is meromorphic in $\gl,$ formula (\ref{e: defi j P by averaging}) also allows us to conclude
in Theorem \ref{t: meromorphy of j H Q}
that
\[
\tag{{\rm I}.3}
\label{e: j Q}
 \gl \mapsto j_H(Q:\xi_M:\gl:\eta)
\]
depends meromorphically on $\gl \in \faqdc.$

Once the meromorphic extension of $(\ref{e: j Q})$ has been established for general $Q \in\cP(A)$
we apply a recent convexity theorem \cite[Thm.~10.1]{BB2014} to determine a large domain on which (\ref{e: j Q})
is holomorphic, see Corollary \ref{c: holomorphy of j H Q}).
The convexity theorem describes the image
of $H$ under the  projection $\fH_{Q,\iq}: G \to \faq$ determined by the Iwasawa decomposition
$G = K (A \cap H) \exp(\faq )  N_Q$ as a convex polyhedral cone described in terms of a subset of 
 the set of roots of $\fa$ in $\fn_Q.$  This description allows one to decide whether
this cone properly contains the origin or equals it. In the latter case it is shown that (\ref{e: j Q}) is holomorphic on all of $\faqdc,$ see Remark
\ref{r: holomorphy j in fh extreme case}.

The definition of the meromorphic family of $H$-fixed generalized vectors (\ref{e: j Q}) allows us
to define Eisenstein integrals $E(Q: \gl)$ essentially as matrix coefficients with
$K$-finite vectors in the induced representation under consideration.
 In particular, the Eisenstein integral depends meromorphically on $\gl.$ Holomorphy
of (\ref{e: j Q}) implies holomorphy of the corresponding Eisenstein integral,
see Corollary \ref{c: holomorphy Eisenstein integral}.

The relation (\ref{e: j P j Q and intertwiner}) leads to a relation between the Eisenstein
integral $E(Q: \gl)$ and the Eisenstein integral $E(P_0,\gl),$ earlier defined in
\cite{Banps2} and \cite{BSft}. This relation amounts to a different normalization of the Eisenstein integral
expressed in terms of a $C$-function, see Corollary \ref{c: functional c functions}.

Finally, in Section \ref{s: the case of the group}, we discuss the case of the group, and express the obtained Eisenstein integrals
in terms of Harish-Chandra's Eisenstein integrals, see Corollary \ref{c: general cor Eisenstein integral in group case}.
 In this case, the Eisenstein integral
$E(Q:\gl)$ coincides with Harish-Chandra's if and only if $Q$ is a  $\succeq$-minimal element
of $\parabsA.$ The latter means that $Q$ is of the form $\bp Q \times \bp Q,$ with $\bp Q \in \cP(\bp A);$ see Corollary \ref{c: char of HC Eisenstein integral}. The result on holomorphy established above, is consistent with the holomorphic dependence of Harish-Chandra's Eisenstein integral, see
Remark \ref{r: holomorphy HC Eisenstein}.
\medbreak
{\bf Acknowledgements\ } We would like to thank Henrik Schlichtkrull
and Mogens Flensted-Jensen for many fruitful discussions. Part of this research was made possible by the  Max Planck Institute for Mathematics in Bonn, 
where both authors participated in the special activity `Harmonic Analysis on Lie groups', 
 organized by Bernhard Kr\"otz and Henrik Schlichtkrull, in the late summer of 2011.
\section{Notation and preliminaries}
\label{s: notation and preliminaries}
In this section we collect some of the notation that will be used
 throughout this article.

We assume that $G$ is a reductive Lie group of the Harish-Chandra class.
Let $\sigma$ be an involutive automorphism of $G$ and   let $\theta$ be a Cartan involution
that commutes with $\sigma$, let $K:= G^\theta$ be the associated maximal compact subgroup.
Let $H$ be an open subgroup of the fixed point subgroup $G^{\sigma}$.
We assume $H$ to be essentially connected. (See \cite[p.24]{Bconv}.)
If $S$ is any closed subgroup of $G,$ we agree to write
 \begin{equation}
\label{e: S cap H}
H_S: = S \cap H.
\end{equation}
A  Lie group  will in general be denoted  by a Roman upper case letter; the associated
 Lie algebra
 by the corresponding lower case gothic letter.
  We denote the infinitesimal involutions associated with 
$\sigma$ and $\Cartan$ by the same symbols, respectively.  As usual, the decompositions of 
$\fg$ into the $+1$ and $-1$ eigenspaces for $\Cartan$ and $\gs$ are denoted by 
$\fg = \fk \oplus \fp = \fh \oplus \fq$ respectively.  
As the involutions $\sigma$ and $\theta$ commute, we have the following decomposition of 
$\fg$ also decomposes as a direct sum of vector spaces
$$
\fg
=(\fk\cap\fh)\oplus(\fk\cap\fq)\oplus(\fp\cap\fh)\oplus(\fp\cap\fq).
$$
\noindent
 We fix a non-degenerate $G$-invariant bilinear form $B$ on $\fg,$
which coincides with the Killing form on $[\fg, \fg],$ is negative definite on $\fk$ and
positive definite on $\fp,$ and for which the above decomposition is orthogonal.
Furthermore, we equip $\fg$ with the positive definite inner product
given by
$$
\inp{\dotvar}{\dotvar} : = - B(\dotvar , \Cartan(\dotvar)).
$$

We fix a maximal abelian subspace $\faq$ of $\fp\cap\fq$ and a maximal abelian subspace $\fa$ of $\fp$ containing $\faq$.
Then $\fa$ decomposes as
$$
\fa=\fah\oplus\faq,
$$
where $\fah=\fa\cap\fh$.   
This decomposition induces natural embeddings of the associated dual 
spaces  $\fahd$ and $\faqd$ into $\fad.$ 
Let $A:=\exp(\fa)$, $\Aq:=\exp(\faq)$ and $\Ah:=\exp(\fah)$.

If $P$ is a parabolic subgroup (not necessarily minimal), then we write $N_P$ for the unipotent radical of $P$.
If $P$ contains $A$ and $\fb$ is a subalgebra of $\fa$, then we write $\gS(P,\fb)$ for the set of weights of $\fb$ in $\fn_{P}$.
Furthermore,  we  write  $\gS(P)$ for $\gS(P,\fa)$, unless
clarity of exposition requires otherwise.
If $\tau$ is an involution of $\fg$ preserving $\fa,$ we agree to
write
\begin{equation}
\label{e: defi gS P tau}
\gS(P,\tau):= \gS(P) \cap \tau\gS(P).
\end{equation}
For a root $\ga \in \gS(\fa) \cap \faqd,$ we note that $\gs\Cartan \ga = \ga,$ so that
$\gs\Cartan$ leaves the root space $\fg_\ga$ invariant.  Accordingly, 
we define the
subset  $\gS(P)_{ -} = \gS(P)_{\gs, -}$  of $\gS(P,\gs\Cartan)$ by
\begin{equation}\label{e: def minus roots}
\gS(P)_{-}
:=\{\ga\in\gS(P,\sigma\theta):\ga\in\faqd\Rightarrow\sigma\theta|_{\fg_{\ga}}\neq I\}.
\end{equation}

Let $M$ denote the centralizer of $A$ in $K$ and let $\parabsA$ denote the set of minimal parabolic subgroups $P \subseteq G$ with $A \subseteq P.$
Then each subgroup $P \in \parabsA$ has a Langlands decomposition
of the form $P = MA N_P$.

\begin{defi}
\label{d: q extreme parabs}
A parabolic subgroup
$P\in\cP(A)$ is said to be
 ${\fq}$-extreme if
$$
\gS(P,\sigma\Cartan) = \gS(P) \setminus \fahd.
$$
The set of these parabolic subgroups is denoted by $\parabsgsA.$
\end{defi}
We will finish this section by comparing $\cP_\gs(A)$ with the set
 $\parabsgsAq$ of minimal $\gs\Cartan$-stable parabolic subgroups of $G$ containing $\Aq.$
We recall from \cite{Banps1} that the latter set is finite
and in bijective correspondence with the set of positive systems
for $\gS(\fg, \faq).$  Indeed, if $\Pi$ is such a positive system
then the corresponding parabolic subgroup $P_\sPi$
from $\parabsgsAq$ equals 
$
P_\sPi = Z_G(\faq) N_\sPi,
$
where
$
\fn_\sPi := \oplus_{\ga \in \sPi} \fg_\ga$  
and 
$N_\sPi:= \exp(\fn_\sPi).
$
The Langlands decomposition of $P_\sPi$ is given by
$$
P_\sPi = M_0 A_0 N_\sPi,
$$
where 
$A_0 = \exp (\fa_0)$ and $M_0 = Z_K(\faq) \exp(\fm_{0}), $ with
$
\fa_0 = \cap_{\ga \in \gS(\fg,\fa) \cap \fahd} \ker \ga, $ and $
\fm_{0}:=Z_{\fg}(\faq)\cap \fa_{0}^{\perp}.
$

Conversely, if $P_0 \in \parabsgsAq$ then the associated positive
system is given by
$$
\gS(P_0, \faq) := \{\ga\in \gS(\fg,\faq)\mid \fg_\ga \subseteq \fn_{P_0}\}.
$$

\begin{lemma}
\label{l: q extreme parabolic}
Let $P \in \parabsA.$ Then the following conditions are
equivalent.
\begin{enumerate}
\itema $P \in \cP_\gs(A);$\vspace{-5pt}
\itemb there exists a $P_0\in \parabsgsAq$ such that $P \subseteq P_0.$
\end{enumerate}
\end{lemma}

\begin{proof}
First assume (a). Then
$\gS(P) \setminus \fahd = \gS(P, \gs \Cartan)$
and we see that the set $\Pi$ of non-zero restrictions $\ga|_{\faq}$,
for $\ga \in \gS(P)\setminus \fahd,$ is a positive system
for $\gS(\fg, \faq).$ Now $N_P = (N_P \cap M_0)N_\sPi$
and we see that $P  \subseteq P_\sPi$ and (b) follows.

Next assume (b).
We first note that
$$
\gS(P_{0},\faq)
=\{\alpha|_{\faq}:\alpha\in\gS(P_{0}),\alpha|_{\faq}\neq0\}.
$$
The minimality of $P_{0}$ implies that $\gS(P_{0},\faq)$ is a positive system  for the root system $\gS(\fg,\faq)$, hence
$$
\gS(\fg,\fa)\setminus\fahd
=\{\alpha\in\gS(\fg,\fa):\alpha|_{\faq}\in\gS(\fg,\faq)\}
=\gS(P_{0})\cup -\gS(P_{0}).
$$

By assumption $P\subseteq P_{0}.$ This implies  $\gS(P_{0})\subseteq\gS(P)$ and
$\gS(P)\cap-\gS(P_{0})=\emptyset$. Hence,
$$ \gS(P)\setminus \fahd = \gS(P_0).$$
Moreover, since $P_{0}$ is $\sigma\theta$-stable the above equality implies $\gS(P) \setminus \fahd \subseteq \gS(P,\sigma\theta).$ As the converse inclusion is obvious, the parabolic $P$ is $\fq$-extreme and (a) follows.
\end{proof}

\section{Minimal parabolic subgroups}
\label{s: minimal parabolics}
\begin{lemma}
Let $P \in \parabsA.$
The set $\gS(P)$ is the disjoint union of $\,\gS(P, \gs)$
and $\gS(P, \gs\Cartan).$
\end{lemma}
\begin{proof}
Let $\ga \in \gS(P).$ Then either $\gs\ga \in \gS(P)$ or
$\gs\Cartan \ga = -\gs \ga \in \gS(P).$ The two cases are exclusive,
and in the first case we have $\ga \in \gS(P, \gs),$ while in the second $\ga \in \gS(P, \gs \Cartan).$
\end{proof}
We define the partial ordering
$\succeq$ on   $\parabsA$ by
\begin{equation}
\label{e: defi ordering on minparabs}
P \succeq Q \iff\;\;
\gS(Q, \gs \Cartan) \subseteq \gS(P, \gs\Cartan)\;\;{\rm and}\;\;
\gS(P, \gs) \subseteq \gS(Q, \gs).
\end{equation}
It is easy to see that this condition on  $P$ and $Q$ implies that $H_{N_P} \subseteq H_{N_Q}.$
The latter condition implies that we have a natural
surjective $H$-map $H/H_{N_P} \to H/H_{N_Q}.$

\begin{lemma}
\label{l: intersection with faqd}
Let $P, Q\in \parabsA,$ and assume that $P \succeq Q.$ Then
\begin{enumerate}
\itema
$\gS(P) \cap \faqd = \gS(Q) \cap \faqd;$
\itemb $\gS(P) \cap \fahd = \gS(Q) \cap \fahd.$
\end{enumerate}
\end{lemma}

\begin{proof}
Let $\ga \in \gS(Q) \cap \faqd.$ Then $\gs\Cartan \ga = \ga$ so that
$\ga \in \gS(Q, \gs\Cartan) \subseteq \gS(P, \gs\Cartan).$ We infer that
$\gS(Q) \cap \faqd \subseteq \gS(P) \cap \faqd.$
Since both sets in this inclusion
are positive systems for the root system $\gS\cap \faqd,$ the converse inclusion
follows by a counting argument.

Assertion (b) is proved in a similar fashion, using $\gs$ in place of $\gs\Cartan$
and referring to the second inclusion of (\ref{e: defi ordering on minparabs}) instead of the first.
\end{proof}
\break
\begin{lemma}
\label{l: characterization ordering minparabs}
Let $P, Q \in \parabsA.$ Then the following statements are equivalent.
\begin{enumerate}
\itema $P \succeq Q;$
\itemb $\gS(P) \cap \gS(\bar Q) \subseteq \gS(P, \gs \Cartan)$ and{\ } $\gS(\bar P) \cap \gS(Q) \subseteq \gS(Q, \gs);$
\itemc $\gS(P) \cap \gS(\bar Q) = \gS(P,\gs \Cartan) \cap \gS(\bar Q, \gs).$
\end{enumerate}
\end{lemma}

\begin{proof}
First assume (a). Let $\ga \in \gS(P) \cap \gS(\bar Q).$ Then $\gs \ga \in \gS(P)$ would lead
to $\ga \in \gS(Q),$ contradiction. Hence, $\ga \in \gS(P, \gs\Cartan).$
The second inclusion of (b) follows in a similar fashion.

Next, (b) is equivalent to $\gS(P) \cap \gS(\bar Q) \subseteq \gS(P, \gs \Cartan) \cap \gS(\bar Q, \gs),$
which is readily seen to be equivalent to (c).

Finally, assume (c) and let $\ga \in \gS(P, \gs).$ Then $\ga \in \gS(P) \setminus \gS(P, \gs \Cartan),$
hence $\ga \notin \gS(\bar Q)$ by the equality of (c) and it
follows that $\ga \in \gS(Q).$ Likewise, $\gs\ga \in \gS(Q)$
and we conclude that $\ga \in \gS(Q, \gs).$
On the other hand, let $\ga \in \gS(Q, \gs \Cartan).$ Then $\ga \in \gS(Q)\setminus \gS(Q, \gs)$.
The equality in (c) is equivalent to
$$
\gS(\bar P) \cap \gS(Q)=\gS(\bar P,\gs \Cartan) \cap \gS( Q, \gs),
$$
which shows  that $\ga \in  \gS(P).$ Likewise, $\gs \Cartan \ga \in \gS(P)$
and we see that $\ga \in \gS(P, \gs \Cartan).$ This proves (a).
\end{proof}
\begin{lemma}
\label{l: comparison orderings parabs}
Let $P,Q, R \in \parabs(A)$ be such that $P \succeq R.$ Then the following assertions are equivalent:
\begin{enumerate}
\itema
 $P\succeq Q \succeq R;$
\itemb
$\gS(P) \cap \gS(\bar Q) \subseteq \gS(P) \cap \gS(\bar R).$
\end{enumerate}
\end{lemma}

\begin{proof}
Assume (a).
By Lemma  \ref{l: characterization ordering minparabs}, the first set in (b) equals $\gS(P, \gs \Cartan ) \cap \gS(\bar Q, \gs),$ which by
(\ref{e: defi ordering on minparabs}) is contained in $\gS(P, \gs\Cartan) \cap \gS(\bar R, \gs).$ The latter set equals
the second set of (b), again by application of Lemma  \ref{l: characterization ordering minparabs}.
Assertion (b) follows.

For the converse implication, assume (b). Then it is well known and easy to show that
\begin{equation}
\label{e: disjoint union root sets}
\gS(P) \cap \gS(\bar R) = (\gS(P) \cap \gS(\bar Q)) \cup (\gS(Q) \cap \gS(\bar R)) \qquad \mbox{\rm (disjoint union)}
\end{equation}
Indeed, it is obvious that the set on the left-hand side of (\ref{e: disjoint union root sets})
is contained in the union on the right-side. For the converse inclusion, we first note that
(b) implies  $\gS(\bar P) \cap \gS(Q) \subseteq \gS(R).$ Now assume that
$\ga \in \gS(Q) \cap \gS(\bar R).$ Then $\ga\notin \gS(\bar P)$ so that  $\ga \in \gS(P) \cap \gS(\bar R).$
Hence, the second inclusion of (\ref{e: disjoint union root sets}) follows as well.

Still assuming (b), we claim that $P \succeq Q.$
To see this, let $\ga \in \gS(P) \cap \gS(\bar Q).$ Then
$ \ga \in  \gS(P) \cap \gS(\bar R) \subseteq \gS(P, \gs\Cartan)\cap \gS(\bar R, \gs)$
by Lemma  \ref{l: characterization ordering minparabs}.
Assume now in addition $\ga \notin \gS(\bar Q,\gs).$ Then $\gs\Cartan \ga \in \gS(P) \cap \gS (\bar Q)\subseteq \gS(P) \cap \gS(\bar R),$ hence
$\ga \in \gS(\bar R, \gs \Cartan),$ contradicting the earlier conclusion that $\ga \in \gS(\bar R, \gs).$
Thus the assumption cannot hold, so that $\ga \in \gS(P, \gs \Cartan) \cap \gS(\bar Q, \gs).$ In view of Lemma  \ref{l: characterization ordering minparabs}
this establishes the claim.

We will now infer (a) by establishing that $Q \succeq R.$  For this, let $\ga \in \gS(Q)\cap \gS(\bar R).$
Then $\ga \in \gS(P) \cap \gS(\bar R)$ by (\ref{e: disjoint union root sets}), which implies that
$ \ga \in \gS(P, \gs\Cartan) \cap \gS(\bar R, \gs)$ by Lemma \ref{l: characterization ordering minparabs}.
Assume now that $\ga \notin \gS(Q, \gs\Cartan).$ Then
$\gs \ga \in \gS(Q) \cap  \gS(\bar R ) \subseteq \gS(P),$ so that $\ga \in \gS(P, \gs),$ contradicting the earlier conclusion that $\ga \in \gS(P, \gs \Cartan).$ Thus, the assumption cannot hold, so that $\ga \in \gS(Q, \gs \Cartan) \cap \gS(\bar R, \gs).$ Applying Lemma \ref{l: characterization ordering minparabs} with $(Q,R)$ in place of $(P,Q),$ we finally obtain that $Q \succeq R.$
\end{proof}

\begin{rem}
Recall that two parabolic subgroups $P, Q \in \parabs(A)$ are said to be
adjacent if $\gS(P) \cap \gS(\bar Q)$ has a one dimensional span in $\fad.$

If $P, Q \in \parabs(A)$ then there exists a sequence
$P = P_0, P_1, \ldots, P_n = Q$ of parabolic subgroups in $\cP(A)$ such that for all $0\leq j < n$
we have $\gS(P) \cap \gS(\bar P_{j}) \subseteq \gS(P) \cap \gS(\bar P_{j+1})$
and such that $P_j$ and $P_{j+1}$ are adjacent.  If in addition $P \succ Q,$ then it follows
from repeated application of the lemma above that
$$
P =P_0 \succ P_1 \succ \cdots \succ P_n =Q.
$$
\end{rem}

Our  next objective in this section is to show that every parabolic subgroup from $\parabs(A)$ is dominated by a $\fq$-extreme one, see Definition \ref{d: q extreme parabs}

Given $Q \in \cP(A),$ we denote the positive Weyl chamber
for $\gS(Q)$ in $\fa$ by $\fa^+(Q).$ Furthermore, we put
\begin{equation}
\label{e: defi faqpQ}
\faqp(Q) = \{H \in \faq\mid \ga(H) > 0,\;\; \forall \ga \in \gS(Q, \gs \Cartan)\}.
\end{equation}
It is readily verified that this set contains the image of $\fap(Q)$ under
the projection $\pr_\iq: \fa \to \faq$; in particular, it is non-empty.

Let $\faq^\reg$ be the set of regular elements in $\faq,$ relative
to the root system $\gS(\faq).$ The connected components
of this set are the chambers for the root system $\gS(\faq).$ The collection
of these is denoted by $\Pi_0(\faqreg).$ It is clear
that $\faq^\reg \cap \faqp(Q)$ is the disjoint union of the chambers
contained in $\faqp(Q).$

We define
$$
\parabsaboveQ: = \{ P \in \cP_\gs(A)  \mid \; P\succeq Q\;\}
$$

\begin{lemma}
\label{l: on cP gs A Q}
Let $Q \in \cP(A).$ Then the assignment  $P \mapsto \faqp(P)$ defines a bijection
from the set $\parabsaboveQ$ onto the set $\{C \in \Pi_0(\faq^\reg)\mid C \subseteq \faqp(Q)\}.$
\end{lemma}

\begin{proof}
We abbreviate  $\cC(Q):= \{C \in \Pi_0(\faq^\reg)\mid C \subseteq \faqp(Q)\}.$
Let $P\in \parabsaboveQ.$ Then a root $\ga \in \gS(P)$ restricts to a non-zero
root on $\faq$ if and only if
$
\ga \in \gS(P)\setminus \fahd.$ The latter set equals $\gS(P) \setminus  \gS(P,\gs) = \gS(P, \gs \Cartan).$
Therefore,  $\faqp(P)$ is a connected component of $\faq^\reg.$ Furthermore,
from $P \succeq Q$ it follows that $\gS(P, \gs \Cartan) \supset \gS(Q, \gs \Cartan),$
which in turn implies that $\faqp(P)\subseteq \faqp(Q).$ It follows that 
$\faqp(P) \in \cC(Q).$ 
It remains to be shown that the map
\begin{equation}
P \mapsto \faqp(P), \; \cP_\gs(A, Q) \to \cC(Q)
\end{equation}
is bijective.
For injectivity, assume that $P_1,P_2 \in \cP_\gs(A,Q)$ and that $\faqp(P_1) = \faqp(P_2).$
Let $\ga \in \gS(P_1).$ If $\ga \in \fahd,$ then $\ga \in \gS(Q) \cap \fahd  \subseteq \gS(P_2).$
If $\ga \notin \fahd,$ then $\ga \in \faqp(P_1, \gs \Cartan)$ and it follows that
$\ga > 0$ on $\faqp(P_1)  = \faqp(P_2),$ which implies that $\ga \in \gS(P_2).$
Thus, we see that $\gS(P_1) \subseteq \gS(P_2)$ which implies $P_1 = P_2.$

For surjectivity,  let $C$ be a chamber in $\cC(Q).$ Let $\Pi_C$ denote
the set of roots $\ga \in \gS(\fa)$ that are strictly positive on $C.$
The set $ \Pi_{\ih} := \gS(Q) \cap \fahd$
is a choice of positive roots for the root system $\gS(\fa) \cap \fahd.$ Hence, there exists
an element $Y \in \fahd$ such that
$$
\Pi_{\ih} = \{\ga \in \gS(\fa) \cap \fahd \mid \ga (Y) > 0\}.
$$
Fix $X \in C$ and put $X_t = X + tY$ for $t \in \R.$ Then there exists $\geps >0$ such that
for $|t| < \geps$ we have  $\ga(X_t) > 0$ for all $\ga \in \Pi_C.$ Fix $0 < t < \geps.$ Then
it follows that $X_t$ is regular for $\gS(\fa)$ and that the associated choice of positive roots
$\Pi := \{\ga \in \gS(\fa) \mid \ga(X_t) > 0\}$ is the disjoint union of $\Pi_C$ and $\Pi_\ih.$
Let $P$ be the parabolic subgroup in $\cP(A)$ with $\gS(P) = \Pi.$
Then $\gS(P) \cap \fahd = \Pi_\ih = \gS(Q)\cap \fahd.$ Furthermore, if $\ga \in \gS(P)\setminus \fahd,$
then 
$\ga \in \Pi_C.$ 
Hence, $\gs \Cartan \ga (X_t) = \ga ( - \gs(X_t)) = \ga(X_{-t}) > 0,$ and we see that
$\ga \in \gS(P, \gs \Cartan).$ It readily follows that $P \in \cP_\gs(A,Q).$
\end{proof}

We finish this section by investigating these structures in the setting
where $H$ is replaced by a conjugate $v H v^{-1},$ with $v \in N_K(\fa) \cap N_{\faq}.$
Let such an element $v$ be fixed. Then $v$ normalizes $\fah$ as well.
 Let $C_v: G \to G$ denote conjugation by $v,$
and put
\begin{equation}
\label{e: defi gs v}
\gs_v: = C_v \after \gs \after C_v^{-1}.
  \end{equation}
  Then $\gs_v$ is an involution of $G$ which commutes with the Cartan involution $\Cartan;$ moreover, since $v$ normalizes $Z_K(\faq),$ the conjugate group $v H v^{-1}$ is readily seen to be an essentially 
  connected open 
  subgroup of $G^{\gs_v}.$
The infinitesimal involution associated with $\gs_v$ is given by
 $
 \gs_v = \Ad(v) \after \gs \after \Ad(v)^{-1}.
 $
 Since $\Ad(v)$ normalizes $\faq$ and $\fah,$ it follows that
 \begin{equation}
 \label{e: restriction gs v to fa}
 \gs_v|_{\fa} = \gs|_\fa
 \end{equation}
 and that $\faq$ is maximal abelian in $\fp \cap \ker ( \gs_v + I). $

 It follows from (\ref{e: defi gS P tau}) and
(\ref{e: restriction gs v to fa}) that
\begin{equation}
\label{e: root sets for v}
\gS(Q, \gs_v) =  \gS(Q , \gs) \quad {\rm and} \quad
\gS(Q, \gs_v \Cartan) =  \gS(Q , \gs\Cartan).
\end{equation}
From this we see that the ordering on $\cP(A)$ defined by (\ref{e: defi ordering on minparabs}) with $\gs_v$ in place of  $\gs$ coincides with the ordering $\succeq.$  It is also clear that   $P \mapsto v^{-1} P v$ preserves $\cP_{\gs}(A).$

\begin{lemma}
\label{l: twisted sigma minus}
Let $Q \in \parabs(A)$ and $v \in N_K(\fa) \cap N_K(\faq).$ Then
\begin{equation}
\label{e: gS Q minus twisted}
\gS(Q)_{\gs_v, -} = v \gS(v^{-1}Q v)_{-}
\end{equation}
\end{lemma}

Let
$
S_v: = \{\ga \in \gS(\fa) \cap \faqd \mid \gs_v \Cartan |_{\fg_\ga} \neq I\}.
$
Then it is readily seen that $S_v = v S_e.$
From (\ref{e: def minus roots}) we now deduce that
$$
\gS(Q)_{\gs_v, -} \cap \faqd =  \gS(Q)\cap  vS_e
= v (\gS(v^{-1} Q v ) \cap S_e) =  v \gS(v^{-1}Q v) _- \cap \faqd.
$$
 
On the other hand,
\begin{eqnarray*}
\gS(Q)_{\gs_v,-}\setminus \faqd &=&
\gS(Q, \gs_v \Cartan) \setminus \faqd
\; \; = \; \;\gS(Q, \gs \Cartan)\setminus \faqd\\
&=& v(\gS(v^{-1} Q v , \gs \Cartan) \setminus \faqd)
\;\; = \;\; 
v \gS(v^{-1}Qv)_- \setminus \faqd
\end{eqnarray*}
and we deduce  (\ref{e: gS Q minus twisted}).
\section{Induced representations and densities}
\label{s: induced reps and densities}
Let $P = M_P A_P N_P$ be  a parabolic subgroup with the indicated Langlands decomposition and let $(\xi, \cH_\xi)$ be a unitary representation in a finite dimensional Hilbert space $\cH_\xi.$ 
The assumption of finite dimensionality is natural for the purpose of this paper. 
Moreover, the following 
definitions, though valid in general, will merely be needed 
for the case that $P$ belongs to either  $\cP(A)$ or $\cP_\gs(\Aq).$ 

For $\mu \in \faPdc$ and  $s \in \N \cup \{\infty\}$ we denote by $C^s(P:\xi:\mu)$ the space of
$C^s$-functions $f: G \to \cH_\xi$ transforming according to the rule
$$
f(man x) = a^{\mu + \rho_P} \xi(m) f(x),
$$
for all $x\in G$ and $(m,a,n) \in M_P\times A_P\times N_P.$
The right regular representation $R$ of $G$ in this space is the $C^s$-version of the normalized induced representation $\Ind_P^G(\xi \otimes \mu \otimes 1).$

We put $K_{M_P}: =K \cap M_P$ and denote by
$C^s(K:\xi) := C^s(K: K_{M_P} : \xi)$ the space of $C^s$-functions $f: K \to \cH_\xi$ transforming according to the rule
$$
f(mk) = \xi(m) f(k), \qquad (k \in K, \, m\in K_{M_P}).
$$
All function spaces introduced so far are assumed to be equipped with the
usual Fr\'echet topologies (Banach when $s< \infty$). The   restriction map $f \mapsto f|_K$
gives topological linear isomorphisms
\begin{equation}
\label{e: iso compact picture}
C^{s}(P:\xi:\mu) \;\;{\buildrel \simeq \over \longrightarrow}\;\; C^{s}(K:\xi),
\end{equation}
intertwining the $K$-actions from the right. Through these, the right regular
actions of the group $G$ may be transferred to continuous representations
of $G$ on $C^s(K:\xi),$ denoted $\pi_{P, \xi, \mu}.$ This realisation $\pi_{P,\xi,\mu}$ is called the compact picture
of the $C^s$-version of the parabolically induced representation $\Ind_P^G(\xi \otimes \mu \otimes 1),$
 see \cite[p.\ 15]{KSII}. 
Let $dk$ denote the normalized Haar measure on $K,$ and let
 $\inp{\dotvar}{\dotvar}_\xi$ denote the inner product of $\cH_\xi.$
 Then it  is well known,  see e.g. \cite[Lemma 8.3.11]{WalHS}, that the sesquilinear pairing
$C(K:\xi) \times C(K:\xi) \to \C$ given by
\begin{equation}
\label{e: sesquilinear pairing on K}
\inp{f}{g}_\xi:= \int_K \inp{f(k)}{g(k)}_\xi\; dk,
\end{equation}
is equivariant for the representations
$\pi_{P,\xi, \mu}$ and $\pi_{P, \xi, - \bar \mu}.$
Accordingly, the above formula gives an equivariant sesquilinear pairing
\begin{equation}
\label{e: pairing induced reps}
C(P:\xi: \mu) \times C(P:\xi : -\bar \mu) \to \C.
\end{equation}
We will usually omit the index $\xi$ in the notation of the pairing (\ref{e: sesquilinear pairing on K}).

We denote by $C^{-s}(P:\xi:\mu)$  the continuous conjugate-linear dual of the Fr\'echet space $C^s(P:\xi: - \bar \mu),$ equipped with the strong dual topology and with the natural dual representation.
Likewise, we denote by $C^{-s}(K:\xi)$ the continuous conjugate-linear dual of $C^s(K:\xi).$
\par
By using the pairing (\ref{e: pairing induced reps}) we obtain equivariant continuous linear injections
$$
C(P:\xi: \mu)  \embeds   C^{-s}(P:\xi:\mu),
$$
for $s \in \N\cup \{\infty\}.$
Likewise, by using the pairing  (\ref{e: sesquilinear pairing on K}) we obtain $K$-equivariant continuous linear injections
$C(K:\xi) \embeds C^{-s}(K:\xi).$ Through the indicated pairings it is readily seen that the isomorphism (\ref{e: iso compact picture}) for $s = 0$ extends to a topological linear isomorphism
\begin{equation}
\label{e: iso minfty P and K}
C^{-s}(P:\xi:\mu) \;\;{\buildrel \simeq \over \longrightarrow}\;\; C^{-s}(K:\xi),
\end{equation}
for all $s \in \N\cup \{\infty\}.$
By transfer we obtain a continuous representation $\pi_{P,\xi, \mu}^{-s}$
of $G$ in the  second space in (\ref{e: iso minfty P and K}), such that the isomorphism becomes $G$-equivariant. It is readily verified that this representation is
dual to the representation $\pi_{P, \xi, -\bar \mu}$ on $C^s(K:\xi).$
We will usually omit the superscript $-s$ in the notation of this dual representation.
\par
For $s, t \in \N$ with $s < t,$ the inclusion map
 $C^t(K:\xi) \to C^s(K:\xi)$ is a compact linear map of Banach spaces which has a dense image and therefore dualizes to a compact linear injection
 $$
 C^{-s}(K:\xi) \to C^{-t}(K:\xi).
 $$
 In view of \cite[Thm.~11]{Kom},  the locally convex space $C^{-\infty}(K:\xi),$ equipped with the strong dual topology, coincides with the inductive limit
 of the Banach spaces $C^{-s}(K:\xi).$ Furthermore, by 
 \cite[Lemma 3]{Kom}
 each bounded subset of $C^{-\infty}(K:\xi)$ is a bounded subset of $C^{-s}(K:\xi)$ for some $s.$

Let now $\Omega$ be a complex manifold. Then by the above mentioned property of bounded subsets of the inductive limit, 
 a function $\gf:  \Omega \to C^{-\infty}(K:\xi)$ is holomorphic if for each $z_0 \in \Omega$ there exists an
open neighborhood $\Omega_0$ of $z_0$ in $\Omega$ and a natural number $s \in \N$ such that $\gf$ maps $\Omega_0$ holomorphically into
the Banach space $C^{-s}(K:\xi).$
A densely defined function
$f$ from $\Omega$ to $C^{-\infty}(K:\xi)$ is said to be meromorphic
if for each $z_0 \in \Omega$ there exists
an open neighborhood $\Omega_0$ and a holomorphic function $q: \Omega_0 \to \C$ such that $q f$ extends holomorphically from ${\rm Dom}(f) \cap \Omega_0$ to $\Omega_0.$

For later use, we record some observations involving the contragredient
$M_P$\--re\-present\-ation $\xi^\vee,$  whose space $\cH_{\xi^\vee}$
is the linear dual of $\cH_\xi.$ The assignment $v \mapsto \inp{v}{\dotvar}$ defines an
$M_P$-equivariant conjugate-linear isomorphism from $\cH_\xi$ onto $\cH_{\xi^\vee}.$ This isomorphism induces a
$K$-equivariant topological conjugate-linear isomorphism from
$C^\infty(K:\xi)$ onto $C^\infty(K:\xi^\vee).$ The latter isomorphism is
equivariant for the representations $\pi_{P, \xi, - \bar \mu}$ and $\pi_{P, \xi^\vee , -\mu},$ respectively, for every $\mu \in \faPdc.$ 
Through this isomorphism,
the pairing (\ref{e: sesquilinear pairing on K}) is transferred to the bilinear pairing
\begin{equation}
\label{e: continuous bilinear pairing K xi}
C^\infty(K:\xi) \times C^\infty(K:\xi^\vee) \to \C
\end{equation}
given by
\begin{equation}
\label{e: bilinear pairing on K}
\inp{f}{g} = \int_K \inp{f(k)}{g(k)}\; dk.
\end{equation}
Furthermore, this pairing is equivariant for the representations
$\pi_{P, \xi, \mu}$ and $\pi_{P, \xi^\vee, -\mu}.$ 
 Through it, we see that
$C^{-\infty}(K: \xi)$ is naturally identified with the continuous linear dual
of $C^\infty(K:\xi^\vee).$ Moreover, this identification realizes the
representation $\pi_{P, \xi, \nu\mu}^{-\infty}$ as the contragredient of $\pi_{P, \xi^\vee, - \mu}^\infty.$  Accordingly, we obtain the $G$-equivariant topological linear isomorphism
$$
C^{-\infty}(P: \xi: \mu) \simeq C^\infty(P:\xi^\vee: - \bar \mu)'.
$$
In the rest of this section we assume that $P \in \parabsA$ and that $(\xi, \cH_\xi)$ is a (not necessarily irreducible)  unitary
representation of $M$ in a finite dimensional Hilbert space $\cH_\xi.$

One of the goals of this paper is to study $H$-invariant distribution vectors of principal series representations.
A first step in the construction of these is the following.
We consider the homogeneous space $H_P\bs H,$ see (\ref{e: S cap H}) for notation, and denote 
 the associated canonical projection by $\pi: H \to \HP\bs H.$ Given $x \in H$ we
write $[x] = \pi(x).$ Furthermore, for $h \in H$ we use the following notation for the right multiplication map,
$$
r_h: \HP\bs H \to  \HP\bs H,
\;\;\; [x] \mapsto [x]h=[xh]
$$
We refer to the appendix, the text preceding (\ref{e: integral of density}), for the notion
of a density on $\HP\bs H$ and the associated notion of the density bundle $\cD_{\HP\bs H}.$
The notion of the pull-back bundle $\pi^*\cD_{\HP\bs H} \to H$ is defined in the same appendix,
in the text before (\ref{e: natural isomorphism density bundle for fibration}).

Let $\fh_P$ denote the Lie algebra of $H_P = H \cap P,$ then $d\pi(e)$ induces a linear isomorphism
$\fh/  \fh_P \simeq T_{[e]}(\HP \bs H).$
We fix a positive density $\omega=\omega_{H_P\bs H} \in \cD_{\fh/\fh_P}.$

If $S \subseteq \gS(P),$ we define the subspace $\fn_S \subseteq \fg$ to be the
direct sum of the root spaces $\fg_\ga,$ for $\ga \in S,$ and we define
$\rho_S \in \fad$ by
$$
\rho_S(X) = \frac12 \tr \left(\ad(X)|_{\fn_S}\right), \quad (X \in \fa).
$$
Furthermore, we agree to abbreviate
$$
\rhoPh: = \rho_{\gS(P) \cap \fahd}.
$$
In the following result we will describe certain densities associated with principal series
representations.
\begin{lemma}
\label{l: density associated with f}
Let $\gl \in \faqdc,$ $f \in C(P:\xi: -\bar \gl + \rhoPh)$ and $\eta \in \cH_\xi^{H_{M}}.$
Then
$$
\bar f_{\eta,\omega}: h \mapsto \inp{\eta}{f(h)}_\xi\, dr_h([e])^{-1*} \omega
$$
defines
a continuous density on the homogeneous space $H_P \bs H.$
\end{lemma}

\begin{proof}
For each $h\in H,$ put $\gf(h) = \bar f_{\eta, \omega}(h).$ Then $\gf(h)$ defines a density
on the tangent space $T_{H_Ph} (H_P\bs H)$ and $\gf: H \to \pi^*(\cD_{\HP\bs H})$
defines continuous section of
the pull-back bundle. It suffices to show that $\gf(h_Ph) = \gf(h)$
for all $h_P \in H_P.$ We note that $H_P = H\cap P = H_{M}\Ah H_{N_P}.$
Accordingly, write $h_P =m a n,$ then
\begin{eqnarray}
\nonumber
\gf(h_P h)
&=& a^{{-\gl }+ \rhoPh + \rhoP} \big\langle\xi(m)^{-1}\eta,f(h)\big\rangle\;
    dr_h([e]h_P)^{-1*}dr_{h_P}([e])^{-1*}\omega\\
\label{e: transformation gf}
&=& a^{ \rhoPh + \rhoP} \Delta(h_P) \gf(h),
\end{eqnarray}
where
$$
\Delta(h_p) = |\det \Ad(h_P)|_{\fh/\fh_P}| = |\det \Ad(h_P)|_{\fh_P}|^{-1}.
$$
Since $H_{N_P}$ is nilpotent, whereas $H_{M}$ is compact, it follows that
$$
\Delta(h_P) = \Delta(a) = |\det \Ad(a)|_{\fh_P}|^{-1}.
$$
Using the decomposition $\fh_P = (\fh\cap \fm) \oplus \fah \oplus (\fh \cap \fn_P)$
we finally see that
$$
\Delta(h_P) = |\det \Ad(a)|_{\fh \cap \fn_P}|^{-1} = a^{- \gd},
$$
where $\gd = \tr( \ad(\dotvar)|_{\fh \cap \fn_P}) \in \fahd.$
We now use that
$$
\fh \cap \fn_P = \bigoplus_{\ga \in (\gS(P) \cap \fahd)}\; \fg_\ga \;\;\oplus \;\; \Big( \bigoplus_{\ga \in \gS(P, \gs)\setminus \fahd}\;
\fg_\ga \Big)^\gs
$$
For each $\ga \in \gS(P, \gs)\setminus \fahd,$ we have $\gs \ga \neq \ga,$ and
the direct sum $\fg_{\ga} \oplus  \fg_{\gs\ga}$ is $\gs$-invariant, so that its intersection
with $\fh$ is given by
$$
(\fg_{\ga} \oplus  \fg_{\gs\ga})^\gs = \{ X + \gs(X)\mid X \in \fg_\ga\}.
$$
The action of an element $H \in \fah$ on this space has trace $\dim( \fg_\ga)\, \ga(H).$
We conclude that
$$
\gd = \big(2\rhoPh + \rho_{ \gS(P,\gs)\setminus \fahd}\big)\big|_{\fah}.
$$
Using the decomposition
$$
\rho_P = \rhoPh + \rho_{ \gS(P,\gs)\setminus \fahd} + \rho_{\gS(P, \gs\Cartan)}
$$
we see that
$$
\big(\rho_P +\rhoPh\big)\big|_{\fah} - \gd
= \rho_{\gS(P, \gs\Cartan)}\big|_{\fah}
=0.
$$
Combining this with (\ref{e: transformation gf})
we infer that $\gf$ is left $H_P$-invariant.
\end{proof}

\begin{rem}
Recall the definition of $\gS(P)_{-}$ in (\ref{e: def minus roots}).
In   Section \ref{Section H-fixed distribution vectors, q-extreme case} 
 we will show that
for all $\gl \in \faqdc$, for which there exists $P_0 \in \parabsgsAq$  with
$\gS(P, \gs\Cartan) \subseteq \gS(P_0)$ such that
$$
\forall \ga \in \gS(P)_- : \qquad   \inp{\Re \gl + \rho_{P_0}}{\ga} \leq 0,
$$
the above density $f_{\eta, \omega}$ is integrable over
$H_P\bs H.$
\end{rem}
\vspace{-12pt}
\section{Comparison of principal series representations}
\label{s: comparison of principal series representations}
In this section we will  compare
the principal series representations with the $\gs$-principal
series defined in \cite{Banps1}. The latter involve parabolic subgroups
$P_0$ from $\cP_\gs(\Aq).$ Each of these has a Langlands decomposition of the form $P_0 = M_0 A_0 N_{P_0},$ see the end of Section \ref{s: notation and preliminaries} for details.

We will now investigate the structure of the group $M_0$ in more detail.
Our starting point is the following lemma.

\begin{lemma}
\label{l: root space in fh}
Let $\ga \in \gS(\fg, \fa) \cap \fahd.$ Then $\fg_\ga \subseteq \fh.$
\end{lemma}

\begin{proof}
Let $\ga$ be as in the assertion. Then $\gs \ga = \ga$ so that $\gs$ leaves
the root space $\fg_\ga$ invariant. Thus, it suffices to show that $\fg_\ga \cap \fq = 0.$ Assume that $X \in \fg_\ga \cap \fq.$
Then
$(X -  \theta X)$ belongs to $\fp \cap \fq$ and centralizes $\faq.$
As the latter space is maximal abelian in $\fp \cap \fq,$ it follows that 
$
X - \Cartan X \in \faq \cap (\fg_\ga + \fg_{-\ga}) = 0.
$
\end{proof}

Let $\fm_{0\rmn}$ be the ideal in $\fm_0$ generated by $\fa \cap \fm_0.$
Since $\fa \cap \fm_0$ has trivial intersection with the center of $\fm_0$,
the ideal $\fm_{0\rmn}$ equals the sum of the simple ideals of non-compact type in $\fm_0.$ It has a unique complementary ideal; this  is contained in the centralizer of $\fa\cap \fm_0$ in $\fm_0,$ hence in $\fm.$

\begin{lemma}
\label{l: fm zero n inside fh}
The ideal $\fm_{0\rmn}$ is contained in $\fm_0 \cap \fh.$
\end{lemma}

\begin{proof}
The algebra $\fm_0$ admits the decomposition
\begin{equation}
\label{e: deco fm 0}
\fm_0 = \fm \oplus (\fa \cap \fm_0) \oplus \bigoplus_{\ga \in \gS(\fg, \fa) \cap \fahd} \fg_\ga.
\end{equation}
Each appearing root space $\fg_\ga$ equals $[\fa \cap \fm_0, \fg_\ga].$
Hence, $\fm_{0\rmn}$ contains the subspace
$$
\fs:= (\fa \cap \fm_0) \oplus \bigoplus_{\ga \in \gS(\fg, \fa) \cap \fahd} \fg_\ga.
$$
It follows that $\fm_{0\rmn}$ contains the subalgebra $\tilde \fs$ of $\fm_0$ generated by $\fs.$ On the other hand, since
$\fm_0 = \fm +\fs$ and $\fm$ normalizes $\fs,$ the algebra $\tilde \fs$ is an ideal of $\fm_{0\rmn}.$
We conclude that $\fm_{0\rmn}$ equals the algebra $\tilde \fs$ generated by $\fs.$

Now $\fa \cap \fm_0 \subseteq \fh$ and each of the root spaces in
(\ref{e: deco fm 0}) is contained in $\fh$ by
Lemma \ref{l: root space in fh}. Therefore, $\fs \subseteq \fh$ and we conclude that $\fm_{0\rmn} = \tilde \fs \subseteq \fh.$
\end{proof}

Let $M_{0\rmn}$ be the connected
subgroup of $M_0$ with Lie algebra $\fm_{0\rmn}.$

\begin{lemma}
\label{l: lemma on deco M zero}
\begin{enumerate}
\itema
$M_{0\rmn}$ is a closed normal subgroup of $M_0.$
\vspace{-5pt}
\itemb
$
M_0 = M M_{0\rmn} \simeq M \times_{M\cap M_{0\rmn}} M_{0\rmn}.
$
\vspace{-5pt}
\itemc
The inclusion map $M \to M_0$ induces a group isomorphism
$
M/M\cap M_{0\rmn} \simeq M_0/M_{0\rmn}.
$%
\vspace{-5pt}
\itemd
$H_{M_0}= H_{M}M_{0\rmn}$
\vspace{-5pt}
\iteme
The inclusion map $M \to M_0$ induces a diffeomorphism
$
M/ H_{M} \simeq M_0 / H_{M_0}.
$
\vspace{-5pt}
\itemf
The group $M_{0\rmn}$ acts trivially on $M_0 / M_{0\rmn}$ and on $M_0 / H_{M_0}.$
\end{enumerate}
\end{lemma}

\begin{proof}
The normality of $M_{0\rmn}$ follows since $\fm_{0\rmn}$ is an ideal of $\fm_{0}$. Since $\fm_{0}$ is reductive, there exists an ideal $\fm_{0c}$ complementary to $\fm_{0\rmn}$. The group $M_{0\rmn}$ is equal to the connected component of $Z_{M_{0}}(\fm_{0c})$ and therefore $M_{0\rmn}$ is closed. This proves assertion (a).
From $\fm_0 = \fm + \fm_{0\rmn}$ and (a) it follows
that $M M_{0\rmn}$ is an open subgroup of $M_0.$ Since $M_0$ is of the Harish-Chandra class, and $M = Z_{K \cap M_0}(\fa \cap \fm_0),$  it follows
that $M$  intersects every connected component of $M_0.$ Hence, $M_0 = M M_{0\rmn}$ and (b) readily follows.
Assertion (c) follows from (b) and (a). We now turn to assertion (d).
From Lemma \ref{l: fm zero n inside fh} it follows that $M_{0\rmn} \subseteq H.$
In view of (b) we now see that
$$
H_{M_0} = [M M_{0\rmn}]\cap H = H_{M}M_{0\rmn},
$$
Hence (d).
From (c) and (d) we obtain a natural fiber bundle
$M/M\cap M_{0\rmn} \to M/H_M$
which corresponds to factorization by the group $F:= H_{M}/(M\cap M_{0\rmn}).$ Likewise, we obtain a natural fiber bundle $M_0 / M_{0\rmn} \to M_0 / H_{M_0}$ which corresponds to factorization by group $F_0:= H_{M_0}/ M_{0\rmn}.$ The isomorphism of (c) maps $F$ onto $F_0,$ hence (e) follows.

Since $M_{0\rmn}$ is normal in $M_0,$ it acts trivially on the quotient $M_0/M_{0\rmn}.$ The second assertion of (f) follows from this as $M_{0\rmn}\subseteq H_{M_0}.$
\end{proof}

Given a continuous Fr\'echet $M_0$-module $V$, we denote its space of
smooth vectors by $V^\infty.$ This comes equipped with the structure of a continuous Fr\'echet $M_0$-module in the usual way.  The continuous linear
dual is denoted by   ${V^\infty}'.$ 

\begin{cor}
Let $(\xi, V)$ be an irreducible continuous representation of $M_0$ in
a Fr\'echet space $V$ such that
\begin{equation}
\label{e: space of H cap M0 distribution vectors}
({V^{\infty}}')^{H_{M_0}} \neq 0.
\end{equation}
Then $\xi|_{M_{0\rmn}}$ is trivial and $\xi|_M$ is irreducible.
In particular, $\xi$ is finite dimensional and unitarizable.
\end{cor}

\begin{proof}
Let $\eta$ be a non-zero element of the space in (\ref{e: space of H cap M0 distribution vectors}). Then there is a unique injective continuous linear  $M_0$-equivariant map $j: V^\infty \to C^{\infty}(M_0 / H_{M_0})$ such that $j^*(\gd_{[e]}) = \eta,$ with $\gd_{[e]}$ denoting the
Dirac measure of $M_0/ H_{M_0}$ at $[e] :=  eH_{M_0}.$
Since $M_{0\rmn}$ acts trivially on $M_0 /H_{M_0}$ it follows that
$M_{0\rmn}$ acts trivially on $V^\infty$ hence on $V.$ We conclude that
$\xi|_{M_{0\rmn}}$ is trivial. By application of Lemma \ref{l: lemma on deco M zero} it follows that $\xi|_{M}$ is irreducible.
\end{proof}

The above result provides motivation for considering only finite dimensional unitary representations of $M_0.$ We note that any such representation restricts to the trivial representation on $M_{0\rmn},$ since the latter group is connected semisimple of the non-compact type. Since $M_0/M_{0\rmn}$ is a compact group, it follows that
\begin{equation}
\label{e: dMzerofu}
\dMzerofu \simeq (M_0/M_{0\rmn})^\wedge,
\end{equation}
where $\widehat M_{0\fu}$ denote the set of equivalence
classes of finite dimensional irreducible unitary representations
of $M_0.$

\begin{lemma}
\label{l: embedding widehat Mzero}
The restriction map $\xi \mapsto \xi_M:= \xi|_M$ induces an injection
\begin{equation}
\label{e: embedding widehat M}
 \widehat M_{0\fu}  \hookrightarrow \widehat M.
\end{equation}
The image of this injection equals $(M/ M \cap M_{0\rmn})^\wedge.$
\end{lemma}

\begin{proof}
It follows from Lemma \ref{l: lemma on deco M zero} (c) that the restriction map induces an isomorphism
$$
(M_0/M_{0\rmn})^\wedge \simeq (M/M\cap M_{0\rmn})^\wedge.
$$
The latter set may be viewed as the subset of $\widehat M$ consisting of
equivalence classes of irreducible unitary representations that are trivial on
$M\cap M_{0\rmn}.$ Now use (\ref{e: dMzerofu}).
\end{proof}

From now on we will use the map
(\ref{e: embedding widehat M}) to view $\widehat M_{0\fu}$ as a subset of $\widehat M.$

\begin{lemma}
\label{l: equality of H fixed vectors}
Let $(\xi,\cH_\xi)$ be a finite dimensional unitary representation of $M_0$
(not necessarily irreducible).
Then
\begin{equation}
\label{e: equality of H fixed vectors}
\cH_\xi^{H_{M_0}} = \cH_{\xi}^{H_{M}}.
\end{equation}
\end{lemma}
\begin{proof}
The space on the left-hand side of the equation is clearly contained
in the space on the right-hand side.
For the converse inclusion, let $\eta \in \cH_\xi$ be an $H_{M}$-fixed
vector. Then $\eta$ is fixed under the group $H_{M}M_{0\rmn},$ which
equals $H_{M_0}$ by Lemma \ref{l: lemma on deco M zero} (d).
\end{proof}

Let $W(\faq)$ denote the Weyl group of the root system $\Sigma(\fg, \faq)$.
Then $W(\faq) \simeq  N_K(\faq)/ Z_K(\faq),$ naturally.  We denote by $W_{K\cap H}(\faq)$ the image
of $N_{K\cap H}(\faq)$ in $W(\faq).$
Let $W(\fa) \simeq N_K(\fa)/Z_K(\fa)$ denote the Weyl group of the root system $\Sigma(\fg, \fa)$. Then restriction to $\faq$ induces an epimorphism
from the normalizer of $\faq$ in $W(\fa)$ onto $W(\faq).$ We may therefore select a finite subset
$
\cW \subseteq N_K(\fa) \cap N_K(\faq)
$
such that $e \in \cW$ and such that the map $v \mapsto \Ad(v)|_{\faq}$ induces a bijection
\begin{equation}
\label{e: defi cW}
\cW\;\;\buildrel{1-1}\over \longrightarrow \; \;W(\faq)/W_{K \cap H}(\faq).
\end{equation}
Let $\xi $ be a finite dimensional unitary representation of $M_0$ (not necessarily irreducible).
Then following \cite{Banps1} we define
\begin{equation}
\label{e: defi V xi v}
V(\xi, v) := \cH_\xi^{M_0\cap vHv^{-1}} = \cH_\xi^{M \cap v H v^{-1}}.
\end{equation}
Here we note that the second equality is valid by Lemma \ref{l: equality of H fixed vectors} applied with $v H v^{-1}$ in place of $H.$
We equip the space in (\ref{e: defi V xi v}) with
the restriction of the inner product
on $\cH_\xi.$
Finally we define the formal direct sum of Hilbert spaces
\begin{equation}\label{e: defi V xi}
V(\xi) := \oplus_{v \in \cW} V(\xi, v).
\end{equation}
For $v\in \cW,$ let 
$
\rmi_v: V(\xi, v) \to V(\xi)$ and $\pr_v: V(\xi) \to V(\xi, v)
$
denote the natural inclusion and projection map,  respectively.

Our goal will be to study $H$-fixed distribution vectors in representations
induced from minimal parabolic subgroups $P \in \parabsA.$ For this it will be convenient
to compare these representations to representations induced from minimal $\gs\Cartan$-stable parabolic subgroups, by using the method of induction  in   stages.

Let $P \in \parabsgsA$, see Definition \ref{d: q extreme parabs}, and let $P_0 \in \parabsgsAq$ be such that
$P \subseteq P_0.$ Let $(\xi, \cH_\xi)$ be a finite dimensional unitary representation of $M_0$ and let
$C^\infty(P_0: \xi : \gl)$
be
defined as in the first part of Section \ref{s: induced reps and densities} for $P_0$ in place of $P.$
We agree to write $\xi_M: = \xi|_M.$
Observe that $P \cap M_0$ is a minimal parabolic subgroup
of $M_0$ with split component $A\cap M_0.$ Moreover, since the set of roots of
$\fa \cap \fm_0$ in $N_P \cap M_0$ equals $\gS(P) \cap \fahd,$ it follows that
$$
\rho_{P \cap M_0} = \rhoPh.
$$
Hence, there is a natural $M_0$-equivariant embedding
$$
i: \xi \embeds \Ind_{M_0 \cap P}^{M_0}(\xi_M \otimes -\rhoPh \otimes 1),
$$
see \cite[Lemma 4.4]{Banps1}. Concretely, the map $i$ from $\cH_\xi$ into
the space
$C^\infty(M_0\cap P: \xi_M: -\rhoPh)$
of smooth vectors  for the principal series
representation on the right-hand side is given by
\begin{equation}
\label{e: defi i}
i(v)(m_0) = \xi(m_0) v,\qquad (v \in \cH_\xi, \;m_0 \in M_0).
\end{equation}
Induction now gives a $G$-equivariant embedding
$$
\Ind_{P_0}^G (\xi \otimes \gl\otimes 1)
\embeds \Ind_{P_0}^G \big(\Ind_{M_0 \cap P}^{M_0}(\xi_M\otimes -\rhoPh \otimes 1)\otimes \gl \otimes 1)\big).
$$
According to the principle of induction  in   stages,  see 
\cite[\S 7.2]{Knappbook},  the
 latter representation is naturally isomorphic with
$\Ind_{P}^G (\xi_M \otimes (\gl - \rhoPh) \otimes 1).$
The resulting $G$-equivariant embedding
\begin{equation}
\label{e: i sharp}
\injP: C^\infty(P_0 : \xi : \gl) \to C^\infty(P: \xi_M : \gl - \rhoPh)
\end{equation}
is given by $(\injP f)(x) = \ev_1 \after i \after f(x)$
for $f \in C^\infty(P_0 : \xi : \gl)$ and $x \in G.$
Here,
$$
\ev_1: C^\infty(M_0\cap P: \xi_M: -\rhoPh) \to \cH_\xi
$$
is given by evaluation
at the identity of $M_0.$ Comparing this with
(\ref{e: defi i}) we see that $\injP$ is
the inclusion map.

By $C^\infty(K: K\cap M_0 : \xi)$ we denote the space of smooth functions $K\to\cH_\xi$ transforming according to the rule
$$
f(mk)=\xi(m)f(k)\qquad(m\in K\cap M_0, \;k\in K).
$$
Likewise, we write $C^\infty(K: M : \xi_M)$ for the space of smooth functions $K\to\cH_\xi$ transforming according to the rule  
$
f(mk)=\xi_M(m)f(k),
$
for all $m\in M$ and $k\in K.$  Note that restriction to $K$ induces topological  linear isomorphisms $C^\infty(P_0 : \xi : \gl) \to C^\infty(K:K\cap M_0:\xi)$ and $C^\infty(P: \xi_M : \gl - \rhoPh)\to C^\infty(K:M:\xi_M)$.

In these compact pictures of the induced representations, $\injP$ becomes
the inclusion map
\begin{equation}
\label{e: i sharp on K}
\inj:  C^\infty(K: K\cap M_0 : \xi) \embeds C^\infty(K: M : \xi_M).
\end{equation}
We now see that we have the following commutative diagram
\begin{equation}
\label{e: diagram i sharp}
\begin{array}{ccc}
C^{\infty}(P_0 : \xi : \gl) & {\buildrel \injP \over \longrightarrow}
 & C^\infty(P: \xi_M : \gl - \rhoPh)\\
 \downarrow && \downarrow\\
C^\infty(K: K\cap M_0 : \xi)&  {\buildrel \inj \over \longrightarrow}&
 C^\infty(K: M : \xi_M).
 \end{array}
\end{equation}
The vertical arrows in this diagram represent the topological
linear isomorphisms induced by restriction to $K$; see
\cite{Banps1} for details.

\begin{lemma}
\label{l: image i sharp and invariants}
The space $C^\infty(K: K\cap M_0 : \xi)$ coincides with
the subspace of left $K\cap M_{0\rmn}$-invariants in $C^\infty(K: M : \xi_M).$
\end{lemma}

\begin{proof}
Let $f \in C^\infty(K: K\cap M_0 : \xi).$
Since $\xi$ is finite dimensional, it follows that $\xi|_{M_0\rmn}$ is trivial.
 Hence, for $k \in K$ and $m_0 \in M_{0\rmn}$
we have that
$$
f(m_0 k) = \xi(m_0) f(k) = f(k).
$$
This establishes one inclusion. For the converse, assume that
$f \in C^\infty(K: M : \xi_M)$ is left  $K \cap M_{0\rmn}$-invariant.
Let $m_0 \in K \cap M_0.$ Then we may write $m_0 = m m_{\rm n}$
with $m \in M$ and $m_{\rm n} \in K\cap M_{0\rmn}.$ Let $k \in K.$
Then
\begin{eqnarray*}
f(m_0 k) = f(m m_{\rm n} k) &=&
\xi(m) f(m_{\rm n} k) = \xi(m) f(k) \\
&=& \xi(m)\xi(m_{\rm n}) f(k) = \xi(m_0)f(k).
\end{eqnarray*}
For the third equality we used that $\xi|_{M_{0\rmn}}$ is trivial.
We thus conclude  that $f$ belongs to $ C^\infty(K: K \cap M_0: \xi).$
\end{proof}

Since $M$ normalizes $K \cap M_{0\rmn},$ we see that for $f \in C^\infty(K: M : \xi_M)$ 
the function $p(f): K \to \cH_\xi$ defined by 
$$
p(f)(k) = \int_{K \cap M_{0\rmn}} f(m_0 k)\; dm_0, \quad (k \in K),
$$
belongs to $C^\infty(K : M : \xi_M)$ again. The associated operator 
\begin{equation}
\label{e: defi p}
p: C^\infty(K: M: \xi_M) \to C^\infty(K: M: \xi_M) 
\end{equation}
is continuous linear and $K$-equivariant. Since $K \cap M_0 = M (K \cap M_{0\rmn}),$ the image of $p$ is
contained in the subspace of left $K \cap M_{0\rmn}$-invariants. Furthermore, $p$ is obviously
the identity on this subspace, so that $p$ is a projection operator with image equal 
to the image ${\rm im} (\inj)$ of $\inj,$ see (\ref{e: i sharp on K}).
It is readily seen that $p$ is symmetric with respect
to the pre-Hilbert structure $\inp{\dotvar}{\dotvar}$ on
$C^\infty(K :  M : \xi_{M})$ obtained by restriction of the inner product from 
$L^2(K) \otimes \cH_\xi.$ Accordingly, $C^\infty(K: M:\xi_M)$ is the direct sum
of ${\rm im}(\inj)$ and its orthocomplement with respect to the give pre-Hilbert structure,
and $p$ is the associated orthogonal projection onto  ${\rm im}(\inj).$ 
Let
\begin{equation}
\label{e: defi proj}
\proj: C^\infty(K: M : \xi_M) \to C^\infty(K: K \cap M_0: \xi)
\end{equation}
be the unique linear map such that $p = \inj \after \proj.$  For later use, we note that the maps 
introduced are related by 
\begin{equation}
\label{e: relations p proj and inj}
p = \inj \after \proj, \qquad \proj\after \inj = I.
\end{equation} 
The map $\proj$ is $K$-equivariant and continuous linear, 
and $\inj$ and $\proj$ are adjoint with respect to the pre-Hilbert structures
$\inp{\dotvar}{\dotvar}.$ 

\begin{lemma} 
\label{l: equivariance of inj and proj}
For all $\gl \in \faqdc,$ the following holds.
\begin{enumerate}
\itema
The map $\inj$ given in {\rm (\ref{e: i sharp on K})} intertwines the representations
$\pi_{P_0, \xi, \gl}$ and $\pi_{P, \xi_M, \gl - \rhoPh}.$%
\vspace{-5pt}%
\itemb
The map $\proj$ intertwines the representations $\pi_{P,\xi_M, \gl + \rhoPh}$ 
and $\pi_{P_0, \xi, \gl}.$ 
\end{enumerate}
\end{lemma}

\begin{proof} 
Since the top horizontal map in (\ref{e: diagram i sharp}) is intertwining,
(a) follows. Using (a) we see that for each 
$x \in G,$ 
$$ 
\inj \after \pi_{P_0, \xi, \gl}(x) = \pi_{P, \xi_M, \gl - \rhoPh}(x) \after \inj.
$$ 
Taking adjoints and using equivariance of the pairings $\inp{\dotvar}{\dotvar}$ 
involved, we infer that 
$$ 
\pi_{P_0, \xi, -\bar \gl}(x^{-1}) \after \proj = \proj \after \pi_{P, \xi_M, - \bar \gl + \rhoPh}(x^{-1})
$$ 
for all $\gl \in \faqdc$ and $x \in G.$ From this, (b) follows.
\end{proof}

For each $\gl \in \faqdc$ we denote the unique lift of the map  (\ref{e: defi proj}) to a map
$C^\infty(P:\xi_M: \gl + \rhoPh) \to C^\infty(P_0: \xi: \gl)$ by $\projP$. Then it follows from 
Lemma \ref{l: equivariance of inj and proj} that the following diagram commutes:
\begin{equation}
\label{e: diagram proj}
\begin{array}{ccc}
C^{\infty}(P : \xi_M : \gl + \rhoPh) & {\buildrel \projP \over \longrightarrow}
 & C^{\infty}(P_0: \xi : \gl)\\
 \downarrow && \downarrow\\
C^\infty(K: M : \xi_M)&  {\buildrel \proj \over \longrightarrow}&
 C^\infty(K: K\cap M_0 : \xi).
 \end{array}
\end{equation}

\begin{rem}
\label{rem: warning on p}
In view of  (\ref{e: relations p proj and inj}) it follows from 
Lemma \ref{l: equivariance of inj and proj} that 
$p$ defined in (\ref{e: defi p}) intertwines the representations $\pi_{P:\xi_M:\gl + \rhoPh}$ 
and $\pi_{P:\xi_M: \gl - \rhoPh}$ of $G$ in $C^\infty(K:M:\xi_M).$ 
Thus, it has a unique lift to an equivariant 
map $C^\infty(P:\xi_M:\gl + \rhoPh) \to C^\infty(P:\xi_M:\gl - \rhoPh).$ However,
we shall never use this lift.
\end{rem}

\begin{prop}
Let $P_0 \in \parabsgsAq$ and $P \in \parabsA$ be such that $P \subseteq P_0.$
Let $\xi \in \widehat M_{0\fu}$ and $\gl \in \faqd.$
Then the embedding {\rm (\ref{e: i sharp})}
has a unique extension to a continuous
linear map
\begin{equation}
\label{e: extended i sharp}
 \injP:  \Cminf (P_0 : \xi : \gl) \to
\Cminf (P: \xi_M : \gl - \rhoPh).
\end{equation}
This extension is $G$-equivariant and maps homeomorphically onto a closed subspace.
As a  map from $\Cminf(K: K\cap M_0: \xi)$ to $ \Cminf(K: M: \xi_M)$
the extended map  is the unique continuous extension of (\ref{e: i sharp on K}). In particular, it is independent of $\gl.$
\end{prop}

\begin{proof}
Let the map 
$$
 \inj: \Cminf(K : K \cap M_0 : \xi) \to \Cminf(K: M : \xi_M)
 $$
be defined as the transposed of $\proj.$ 
Then $\inj$ is a continuous linear extension
of the bottom horizontal map of (\ref{e: diagram i sharp}). This continuous extension is unique
by density of $C^\infty(K: K \cap M_0 : \xi)$ in $\Cminf(K: K \cap M_0 : \xi).$
Likewise, the adjoint $i^{\# \,\mathrm{T}}$ 
of the bottom horizontal map in (\ref{e: diagram i sharp})
is the continuous linear extension of the projection map $\proj$ which we denote by 
$$ 
\proj: C^{-\infty}(K:M:\xi_M) \to C^{-\infty}(K: M_0 \cap K : \xi)
$$ 
as well. Finally, the transpose $p^{\rm T}$ is the unique continuous linear extension of $p$ to a 
continuous linear map, denoted
$$ 
p: C^{-\infty} (K:M:\xi_M) \to C^{-\infty}(K:M:\xi_M).
$$ 
By transposition we see that the relations (\ref{e: relations p proj and inj}) remain valid for the extensions
of these maps to the spaces of generalized functions involved. In particular, it 
follows that the extended map $\inj$ is a homeomorphism onto the kernel of the extended
map $p - I.$ In particular, it has closed image.

By transfer under the vertical isomorphisms in the diagram (\ref{e: diagram i sharp})
we see that $\injP$ has a unique continuous linear extension 
(\ref{e: extended i sharp}) with closed image. The extension is $G$-equivariant because it is so on the dense subspace of smooth functions.
\end{proof}

\begin{rem}
By a similar argument it follows that the map $\projP$ represented by the 
top horizontal arrow in (\ref{e: diagram proj}) has a unique continuous linear extension
to a surjective equivariant map $\projP: C^{-\infty}(P:\xi_M:\gl + \rhoPh) \to C^{-\infty}(P_0:\xi:\gl).$ 
However, we shall not need this in the present paper.
\end{rem}

\section{H-fixed distribution vectors, the q-extreme case}
\label{Section H-fixed distribution vectors, q-extreme case}
We retain the notation of the previous section. In particular, we assume that
$P\in \parabsgsA$ and that $P_0 \in \parabsgsAq$ contains $P.$ We will now construct $H$-fixed
distribution vectors in $P$-induced representations, by comparison with
the $H$-distribution vectors in $P_0$-representations as defined in \cite{Banps1}.

We assume that $\xi$ is a finite dimensional unitary representation of $M_0$
 and put $\xi_M = \xi|_M.$ Furthermore,
we assume that $\eta \in V(\xi, e),$ see (\ref{e: defi V xi v}).

Following \cite[(5.4)]{Banps1} we define the function $\geps_1(P_0: \xi: \gl:\eta)$
for $\gl \in \faqdc$ by
$$
\left\{
\begin{array}{l}
\geps_1(P_0: \xi:  \gl:  \eta) = 0  \qquad{\rm outside}\quad  P_0H\\
\geps_1(P_0: \xi: \gl: \eta)(namh) = a^{\gl + \rho_{P_0}}\xi(m) \eta,
\end{array}
\right.
$$
for $m \in M_0, a \in A_0, n \in N_0$ and $h \in H.$ Clearly, for every $\gl \in \faqdc$
the function $\geps_1(P_0: \xi: \gl: \eta)$ is continuous outside the set $\partial(P_0 H)$
which has measure zero in $G.$ By right-P-equivariance, the restriction of this function to $K$ is continuous outside
$\partial (K \cap P_0 H),$ which has measure zero in $K.$

Let $\gS(P_0, \faq)_-$ denote the space of $\faq$-roots in $\fn_{P_0}$
such that  $\ker( \Cartan \gs + I) \cap \fg_\ga \neq 0.$
In case $\xi$ is irreducible, it follows from \cite[Prop. 5.6]{Banps1} that the function
$\geps_1(P_0: \xi: \gl: \eta)$ is continuous on $G$ for all
$\gl \in \faqdc$ with $\inp{\Re \gl + \rho_{P_0}}{\ga} < 0$ for all $\ga \in \gS(P_0, \faq)_-.$
By decomposition into irreducibles one readily sees that this result
is also valid for an arbitrary finite dimensional unitary representation
of $M_0.$

\begin{lemma} Let $\xi$ be a finite dimensional unitary representation of $M_0$
and assume that $\gl \in \faqdc$ satisfies
\begin{equation}
\label{e: condition integrability geps}
\forall  \ga \in \gS(P_0, \faq)_- :\qquad \inp{\Re \gl + \rho_{P_0}}{\ga} \leq  0.
\end{equation}
Then the function $\geps_1(P_0: \xi: \gl: \eta)$ is measurable and locally bounded on $G,$
and its restriction to $K$ is measurable and bounded on $K,$ uniformly for $\gl$ in the indicated subset
of $\faqdc.$ Finally, $\geps_1(P_0: \xi: \gl: \eta)|_K$ depends continuously on $\gl$ as a function
with values in $L^1(K)\otimes \cH_\xi.$
\end{lemma}

\begin{proof}
We may as well assume that $\xi$ is irreducible. The assertions about measurability
have been settled above. For the assertions about boundedness, it suffices to
consider the restriction of the function to $K.$ From the argument in the
proof of \cite[Prop. 5.6]{Banps1}, which in turn relies on the convexity theorem
of \bib{Bconv}, it follows that for all $\gl$ in the indicated region we have
$$
\sup_K \| \geps_1(P_0: \xi: \gl: \eta)\|_\xi  \leq \|\eta\|_\xi.
$$
We obtain the final assertion by observing that $\geps_1(P_0: \xi: \gl: \eta)$ depends pointwise
continuously on $\gl$ and applying Lebesgue's dominated convergence theorem.
\end{proof}

\begin{prop}\label{p: convergence density on HP H}
Let $P\in \parabsgsA$ and $P_0 \in \parabsgsAq$ be such that $P \subseteq P_0.$
Let $\xi$ be a finite dimensional unitary representation of $M_0$
and   $\eta \in V(\xi,e).$
Let $\gl \in \faqdc$ be such that
$$
\forall \ga \in \gS(P)_-: \qquad \inp{\Re \gl + \rho_{P_0}}{\ga} \leq 0.
$$
Finally, let $f \in C^\infty(P:\xi_M:-\bar \gl + \rhoPh).$ Then the density
$\bar f_{\eta,\omega},$ defined in Lemma \ref{l: density associated with f},
is integrable. Moreover,
\begin{equation}
\label{e: integral f and epsilon}
\int_{\HP \bs H}  \bar f_{\eta, \omega} =
c_\omega \inp{\injP \big(\geps_1(P_0:\xi: \gl: \eta)\big)}{f},
\end{equation}
with $c_\omega > 0$ a constant depending on the normalization of the positive density $\omega.$
\end{prop}

\begin{proof}
By the assumption on $\gl,$ the function $\geps = \injP\big(\geps_1(P_{0}:\xi:\gl:\eta)\big)$ is locally integrable on $K.$
It follows that the expression on the right-hand side of (\ref{e: integral f and epsilon}) equals the integral
$$
\int_K \inp{\geps(k)}{f(k)}_\xi\; dk,
$$
where $dk$ denotes normalized Haar measure on $K.$ The integrand is
left $M$-invariant, so that the integral may also be written as the integral
over $k \in M\bs K,$ with $dk$ replaced with the normalized invariant
density $d\bar k$ on $M \bs K.$ This density may be viewed as the section
of the density bundle over $M \bs K$ given by
$$
k \mapsto dr_k([e])^{-1*} \omega_{M\bs K}
$$
with $\omega_{M\bs K}$ a suitable positive density on $\fk/\fm \simeq T_{[e]} (M \bs K).$
We now obtain that
\begin{equation}
\label{e: first integral for pairing}
\inp{\geps}{f}
= \int_{M\bs K} \inp{\geps(k)}{f(k)}_\xi\; dr_k([e])^{-1*} \omega_{M\bs K}.
\end{equation}
Let $\phi: M\bs K \to P\bs G$ be the diffeomorphism induced by the inclusion
$K \to G.$
Then we find that the pull-back under $\phi$ of the density in the
integral in the right-hand side of (\ref{e: first integral for pairing}) equals
$$
x \mapsto \inp{\geps(x)}{f(x)}\;  dr_x([e])^{-1*}
d\phi([e])^{-1*}\omega_{M\bs K}.
$$
Since $P_0 H = P H,$ it follows that
that the above density is supported by $PH.$
Writing $\omega_{P \bs G} = d\phi([e])^{-1*}\omega_{M \bs K},$
we obtain that the integral in (\ref{e: first integral for pairing}) equals
\begin{equation}
\label{e: second integral for pairing}
\int_{P \cdot H} \inp{\geps(x)}{f(x)} dr_x([e])^{-1*}
\omega_{P \bs G}.
\end{equation}

Let $\psi: \HP \bs H \to P \bs G$ be the natural open embedding
induced by the inclusion map $H \to G.$ Then
$|d\psi([e])^* \omega_{P\bs G} | = c_\omega^{-1} \omega_{P\bs G}$
for a positive constant $c_\omega.$ We now observe that
\begin{eqnarray*}
\lefteqn{
\psi^*\Big( Px \mapsto \inp{\geps(x)}{f(x)}_\xi\,\, dr_x([e])^{-1*} \omega_{P \bs G}\Big)}\\
&=& c_\omega^{-1}  \cdot \Big( H_P h \mapsto \inp{\geps(h)}{f(h)}_\xi\,\, dr_h([e])^{-1*} \omega\Big) \\
&=& c_\omega^{-1} \bar f_{\eta, \omega}.
\end{eqnarray*}
By invariance of integration of densities under diffeomorphisms,
we see that (\ref{e: second integral for pairing}) equals
$$
c_\omega^{-1} \int_{\HP\bs H} \bar f_{\eta, \omega}.
$$
\end{proof}

For $\gl \in \faqdc$ such that the conditions of the above theorem
are fulfilled, and for $\eta \in V(\xi,e),$
we define the conjugate-linear functional
$j_H(P: \xi_M: \gl:\eta)$ on $C^\infty(P: \xi_M: -\bar \gl + \rhoPh)$ by
\begin{equation}
\label{e: defi J H P}
\inp{j_H(P: \xi_M: \gl:\eta)}{f} =
c_\omega^{-1} \int_{H_P \bs H} { \bar f_{\eta, \omega}},
\end{equation}
for $f \in C^\infty(P: \xi_M: -\bar \gl + \rhoPh).$

We now recall the definition of the $H$-fixed distribution vector $j(P_{0},\xi,\lambda)$ from \cite[Section 5]{Banps1}.
For $\gl \in \faqdc$ such that
$$
\forall \ga \in \gS(P)_-: \qquad \inp{\Re \gl + \rho_{P_0}}{\ga} \leq 0
$$
and for $v\in\cW$ and $\eta\in V(\xi,v)$ we define $\epsilon_{v}(P_{0}:\xi:\lambda:\eta):G\to\cH_{\xi}$ by
$$
\left\{
\begin{array}{l}
\geps_v(P_0:\xi: \gl: \eta) = 0  \qquad{\rm outside}\quad  P_0vH\\
 \geps_v(P_0:\xi:\gl:\eta)(namvh) = a^{\gl + \rho_{P_0}}\xi(m) \eta.
\end{array}
\right.
$$
We further define
$$
j(P_{0}: \xi : \gl)(\eta)
=\sum_{v \in \cW} \geps_{v}(P_{0} : \xi : \gl: \eta_{v}),
    \qquad\big(\eta\in V(\xi)\big).
$$
Then $j(P_{0}:\xi:\lambda)$ is a map $V(\xi)\to \Cminf(P_{0}:\xi:\lambda)^{H}$, hence defines an element in $V(\xi)^* \otimes \Cminf(K : K \cap M_0:\xi)$. The map $\lambda\mapsto j(P_{0}:\xi:\lambda)$ extends to a meromorphic $V(\xi)^* \otimes \Cminf(K : K \cap M_0:\xi)$-valued function on $\faqdc$. See \cite[Section 5]{Banps1} for details.
(Strictly speaking the definition in \cite{Banps1}
is given for $\xi$ irreducible, but the definition works equally well in general.)

Proposition \ref{p: convergence density on HP H} now has the following corollary.

\begin{cor}
Let $\xi$ be a finite dimensional unitary representation of $M_0.$
The map $\gl \mapsto j_H(P:\xi_M :\gl)$ extends to a meromorphic
$V(\xi, e)^* \otimes \Cminf(K : M : \xi_{M})$-valued function.
Moreover,
$$
j_H(P: \xi_{M} : \gl) = \inj \after j(P_0: \xi : \gl)\after \rmi_e
$$
as an identity of meromorphic
$V(\xi, e)^* \otimes \Cminf(K : M : \xi_{M})$-valued functions.
In particular,
$$
j_H(P: \xi_M : \gl) \in V(\xi, e)^* \otimes
\Cminf(P: \xi_M: \gl - \rhoPh)^H
$$
for generic $\gl \in \faqdc.$
\end{cor}

\begin{rem} In the above formulation we have used the notation $\inj$ rather than $\injP,$ to emphasize
that $j(P_0: \xi : \gl)\after \rmi_e$ is viewed as  a  $\gl$-dependent element of the space
$V(\xi, e)^* \otimes \Cminf(K : M : \xi_{M}).$ 
\end{rem}

Let $v \in \cW.$
Motivated by the definition of
$j(P_{0}:\xi:\gl)$ and the above identity, we define the meromorphic $\Hom\big(V(\xi), \Cminf(K:M:\xi_M)\big)$-valued
map $j(P:\xi_{M}:\dotvar)$ by
$$
j(P: \xi_{M} : \gl) =
\sum_{v \in \cW} \pirep_{P, \xi_M, \gl - \rhoPh}(v^{-1})
j_{v H v^{-1}}(P : \xi_{M} : \gl) \after \pr_v.
$$

\begin{cor}
\label{c: equality of jP and jP0}
Let $\xi $ be a finite dimensional unitary representation of $M_{0}.$ Then
\begin{equation}
\label{e: j P and j P zero}
j(P: \xi_M : \gl) = \inj \after j(P_0: \xi :\gl)
\end{equation}
as an identity of meromorphic
$V(\xi)^* \otimes \,\Cminf(K: M : \xi_M)$-valued
functions of $\gl \in \faqdc.$
In particular, for $\eta \in V(\xi)$ and generic $\gl \in \faqdc,$
$$
j(P:\xi_M:\gl)(\eta) \in \Cminf(P:\xi_M: \gl - \rhoPh)^H.
$$
\end{cor}

We recall that the map $p$ in the result below is exclusively used in the compact
picture of the induced representations, see Remark \ref{rem: warning on p}.
\begin{cor}
\label{c: p and j P}
Let $\xi \in \dMzerofu$ and $\eta \in V(\xi).$ Then for every $x \in G,$ 
$$ 
p \after \pi_{P, \xi_M, \gl - \rhoPh}(x) j(P:\xi_M:\gl)(\eta) =  \pi_{P,\xi_M, \gl- \rhoPh}(x) j(P:\xi_M:\gl)(\eta)
$$ 
as an identity of meromorphic $C^{-\infty}(K:M:\xi_M)$-valued functions of $\gl \in \faqdc.$ 
\end{cor}

\begin{proof}
Use (\ref{e: j P and j P zero}) and note that $\pi_{P, \xi_M, \gl- \rhoPh}(x) \after \inj = \inj \after 
\pi_{P_0, \xi, \gl}(x)$ and $p \after \inj = \inj,$ see (\ref{e: relations p proj and inj}).
\end{proof}

\section{An important fibration}
\label{s: an important fibration}
In this section we will apply Fubini's theorem, as formulated in
the appendix, Theorem \ref{t: fubini for group fibering}, to an important fibration.
The main result will be needed for the definition of distribution vectors for induced representations with
 $P \in \parabsA$ not necessarily contained in a parabolic subgroup from $\parabsgsAq.$

We assume that $P, Q \in \parabsA$ and that $P \succeq Q.$ There exists $X \in \faq$ such that
\begin{enumerate}
\itema
$\ga(X) \neq 0$ for all $\ga \in \gS(P)\setminus \fahd;$
\vspace{-5pt}
\itemb
$\ga(X) > 0$ for all $\ga \in \gS(P, \gs\Cartan).$
\end{enumerate}
Since $\gS(Q, \gs\Cartan) \subseteq \gS(P, \gs\Cartan),$ it follows
that (a) and (b) are also valid with $Q$ in place of $P.$
We now put
$$
\fn_{Q,X}:= \bigoplus_{\substack{\ga \in \gS(Q)\\ \ga(X) > 0}}
\;\; \fg_\ga,\quad {\rm and} \quad N_{Q,X}:= \exp(\fn_{Q,X}).
$$

\begin{lemma}
\label{l: iso NQ mod H}
The multiplication map $(n_1, n_2) \mapsto n_1 n_2$  is  a diffeomorphism
$$
H_{N_Q}\times N_{Q,X} \;{\buildrel\simeq \over  \longrightarrow }\;N_Q.
$$
\end{lemma}

This result is contained and proven in \cite[Prop. 2.16]{BB2014}.

\begin{lemma}
\label{l: inclusion NQ X}
Let $Q,P \in \parabsA$ be such that $P \succeq Q.$ Let $X \in \faq$ be such that (a) and (b) are valid.
Then
$$
N_{Q,X} \subseteq N_{P,X}.
$$
\end{lemma}

\begin{proof}
Let $\ga \in \gS(Q)$ be such that $\ga(X) > 0.$ Then it suffices to show that $\ga \in \gS(P).$
Assume this were not the case. Then either $-\ga \in \gS(P, \gs),$ or $-\ga \in \gS(P, \gs\Cartan).$ In the first case it would follow
that $-\ga \in \gS(Q,\gs),$ which contradicts the assumption that $\ga \in \gS(Q).$ In the second it would
follow that $-\ga(X) > 0$ which contradicts the assumption that $\ga(X) > 0.$
\end{proof}

\begin{lemma}\label{l: H_NP bs H_NQ simeq (NQ cap NP)bs NQ}
The inclusion map $H_{N_Q}\to N_Q$ induces a diffeomorphism
$$
\gf: H_{N_P}\bs H_{N_Q}  \;\; {\buildrel \simeq \over \longrightarrow} \;\; (N_Q \cap N_P) \bs N_Q.
$$
\end{lemma}

\begin{proof}
It follows from Lemma \ref{l: iso NQ mod H} that the natural map $H_{N_Q}\to N_{Q,X} \bs N_Q$ is a diffeomorphism
onto. By application of Lemma \ref{l: inclusion NQ X} it now follows that the natural map
$$
p: H_{N_Q}\to (N_{Q} \cap N_P) \bs N_Q
$$
is a surjective submersion. The map $p$ intertwines the natural $H_{N_Q}$-actions, and
the fiber of $[e]$ equals $H_{N_Q}\cap N_P = H_{N_P}.$ Thus, $\gf$ is induced by $p$
and is a diffeomorphism onto.
\end{proof}

\begin{lemma}
\label{l: psi and nilpotent qutotient}
The inclusion map $N_Q \cap \bar N_P \to N_Q$ induces a diffeomorphism
$$
\psi: N_Q\cap \bar N_P  \buildrel \simeq
\over \longrightarrow
(N_Q \cap N_P) \bs N_Q.
$$
\end{lemma}

\begin{proof} This is well known.
\end{proof}

\begin{lemma}
\label{l: pull back of omega by psi}
Let $\phi$ and $\psi$ be as in Lemma \ref{l: psi and nilpotent qutotient} and Lemma \ref{l: H_NP bs H_NQ simeq (NQ cap NP)bs NQ}. The map $\Phi:= \gf^{-1} \after \psi$ is a diffeomorphism from $N_Q \cap \bar N_P$ onto
$H_{N_P}\bs H_{N_Q}.$ Moreover, let $\omega$ be a positive $H_{N_Q}$-invariant
density on the image manifold. Then $\Phi^*(\omega)$ is a choice of Haar measure on $N_Q \cap \bar N_P.$
 \end{lemma}

\begin{proof}
Being the composition of two diffeomorphisms, $\Phi$ is a diffeomorphism.
We note that $\Phi^*(\omega) = \psi^* \gf^{-1*} (\omega).$
Let $dn$ be a choice of positive $N_Q$-invariant density on $(N_Q\cap N_P)\bs N_Q.$
Since $\gf$ is $H_{N_Q}$ intertwining, it follows that $\gf^*(dn)$ is
a positive $H_{N_Q}$-invariant density on $H_{N_P}\bs H_{N_Q}.$ By uniqueness of
positive invariant densities up to positive scalars, it follows that $\gf^*(dn) = c \omega$
for some $c>0$, so that also $\gf^{-1*}(\omega) = c^{-1} dn.$ By equivariance, it follows that $\psi^*(dn)$ is a choice
of Haar measure on $N_Q \cap \bar N_P.$ Thus, $\Phi^*(\omega) = c^{-1} \psi^*(dn)$
is as required.
\end{proof}

In view of this lemma we may fix invariant measures $d \bar n$
on $N_Q \cap \bar N_P$ and $dh$ on $H_{N_P}\bs H_{N_Q}$
such that $\Phi^*(dh) = d \bar n.$

\begin{lemma}
\label{l: nilpotent change of variables}
Let $f: G \to \C$ be a left $N_P$-invariant measurable function.
Then the following statements are equivalent.
\begin{enumerate}
\itema
$f$ is absolutely integrable over $H_{N_P}\bs H_{N_Q}. $
\vspace{-5pt}
\itemb
$f$ is absolutely integrable over $N_Q \cap \bar N_P.$
\end{enumerate}
If any of these statements hold, then with invariant measures normalized
as above,
$$
\int_{H_{N_P}\bs H_{N_Q}} f(h)\, d h
=
\int_{N_Q \cap \bar N_P} f(\bar n) \, d\bar n.
$$
\end{lemma}

\begin{proof}
As $\Phi^*(dh) = d\bar n$  it suffices to show
that $\Phi^*(f|_{H_{N_Q}}) = f|_{N_Q \cap \bar N_P}.$
Since $f$ is left $N_P$-invariant, this follows from the obvious fact that
for $\bar n \in N_Q\cap \bar N_P$ the canonical images of
$\bar n$ and $\Phi(\bar n)$ in $N_P \bs G$ coincide.
\end{proof}

Fix $P,Q \in \parabsA$ and assume that $P \succeq Q.$ Then $H_Q = H\cap Q$ contains $H_P = H \cap P.$ We note that
$H_P \simeq H_{M} \Ah H_{N_P}$ and that $H_Q$ admits a similar decomposition.

We shall now apply the results in the Appendix with $H, H_Q$ and $H_P$ in place of $G, H$ and $L,$ respectively.

Let $\omega_{H_P\bs H} \in \cD_{\fh/\fh_P}, \,\omega_{H_Q\bs H} \in \cD_{\fh/\fh_Q}$
and
$\omega_{H_P\bs H_Q} \in \cD_{\fh_Q/\fh_P}$
be such that
$\omega_{H_P\bs H} = \omega_{H_P\bs H_Q} \otimes \omega_{H_Q\bs H}$
in accordance with the identification
$\cD_{\fh/\fh_P} = \cD_{\fh_Q/\fh_P}  \otimes \cD_{\fh/\fh_Q}$
induced by the natural short exact sequence
$$
0 \longrightarrow  \fh_Q/\fh_P  \longrightarrow  \fh/\fh_P  \longrightarrow  \fh/\fh_Q \longrightarrow 0.
$$
See (\ref{e: short exact sequence}) and Lemma \ref{l: iso independent of splitting} for details.
We observe that
$$
H_P \bs H_Q \simeq H_{N_P} \bs H_{N_Q}
$$
naturally. Using the associated natural isomorphism of the tangent spaces
at the origins, we view $\omega_{H_P\bs H_Q}$ as a density on the quotient
$(\fh \cap \fn_Q) /(\fh \cap \fn_P).$ By unimodularity of the groups $H_{N_Q}$ and $H_{N_P},$ it follows that
$$
dn: n \mapsto dr_n([e])^{-1*}\omega_{H_P\bs H_Q}
$$
defines a choice of right $H_{N_Q}$-invariant density on
$H_{N_P} \bs H_{N_Q}.$ We define the character $\gD_{H_P\bs H}$ of $H_P$
as in Appendix, Equation (\ref{e: defi Delta sub}) with $H$ and $H_P$ in place of $G$ and $L,$
respectively.
Likewise, the space $\cM(H: H_P : \Delta_{H_{P}\bs H})$
is defined as in the text subsequent to
(\ref{e: defi Delta sub}).

\begin{thm}
\label{t: fiber integral J}
Let $f \in \cM(H: H_P : \Delta_{H_{P}\bs H})$ and let $f_P:= f_{\omega_{H_P\bs H}}$ be the associated measurable density on $H_P\bs H.$ Then the following assertions (a) and (b) are equivalent.
\begin{enumerate}
\itema
The density $f_P$ is absolutely integrable.
\vspace{-5pt}
\itemb
There exists a left $H_Q$-invariant set $\cZ$ of measure zero in $H$ such that
\begin{enumerate}
\item[{\rm (1)}]
for every $x \in H\setminus \cZ ,$ the integral
$$
A_x(f): = \int_{H_{N_P}\bs H_{N_Q}} f(nx)\, dn
$$
is  absolutely convergent;
\item[{\rm (2)}]
the function $A(f): x \mapsto A_x(f)$ belongs to  $\cM(H:H_Q:\Delta_{H_Q\bs H}).$
\item[{\rm (3)}]
the density $A(f)_{Q}:= A(f)_{\omega_{H_Q\bs H}}$ is absolutely integrable.
\end{enumerate}
\end{enumerate}
If any of the conditions (a) and (b) is fulfilled, then
$$
\int_{H_P \bs H} f_{P} = \int_{H_Q \bs H} A(f)_{Q}.
$$
\end{thm}

\begin{proof}
We will use the notation introduced in the text before the theorem.
The inclusion map $H_{N_Q} \to H_Q$ induces a diffeomorphism
$$
\phi: H_{N_P} \bs H_{N_Q} \to H_P \bs H_Q.
$$
Fix $x\in H$ and let
$f_{P,x}$  be the density on $H_P\bs H_Q$ given by
$$
f_{P,x} (H_P h) = \Delta_{H_Q\bs H}(h)^{-1} f(hx)dr_h([e])^{-1*}\omega_{H_P\bs H_Q}.
$$
By nilpotence, $\Delta_{H_Q \bs H}(n) = 1$ for $n \in H_{N_Q}.$ It follows
that
$$
\phi^*(f_{P, x})(H_{N_P}\cdot n) = f(n) dr_n([e])^{-1*} \omega_{H_P\bs H_Q} =
f(n) dn.
$$
In accordance with the notation of Theorem \ref{t: fubini for group fibering} we denote the integral of $f_{P,x}$ over $H_P\bs H_Q$ by $I_x(f).$ Then it follows by invariance of integration that the integral for $I_x(f)$ converges absolutely if and only if the integral $A_x(f)$ converges absolutely, while in case of convergence,
$$
I_x(f) = \int_{H_{N_P}\bs H_{N_Q}} \phi^*(f_{P,e})
= A_x(f).
$$
All assertions now follow by application of Theorem \ref{t: fubini for group fibering}.
\end{proof}

\section{H-fixed distribution vectors, the general case}
\label{s: H fixed distribution vectors, general case}
Recall the definition of $\Sigma(P)_{-}$ in (\ref{e: def minus roots}).

\begin{thm}
\label{t: convergence j P and intertwiner}
Let $P \in \cP_\gs(A),$ let $\xi$ be a finite dimensional unitary representation of $M_0$ and $\eta \in V(\xi,e).$ Assume that $\gl \in \faqdc$ satisfies
\begin{equation}
\label{e: condition gl for j H P}
\inp{\Re \gl + \rho_P- \rhoPh}{\ga} \leq  0,\qquad \mbox{\rm for all}\;\; \ga \in \gS(P)_-.
\end{equation}
Furthermore, let $f \in C^\infty(P: \xi_M : - \bar \gl + \rhoPh). $ Then
$$
j_H(P\col\xi_{M} \col \gl\col \eta)f =  \int_{H_P\bs H} \inp{\eta}{f(h)}\,\, dr_h([e])^{-1*}\omega_{H_P\bs H},
$$
with absolutely convergent integral.

Let  $Q \in \parabsA$ be a second parabolic subgroup,
with $P\succeq Q.$ Then for all $x \in G,$
$$
A(Q:P:\xi_{M}: -\bar \gl + \rhoPh)f (x) = \int_{N_Q \cap \bar N_P} f(n x) \; dn
$$
with absolutely convergent integral. Finally,
\begin{eqnarray}\nonumber
\lefteqn{
j_H(P\col\xi_M \col \gl\col \eta)f
}\\
&=& \int_{H_Q\bs H} \inp{\eta}{[A(Q:P: - \bar \gl+\rhoPh)f](h) }\; dr_h([e])^{-1*} \omega_{H_Q\bs H},
\label{e: j H P and intertwiner}
\end{eqnarray}
with absolutely convergent integral.
\end{thm}

\begin{proof}
Observe that the function $f$ restricted to $H_P\bs H$
belongs to $C^\infty(H: H_P : \Delta_{H_{P}\bs H}).$
The first assertion now follows from Proposition \ref{p: convergence density on HP H} and Equation (\ref{e: defi J H P}).

We will now apply Theorem \ref{t: fiber integral J}.
For $x \in H$ the fiber integral takes the form 
$$
A_x(f|_H) = \int_{H_{N_P} \bs H_{N_Q}} f(nx) \; dn,
$$
which by Lemma \ref{l: nilpotent change of variables} equals
$$
\int_{N_Q \cap \bar N_P} f(\bar n x) \; d \bar n.
$$
The latter is just the integral for the standard intertwining operator
$A(Q:P:\xi_{M}: -\bar \gl + \rhoPh)$ (up to suitable normalization). This integral is known to converge absolutely in case
\begin{equation}
\label{e: condition convergence A Q P}
\Re\inp{-\gl + \rhoPh}{\ga} > 0,\quad  \forall \ga \in \gS(P) \cap \gS(\bar Q).
\end{equation}
If $\ga \in \gS(P)\cap \gS(\bar Q),$ then $\ga \in \gS(P)\setminus \gS(Q)$
so that $\ga \notin \faqd$ and $\alpha\notin\gS(P,\sigma)$ from which we conclude that
$\gl \in \gS(P, \gs \Cartan )\setminus\faqd  \subseteq \gS(P)_-.$ It then follows from (\ref{e: condition gl for j H P}) that
$$
\Re \inp{-\gl +\rhoPh}{\ga} >  \Re\inp{- \gl + \rhoPh - \rho_P }{\ga} \geq 0
$$
and we see that (\ref{e: condition convergence A Q P}) is satisfied. This implies the second assertion. The final assertion now follows by application of
Theorem \ref{t: fiber integral J}.
\end{proof}

In the following we will need to use the $K$-fixed
function in the induced representation $\Ind_Q^G(1\otimes \mu \otimes 1),$ for $Q \in \cP(A)$ and $\mu \in \fadc.$ More precisely, given such $Q$ and $\mu$ we define the function $\one_{Q, \mu}: G \to \C$ by
$$
\one_{Q,\mu}(nak) := a^{\mu + \rho_Q},\qquad (k \in K, \,a \in A,\, n \in N_Q).
$$
Thus, $\one_{Q,\mu}$ is the unique function in $C^\infty(Q: 1: \mu)$ satisfying
$\one_{Q, \mu}|_K =1.$

\begin{cor}
\label{c: convergence one Q rhoP}
Let $Q \in \cP(A), $
$P \in \cP_\gs(A)$ and assume that  $P\succeq Q.$ Then
\begin{equation}
\label{e: density one Q rhoP}
h \mapsto  \one_{Q, \rho_P}(h) \; dr_h([e])^{-1*} \omega_{H_Q\bs H}
\end{equation}
defines a density on $H_Q\bs H$ which is absolutely integrable.
\end{cor}

\begin{proof}
We apply Theorem \ref{t: convergence j P and intertwiner} with $\xi = 1,$ $\cH_\xi = \C$ and $\eta =1.$
Furthermore, we take $\gl = - \rho_P + \rhoPh \in \faqd$ so that $-\bar \gl + \rhoPh = \rhoP$, and we take
$
f = \one_{P, \rho_P}.
$
It follows from the mentioned theorem that
the integral for
$
A(Q:P: 1: \rhoP)f
$
converges absolutely. By equivariance, it gives a $K$-fixed element of $C^\infty(Q:1:\rhoP),$
so that
$$
A(Q:P: 1: \rhoP)f  = A(Q:P: 1: \rhoP)\one_{P, \rhoP} = c(Q:P:\rho_P) \one_{Q, \rhoP},
$$
for some constant $c(Q:P:\rho_P)\in\C$. Evaluating this identity in the unit element we find
$$
c(Q:P:\rho_P) = \int_{N_Q \cap \bar N_P} \one_{P, \rho_P}(\bar n) \; d\bar n,
$$
which of course is the integral representation of a partial $c$-function. As the integrand is
everywhere positive, it follows that $c(Q:P:\rho_P)$ is a positive real number.
It now follows from the final assertion of Theorem \ref{t: convergence j P and intertwiner}
that
$$
h \mapsto
 c(Q:P:\rho_P)\cdot \one_{Q, \rhoP}(h)  \; dr_h([e])^{-1*} \omega_{H_Q\bs H}
$$
defines a density on $H_Q\bs H$ which is absolutely integrable. By positivity of
$c(Q:P:\rhoP)$ all assertions now follow.
\end{proof}

Let $\Gamma(Q)$ denote the cone in $\faq$ spanned by the
elements $H_\ga + \gs \Cartan H_\ga,$ for $\ga \in \gS(Q)_-,$ where the latter set is
defined as in (\ref{e: def minus roots}). The (closed) dual
cone in $\faqd$ is readily seen to be given by
\begin{equation}
\label{e: defi dual cone}
\Gamma(Q)^\circ := \{ \gl \in \faqd  \mid  \inp{\gl}{\ga} \geq 0, \;\; \forall \ga \in  \gS(Q)_-\}.
\end{equation}

\begin{lemma}
\label{l: estimate one Q mu}
Let $Q \in \cP(A)$. Let $\mu \in  \Gamma(Q)^\circ.$
Then
\begin{equation}
\label{e: estimate one Q mu}
0< \one_{Q,\mu - \rho_Q}(h) \leq 1, \qquad (h \in H).
\end{equation}
\end{lemma}

\begin{proof}
It follows from \cite[Thm 10.1]{BB2014} that if $h=nak$ with $n\in N_{Q}$, $a\in A$ and $k\in K$, then
$$
\prq \log a \in -\Gamma(Q),
$$
where $\prq$ denotes the projection $\fa\to\faq$. Therefore
$$
\one_{Q,\mu - \rho_Q}(h)
=a^{\mu}
=e^{\mu(\prq \log a)}
\leq 1.
$$
This establishes the upper bound. The lower bound is trivial.
\end{proof}

The above result will play a crucial role in the proof of a domination expressed in
the following lemma.

\begin{lemma}
\label{l: domination of f by oneQ}
Let $Q \in \cP(A)$ and let $\xi$ be a finite dimensional unitary representation
of $M_0$.
Let $P \in \cP_\gs(A)$ and assume that  $P\succeq Q;$ thus, in particular, $\rhoPh = \rhoQh$. Furthermore, assume that  $\gl \in \faqdc$ satisfies
\begin{equation}
\label{e: condition for domination}
\Re \inp{\gl + \rho_P - \rhoPh}{\ga} \leq 0 \qquad\mbox{{\rm for all}} \;\; \ga \in \gS(Q)_-.
\end{equation}
Then for every $f \in C^\infty(Q:\xi_{M}: -\bar \gl + \rhoQh),$ we have
\begin{equation}
\label{e: domination of f by oneQ}
\|f(h)\|_{\xi} \leq  \sup_{k \in K}\|f(k)\|_\xi \cdot \one_{Q,\rho_P}(h),\qquad (h \in H).
\end{equation}
\end{lemma}

\begin{proof}
Since $P \succeq Q,$ we have $\rhoPh = \rhoQh.$
Thus, if  $k \in K$ and $u \in Q$ then
$$
f(uk) = \one_{Q, -\bar \gl + \rhoPh}(uk)\, f(k) .
$$
It follows that
\begin{equation}
\label{e: first estimate f}
\|f(x)\|_\xi \leq \sup_{k \in K} \|f(k)\|_\xi  \cdot \one_{Q, \mu + \rhoP}(x),\qquad (x \in G),
\end{equation}
where $\mu = -\Re \gl - \rhoP + \rhoPh.$
For $x \in G$ we have
$$
\one_{Q,  \mu + \rhoP}(x) = \one_{Q, \mu - \rho_Q}(x) \one_{Q, \rhoP}(x).
$$
As $\mu \in \Gamma(Q)^\circ$ by (\ref{e: condition for domination}), it follows by application of Lemma \ref{l: estimate one Q mu} that
\begin{equation}
\label{e: second estimate f}
\one_{Q, \mu + \rhoP }(h) \leq   \one_{Q,\rhoP}(h), \qquad (h \in H).
\end{equation}
The required estimate (\ref{e: domination of f by oneQ})
follows from combining (\ref{e: first estimate f}) and
(\ref{e: second estimate f}).
\end{proof}

For the formulation of the next result, we note that the set of $\gl \in \faqdc$ satisfying condition
(\ref{e: condition for domination})
is given by
\begin{equation}\label{e: defi A_PQ}
\holset_{P,Q} := -(\rho_P - \rhoPh) - \Gamma(Q)^\circ + i\faqdc.
\end{equation}

\begin{cor}
\label{c: first convergence j H Q}
Let $Q \in \cP(A),$ $\xi$ a finite dimensional unitary representation
of $M_0$ and $\eta \in V(\xi,e).$
Let $P \in \cP_\gs(A)$ and assume that  $P\succeq Q;$ thus, in particular, $\rhoPh = \rhoQh$.
Let $\lambda\in\holset_{P,Q}$.
Then for every $f \in C(Q:\xi_M: -\bar \gl + \rhoQh),$ the integral
\begin{equation}
\label{e: integral for j H Q}
j_H(Q : \xi_M: \gl :\eta) f : = \int_{H_Q\bs H} \inp{\eta}{f(h)}\; dr_h([e])^{-1*} \omega_{H_Q\bs H}
\end{equation}
converges absolutely.
\end{cor}

\begin{proof}
It follows by application of Lemma \ref{l: domination of f by oneQ} that
\begin{equation}
\label{e: estimate eta f on oneQ}
| \inp{\eta}{f(h)}| \leq \|\eta\|_\xi \sup_{k \in K}\|f(k)\|_\xi \cdot
\one_{Q, \rhoP}(h), \qquad (h \in H).
\end{equation}
The result now follows
from Corollary \ref{c: convergence one Q rhoP}.
\end{proof}

Working in the setting of the above corollary,
if $f \in C(K:M: \xi_M),$ then for $\mu \in \fadc$ we define
$f_\mu  \in C(Q:\xi_M: \mu)$ by $f_\mu|_K =f.$ Furthermore, we define
$$
j_H(Q: \xi_M : \gl: \eta) (f) :=  j_H(Q: \xi_M : \gl: \eta) (f_{-\bar \gl + \rhoQh})
$$
for $\gl \in\holset_{P,Q}$.
Accordingly, $j_H(Q:\xi_M : \gl : \eta)$ is viewed as an element  
of $C^{-0}(K:M:\xi_M),$
  see the beginning of Section \ref{s: induced reps and densities}.
Given $f \in C(K:M:\xi_M),$ we agree to write
$$
\inp{j_H(Q:\xi_M:\gl:\eta)}{f} = {j_H(Q:\xi_M:\gl:\eta)}(f)
$$
and $\inp{f}{j_H(Q:\xi_M:\gl:\eta)} $ for its conjugate. Then
$$
\inp{f}{j_H(Q:\xi_M:\bar \gl:\eta)}
= \int_{H_Q\bs H} \inp{f_{- \gl + \rhoQh}(h)}{\eta} \; dr_h([e])^{-1*} \omega_{H_Q\bs H}.
$$

\begin{cor}
\label{c: holomorphy of j H Q}
Let notation be as in Corollary \ref{c: first convergence j H Q}.
Then
\begin{equation}
\label{e: the function j H Q}
\gl \mapsto  j_H(Q: \xi_M: \gl:\eta)
\end{equation}
is a continuous $C^{-0}(K:M:\xi_M)$-valued function on the closed subset
$\holset_{P,Q}$ of  $\faqdc.$
Its restriction to the interior of $\holset_{P,Q}$ is
holomorphic as a $C^{-0}(K:M:\xi_M)$-valued function.
\end{cor}

\begin{proof}
It is clear that $\gl \mapsto f_{-\bar \gl + \rhoPh}|_H$ is a holomorphic
$C(H)\otimes \cH_\xi$-valued function on $\faqdc$ satisfying the uniform estimate
$$
| \inp{\eta}{f_{-\bar \gl + \rhoPh}} \leq  \|\eta\|_\xi \sup_{k \in K}\|f(k)\|_\xi \cdot
\one_{Q, \rhoP}(h), \qquad (h \in H),
$$
for all $\gl \in \holset_{P,Q},$ by application of (\ref{e: estimate eta f on oneQ}).
In view of Corollary \ref{c: convergence one Q rhoP} the result now
follows by application of the dominated convergence theorem.
\end{proof}

The following lemma will be useful for later use.
If $Q \in \cP(A),$  we have that $\gS(Q, \gs\Cartan)|_{\faq} \subseteq \gS(\faq). $
In accordance with (\ref{e: defi faqpQ}) we define
$$
\faqdp(Q): = \{\gl \in \faqd\mid \inp{\gl}{\ga} > 0, \;\;  \forall \ga \in \gS(Q, \gs\Cartan)\}.
$$
This set is a non-empty open subset of $\faqd,$ see the
text below (\ref{e: defi faqpQ}).

\begin{lemma}
\label{l: subset of holset PQ}
Let $Q \in \cP(A)$ and $P \in \cP_\gs(A,Q).$
Then
$$
\holset_{P,Q} \supset -(\rho_P - \rhoPh) -\faqdp(Q) + i\faqd.
$$
\end{lemma}
\begin{proof}
In view of (\ref{e: defi A_PQ}) it suffices to show that
$\Gamma(Q)^\circ \supset \faqdp(Q).$ This is a straightforward consequence of the fact that $\gS(Q)_{-} \subseteq \gS(Q, \gs \Cartan),$
by (\ref{e: def minus roots}) and (\ref{e: defi ordering on minparabs}).
\end{proof}

\begin{thm}
\label{t: meromorphy of j H Q}
Let $Q \in \cP(A),$ $P \in \parabs_\gs(A)$ such that $P\succeq Q.$ Let $\xi $ be a finite dimensional unitary representation of $M_0$ and $\eta \in V(\xi,e)$.
Then the $\Cminf(K:M:\xi_M)$-valued function
\begin{equation}
\label{e: j meromorphic}
\gl \mapsto j_H(Q\col \xi_M \col \gl\col \eta),
 \end{equation}
defined by (\ref{e: integral for j H Q}),
extends to a meromorphic function on $\faqdc$ with values in $\Cminf(K:M:\xi_M).$
Furthermore, up to a positive factor,
depending on the normalization of the Haar measure on
$N_{P} \cap \bar N_{Q},$
\begin{equation}
\label{e: j H P and j H Q}
j_H(P\col \xi_M \col \gl\col \eta)
= A(P\col Q\col \xi_M \col \gl - \rho_{Qh}) j_H(Q\col \xi_M \col \gl\col \eta)
\end{equation}
as an identity of $\Cminf(K:M: \xi_M)$-valued meromorphic functions
in $\gl \in \faqdc.$
Finally, the function (\ref{e: j meromorphic})
is continuous on the set $\holset_{P,Q}$ defined in (\ref{e: defi A_PQ})
and holomorphic on its interior.
\end{thm}

\begin{rem}
\label{r: holomorphy j in fh extreme case}
In particular, if $\gS(Q)_- = \emptyset,$ it follows that $\Gamma(Q)^\circ = \faqd$ so that $j_H(Q\col \xi_M \col \dotvar)$ is holomorphic everywhere.
\end{rem}

\begin{proof}
Without loss of generality, we may assume that $\xi$ is irreducible.
Then it follows from (\ref{e: j H P and intertwiner}) combined with (\ref{e: integral for j H Q}) that
\begin{equation}
\label{e: j P and A j Q}
\inp{j_H(P\col \xi_M \col \gl\col \eta)}{f} =
\inp{j_H(Q\col \xi_M \col \gl\col \eta) }{A(Q\col P\col \xi_M \col - \bar \gl + \rho_{Ph})f }
\end{equation}
for all $\gl \in \holset_{P,P}$ and $f \in C^\infty(P:\xi_M:-\bar \gl + \rhoPh).$

The standard intertwining operator $A(Q:P:\xi_M:\nu)$
from the induced representation $\Ind_{P}^G(\xi_M \otimes \nu\otimes 1)$ to the representation $\Ind_Q^G(\xi_M\otimes \nu \otimes 1)$
may be viewed as a meromorphic function of $\nu \in \fadc,$
with values in the space $\End(C^\infty(K:M:\xi_M)) $ (equipped with the strong topology),
see \cite[Thm.\ 1.5]{VogWal} and \cite[Thm.\ 1.5]{BScfunc}.
Its singular locus is contained in a locally finite union of hyperplanes
of the form $\mu + \ker \ga,$ with  $\mu \in \fad$ and $\ga\in \gS(P) \cap \gS(\bar Q),$
see \cite[Rem.\ 1.6]{BScfunc}.
Since $\gS(P)\cap \gS(\bar Q) \cap \fahd = \emptyset$ in view of  Lemma
\ref{l: intersection with faqd} (b), none of these singular hyperplanes
contain $\faqdc,$ so that $A(Q:P:\xi_M:\dotvar)$ restricts to a meromorphic function
on $\faqdc.$

The operator $A(P: Q : \xi_M :\nu)$ has a similar meromorphic behavior,
and since the induced representation $\Ind_Q^G(\xi_M \otimes \nu \otimes 1)$ is irreducible for generic $\nu \in \faqdc$ it follows that
\begin{equation}
\label{e: product of intertwiners}
A(Q: P: \xi_M : \nu) \after A(P:Q:\xi_M :\nu) = \eta(P:Q:\xi_M: \nu)\, \rmI
\end{equation}
as an identity of $\End(C^\infty(K:M: \xi_M))$-valued functions of $\nu \in \fadc.$ Here
$\eta = \eta(P:Q:\xi_M:\dotvar)$ is a meromorphic $\C$-valued function on $\fadc.$
By the usual product decomposition of intertwining operators it follows that
$\eta$ admits a decomposition of the form
$$
\eta(\nu) = \prod_{\ga \in \gS(P)\cap \gS(\bar Q)} \eta_\ga(\inp{\nu}{\ga}),
$$
where the $\eta_\ga$ are meromorphic functions on $\C.$
We now fix $g \in C^\infty(K:M:\xi_M).$
By substituting $f = A(P:Q:\xi_M:-\bar\gl + \rhoQh)g$ in (\ref{e: j P and A j Q}) we infer that
\begin{eqnarray*}
\lefteqn{\inp{j_H(P\col \xi_M \col \gl\col \eta)}{A(P\col Q\col \xi_M \col - \bar \gl + \rho_{Qh})g}
=}\qquad \qquad\qquad\qquad\qquad \\
\qquad\qquad&=&
\inp{j_H(Q\col \xi_M \col \gl\col \eta)}{\eta(-\bar \gl + \rhoQh) g}.
\end{eqnarray*}
By using that $A(Q:P:\xi_M : \gl - \rhoQh)$ is the Hermitian conjugate of
$A(P\col Q\col \xi_M \col - \bar \gl + \rho_{Qh})$, see \cite[Prop.\ 7.1 (iv)]{KSII},
and that $\rhoQh = \rhoPh,$ it follows that
\begin{equation}
\label{e: j H Q expressed in j H P}
j_H(Q:\xi_M: \gl : \eta)
=\overline{\eta(-\bar \gl + \rhoPh)}^{-1} A(Q:P:\xi_M : \gl - \rhoQh)j_H(P:\xi_M : \gl:\eta),
\end{equation}
for generic  $\gl \in \holset_{P,P}.$

Let $\Omega \subseteq \faqd$ be a relatively
compact open subset.
Then there exists a constant $s \in \N$ such that
$\gl \mapsto j_H(P:\xi_M : \gl: \eta)$ is meromorphic on
$\Omega+i\faqd,$ with values in the Banach space $C^{-s}(K:M: \xi_M),$ see  Section \ref{s: induced reps and densities} and
 \cite[Thm.\ 9.1]{Banps2} for details.
Furthermore, there exists a constant $r \in \N$ such that $A(Q:P:\xi_M:\gl - \rhoPh)$ depends meromorphically on
$\gl\in \Omega+i\faqd,$ as a function with values in the Banach space
of bounded linear maps from $C^{-s}(K:M:\xi_M)$ to $C^{-s-r}(K:M:\xi_M).$
Combining these observations with (\ref{e: j H Q expressed in j H P}) we see
that $\gl \mapsto j_H(Q:\xi_M :\gl:\eta)$ is a meromorphic function
on $\faqdc$ with values in $\Cminf(K:M:\xi_M),$ equipped with the strong
dual topology. Its continuity on $\holset_{P,Q}$ and holomorphy on the interior of this
set follows from Corollary \ref{c: holomorphy of j H Q}.
By meromorphic continuation it now follows that  (\ref{e: j P and A j Q}) is valid as an identity of meromorphic functions. Since $A(P:Q:\xi_M : \gl - \rhoQh)$ is the Hermitian conjugate of the intertwining
operator appearing in that identity, whereas the identity holds for all
$f \in C^\infty(K:M:\xi_M),$ it follows that (\ref{e: j H P and j H Q})
is valid as an identity of meromorphic $\Cminf(K:M:\xi_M)$-valued functions of $\gl \in \faqdc.$
\end{proof}

Let $Q \in \parabs$ be fixed for the moment.
Then the function (\ref{e: j meromorphic})
is independent of the choice of $P \succeq Q,$ whereas the description of the
domain of holomorphy depends on it. This motivates the definition
of the following closed subset of $\faqdc,$
\begin{equation}\label{e: defi Omega Q}
\holset_Q:= \bigcup_{P \in \parabs_\gs(A,Q)} \holset_{P,Q},
\end{equation}
where the union is taken over the finite non-empty set $\parabs_\gs(A,Q)$ of
parabolic groups $P \in \cP_\gs(A)$ with $P \succeq Q,$ see Lemma \ref{l: on cP gs A Q}.
The function (\ref{e: j meromorphic})
 is continuous on $\holset_Q$ and holomorphic on the interior of this set.
We can actually improve on this result.

In fact,  let $\Gamma(Q)^\circ$ be as in (\ref{e: defi dual cone}). We denote by $B: \faq \to \faqd$ the linear
isomorphism induced by the inner product on $\faq.$ Then for $\ga \in \gS(\faq)$
we have $B(H_\ga) = \ga^\vee.$ Therefore, $B(\Gamma(Q))$ is the cone
spanned by the $\faq$-roots from $\pr_\iq (\gS(Q)_-).$

Let $\widehat \holset_Q$ denote the hull in $\faqdc$ of the set
$\holset_Q$ with respect to the functions $\Re\, \inp{\dotvar}{\ga}$ with $\ga \in \gS(\faq)\cap B( \Gamma(Q))$, i.e.,
\begin{equation}
\label{e: defi hull cA Q}
\widehat \holset_Q := \{ \gl \in \faqdc \mid  \Re \inp{\gl}{\ga}\leq  \sup \,\Re \inp{\holset_Q}{\ga}, \;\;\forall \ga \in \gS(\faq)\cap B(\Gamma(Q)) \}.
\end{equation}
Since the roots from $\gS(\faq)\cap B(\Gamma(Q))$ satisfy
$\inp{\ga}{\dotvar} \leq 0$ on $-\Gamma(Q)^\circ$ it follows that
we can describe the given hull by means of inequalities as follows:
\begin{equation}
\label{e: hull by inequalities}
\widehat \holset_Q = \{ \gl \in \faqdc \mid \Re\inp{\gl}{\ga} \leq
\max_{P \in \cP_\gs(A,Q)}\inp{-\rho_P}{\ga},\;\; \forall \ga \in \gS(\faq) \cap B(\Gamma(Q))  \}.
\end{equation}

\begin{cor}
Let $Q \in \cP(A),$ $\xi \in \dMzerofu$ and $\eta \in V(\xi,e).$
Then the $\Cminf(K:M:\xi_M)$-valued function $\gl \mapsto j_H(Q\col \xi_M\col \gl\col\eta)$
is holomorphic on an open neighborhood of $\widehat \holset_Q.$
\end{cor}

\begin{proof}
From (\ref{e: j H Q expressed in j H P}) we infer that the singular locus of $\gl \mapsto j_H(Q:\xi_M:\gl:\eta)$ is the union of a  locally finite collection $\cH$  of hyperplanes of the form
${\rm H}_{\ga, \mu} = \mu + (\ga^\perp)_\iC$ with $\ga \in \gS(\faq)$ and $\mu \in \faqd.$
Indeed, the singular loci of the meromorphic ingredients on the right-hand side of that formula are all of this form, by \cite[Lemma 3.2]{BSfinv}, \cite[Rem.\ 1.6]{BScfunc} and (\ref{e: product of intertwiners}).

Let $\mu$ be a singular point of  $j_H(Q:\xi_M:\dotvar :\eta),$
i.e., a point in the union of the singular hyperplanes.
Then there exists a root $\ga \in \gS(\faq)$ such that
${\rm H}_{\ga, \mu}$ is a singular hyperplane. By analytic continuation it follows that
${\rm H}_{\ga, \mu} \cap \holset_Q = \emptyset.$
From the fact that the cone $\Gamma(Q)^{\circ}$ has non-empty interior it follows that
 the set $\holset_Q\cap \faqd$ is connected, hence
is contained in one connected component of $\faqd\setminus{\rm H}_{\ga, \mu}$. Replacing $\ga$ by $-\ga$ if necessary, we may assume that
$$
\holset_{Q}\cap\faqd\subseteq\{\lambda\in\faqd:\langle\alpha,\lambda\rangle\geq c\}
$$
for some $c\in\R$. This in turn implies that
$$
\Gamma(Q)^{\circ}
\subseteq\{\lambda\in\faqd:\langle\alpha,\lambda\rangle\geq0\}
=\{\lambda\in\faqd:\lambda(H_{\alpha})\geq0\}.
$$
Since $\Gamma(Q)^{\circ}$ has open interior, $\alpha$ does not vanish on $\Gamma(Q)^{\circ}$. Using that $\Gamma(Q)^{\circ}$ is a cone, we find
$$
\langle\alpha,\Gamma(Q)^{\circ}\rangle=\R_{\geq0}.
$$
In particular this implies $H_{\alpha}\in\Gamma(Q)^{\circ\circ}=\Gamma(Q)$ and thus we conclude  that $\ga \in \gS(\faq) \cap B(\Gamma(Q)).$

For any $P \in \parabs_\gs(A)$ with $P \succeq Q,$ the singular hyperplane ${\rm H}_{\ga,\mu}$
does not intersect $-\rho_P + \rhoPh - \Gamma(Q)^\circ$, hence
$$
\langle\ga,\mu\rangle \notin -\langle\ga,\rho_P\rangle-\R_{\geq0}.
$$
This implies that $\inp{\ga}{\mu} > -\inp{\ga}{\rho_P}.$
We conclude that
$$
\inp{\ga}{\mu} >\max_{P\in \parabs_\gs(A), P\succeq Q} \inp{\ga}{-\rho_P}
$$
so that $\mu \notin \widehat \holset_Q.$ Thus, $\widehat \holset_Q$ is disjoint
from the singular locus.
\end{proof}

Let $Q\in \parabsA$ and $\xi \in \dMzerofu.$ We define the meromorphic $V(\xi,e)^*\otimes \Cminf(K:M:\xi_M)$-valued function $j_H(Q:\xi_M :\dotvar)$ on $\faqdc$ by
$$
j_H(Q:\xi_M :\gl)(\eta) = j_H(Q: \xi_M :\gl: \eta)
$$
for generic $\gl \in \faqdc$ and $\eta \in V(\xi, e).$  Furthermore,
we define the meromorphic $V(\xi)^*\otimes \Cminf(K:M:\xi_M)$-valued function $j(Q:\xi_M : \dotvar)$ on $\faqdc$ by
\begin{equation}
\label{e: defi j without H}
j(Q: \xi_M : \gl) =
\sum_{v \in \cW} \pirep_{Q, \xi_M, \gl - \rhoQh}(v^{-1}) \, j_{v H v^{-1}}(Q : \xi_M : \gl)\after \pr_v.
\end{equation}
Here $j_{vHv^{-1}}$ is defined for the data $\gs_v, \,vHv^{-1}$ 
in place of $\gs, H.$ This definition is allowed since  $\cW \subseteq N_K(\fa) \cap N_K(\faq)$
(see text preceding (\ref{e: defi cW})),  so that $A, \, \Aq, \, M_0$ and $\cP(A)$ are invariant under conjugation by $v$ and $\faq$ is maximal abelian in 
$\fp \cap \Ad(v)\fq.$
See also the discussion at the end of Section 
\ref{s: minimal parabolics}.

In order to formulate our next result, 
we define, for $v \in \cW,$ the set $\Omega_{v, Q}$ as  $\Omega_Q$ in 
(\ref{e: defi Omega Q}), with $v H v^{-1}$
in place of $H.$ Likewise, we define $\widehat \Omega_{v,Q}$ to be the
set $\widehat \Omega_Q$ defined as in (\ref{e: defi hull cA Q}), with 
$v H v^{-1}$ in place of $H.$ 

\begin{lemma}
\label{l: holset v Q}
 Let $Q \in \cP(A).$ Then for each $v \in \cW,$ we have
 $$
 \Omega_{v,Q}  = v\Omega_{v^{-1} Q v}\quad{\rm and} \quad
\widehat  \Omega_{v,Q}  = v\widehat \Omega_{v^{-1} Q v}.
 $$
 \end{lemma}
 \begin{proof}
 In view of Lemma \ref{l: twisted sigma minus}
 the cone $\Gamma (v, Q),$ defined as $\Gamma(Q)$ with $\gs_v$ in place of $\gs,$ is given by
$
\Gamma(v,Q) = v \Gamma(v^{-1} Q v).
$
Likewise, its dual, defined as in (\ref{e: defi dual cone}) is given by
$$
\Gamma(v,Q)^\circ = v \Gamma(v^{-1} Q v)^\circ.
$$
From (\ref{e: defi A_PQ}) and  (\ref{e: defi Omega Q}), with $\gs_v$ in place of $\gs,$ we now find, with obvious notation, $\holset_{v, P, Q} = v \holset_{v^{-1}Pv, v^{-1}Qv},$ for $P\in \cP_{\gs}(A)$ with $P \succeq Q.$ Taking the union over such $P,$ we obtain the first asserted equality.

The second equality follows from the first, by taking the hull of the sets
$\holset_{v,Q}$ and $v\holset_{v^{-1}Qv}$ with respect to the functions $\Re\inp{\dotvar}{\ga}$ with $\ga \in \gS(\faq) \cap B(\Gamma(v,Q)).$ The first hull equals $\widehat \holset_{v,Q}$ by definition. Using that
$$
\gS(\faq) \cap B(\Gamma(v,Q)) =
\gS(\faq) \cap B(v\Gamma(v^{-1}Q v))
=
v(\gS(\faq) \cap B(\Gamma(v^{-1}Qv)),
$$
we see that the second hull equals $v \widehat \Omega_{v^{-1}Qv}.$
\end{proof}

We define the following closed subsets of $\faqdc,$
\begin{equation}
\label{e: defi Upsilon}
\Upsilon_Q = \bigcap_{v \in \cW} v \Omega_{v^{-1}Qv},\qquad \hatUpsilon_Q =  \bigcap_{v \in \cW} v \widehat\Omega_{v^{-1}Qv}.
\end{equation}

The following lemma guarantees in particular that the set $\Upsilon_Q,$ and hence also the bigger set $\hatUpsilon_Q,$ have non-empty interior.

 \begin{lemma}
 Let $Q \in \cP(A).$ Then for every $P \in \cP_\gs(A,Q),$ we have
 $$
 \Upsilon_Q \;\; \supset \;\;  -(\rhoP- \rhoPh) - \faqdp(Q) + i\faqd.
 $$
 \end{lemma}
\begin{proof}
Fix $v \in \cW.$ Then $v^{-1}Pv$ belongs to $ \cP_\gs(A, v^{-1}Q v),$ hence it follows from (\ref{e: defi Omega Q}) and
Lemma \ref{l: subset of holset PQ}
$$
\Omega_{v^{-1} Q v} \supset -  ( \rho_{v^{-1}Pv}  -\rho_{v^{-1}Pv\ih}) -\faqdp(v^{-1}Qv) + i\faqd.
$$
Applying $v$ we obtain
$
v \Omega_{v^{-1} Q v} \supset - (\rho_P - \rhoPh) - \faqdp(Q) + i \faqd.
$
As this is true for each $v \in \cW,$ the asserted inclusion follows.
\end{proof}

\begin{lemma}
\label{l: holomorphy of j}
Let $Q \in \cP(A)$ and $\xi \in \dMzerofu.$ Let $\eta \in V(\xi).$
\begin{enumerate}
\itema
For each $v\in \cW$ the defining integral for the corresponding term in
(\ref{e: defi j without H}) is absolutely convergent for every $\gl \in \convergeset_Q.$\vspace{-5pt}
\itemb
The meromorphic $C^{-\infty}(K:M:\xi_M)$-valued function
$\gl \mapsto j(Q:\xi_M:\gl:\eta)$ is holomorphic on an open neighborhood of the set $\hatUpsilon_Q.$
\end{enumerate}
\end{lemma}

\begin{proof}
It follows from (\ref{e: defi j without H}) and
Corollary \ref{c: first convergence j H Q} that the integral for $j_{vHv^{-1}}(Q:\xi_M: \gl  :\eta_v)$ is convergent for $\gl \in \holset_{v,Q}.$ This set contains $\convergeset_Q,$ by (\ref{e: defi Upsilon}) and Lemma \ref{l: holset v Q}, and we see that (a) follows.

It follows from Corollary \ref{c: holomorphy of j H Q}
applied with $\gs_v$ in place of $\gs$ that the mentioned function is holomorphic on an open neighborhood of ${\widehat\holset}_{v, Q}.$ From this we deduce that $j(Q:\xi_M\dotvar:\eta)$ is holomorphic on an
open neighborhood of the intersection of the sets $
\widehat \holset_{v,Q},$
for $v \in \cW.$   This intersection equals $\hatUpsilon_Q,$ by
(\ref{e: defi Upsilon}) and Lemma \ref{l: holset v Q}.
\end{proof}

We finish this section relating the constructed functions
$j(Q:\xi_M:\cdot),$ for different $Q,$ by intertwining operators.
\begin{thm}
Let $Q \in \parabsA$ and $\xi \in \dMzerofu.$ Then the following assertions are valid.
\begin{enumerate}
\itema
For every $\eta \in V(\xi)$ and generic $\gl \in \faqdc,$ the element
$j(Q:\xi_M:\gl)(\eta) $ of the space $\Cminf(K:M: \xi_M)$ is
$\pirep_{Q, \xi_M, \gl - \rhoQh}(H)$-invariant.\vspace{-5pt}
\itemb
If $Q,Q' \in \parabsA$ and $Q'\succeq Q,$ then (up to normalization),
\begin{equation}
\label{e: j Qprime and j Q}
j(Q': \xi_M : \gl) = A(Q':Q: \xi_M: \gl - \rho_{Qh}) \after  j(Q: \xi_M : \gl),
\end{equation}
as an identity of meromorphic
$V(\xi)^* \otimes \Cminf(K:M:\xi_M)$-valued functions in the variable $\gl\in \faqdc.$
\end{enumerate}
\end{thm}

\begin{proof}
We start with (b). Let $P \in \parabsgsA$ be such that $P \succeq Q'.$
Then by application of Lemma \ref{l: comparison orderings parabs}
it follows that $\gS(P) \cap \gS(\bar Q') \subseteq \gS(P) \cap \gS(\bar Q)$
so
\begin{equation}
\label{e: composition of intertwiners}
A(P:Q:\xi_M:\gl) = A(P:Q':\xi_M:\gl) \after A(Q':Q:\xi_M:\gl)
\end{equation}
as a meromorphic identity in $\gl\in\faqdc.$ See \cite[Cor.\ 7.7]{KSII} for details.
Using (\ref{e: j H P and j H Q}) both with $Q$ and with $Q'$ in place of $Q$ we find
$$
A(P:Q':\xi_M:\gl)\after j_H(Q':\xi_M:\gl) = A(P:Q: \xi_M:\gl) \after j_H(Q:\xi_M: \gl)
$$
combining this with (\ref{e: composition of intertwiners}) and using
that $A(P:Q': \xi_M:\gl)$ is injective for generic $\gl,$ we obtain that
$$
j_H(Q': \xi_M: \gl) = A(Q':Q:\xi_M :\gl - \rhoQh) j_H(Q:\xi_M:\gl)
$$
for generic $\gl \in \faqdc.$ Since the expressions on both sides of the equation are meromorphic $V(\xi, e)^*\otimes \Cminf(K:M:\xi_M)$-valued functions, the identity holds as an identity of meromorphic functions.
The identity also holds with $H$ replaced by $vHv^{-1},$ as an identity
of $V(\xi, v)^*\otimes \Cminf(K:M: \xi_M)$-valued meromorphic functions of $\gl \in \faqdc.$ If we apply this to each of the
terms of the sum in (\ref{e: defi j without H}) we obtain (\ref{e: j Qprime and j Q}). This establishes (b).

We now turn to (a). Fix $P \in \parabsgsA$ such that $P\succeq Q.$
Then assertion (a) holds with $P$ in place of $Q,$ in view of
Corollary \ref{c: equality of jP and jP0}.
To establish assertion (a) for $j(Q:\xi_M:\gl)(\eta)$ as well, we use (b) with
$Q' = P.$ Then assertion (a) follows from the fact that $A(P:Q:\xi_M : \gl - \rhoQh)$
is intertwining and injective for generic $\gl \in \faqdc.$
\end{proof}

\section{Eisenstein integrals}
\label{s: Eisenstein integrals}
In this section we will extend the definition of Eisenstein
integrals for minimal $\gs\Cartan$-stable parabolic subgroups
from $\parabsgsAq$ to similar integrals for minimal parabolic subgroups
from $\parabsA.$

First we need to carefully discuss the parameter space for the Eisenstein
integral.
In view of Lemma \ref{l: lemma on deco M zero},  it follows that the inclusion map $M \to M_0$ induces a diffeomorphism
$
M/H_{M}\simeq  M_0/H_{M_0}.
$
This diffeomorphism induces a topological linear isomorphism $C^\infty(M/H_{M} ) \simeq C^\infty(M_0/H_{M_0})$ via which we will identify the elements of these spaces.

Let $(\tau, V_\tau)$  be a finite dimensional unitary representation of
$K.$ Then we define $\tau_{M_0}$ to be the restriction of $\tau$ to $M_0.$
Likewise, we define $\tau_M$ to be the restriction of $\tau$
to $M.$ Then $\tau_{M_0}$ and $\tau_M$ have the same representation space.

We define $C^\infty (M_0/H_{M_0}: \tau_{M_0})$ to be the space of smooth  functions $\psi:M_0/H_{M_0}\to V_{\tau}$ satisfying the transformation rule
$$
\psi(kx)=\tau(k)\psi(x)\qquad \big(k\in K\cap M_0, \;x\in M_0/H_{M_0}\big).
$$
 Similarly, we define $C^\infty (M/H_{M}: \tau_{M})$ to be the space of smooth  functions $\psi:M/H_{M}\to V_{\tau}$ satisfying the transformation rule
$$
\psi(mx)=\tau(m)\psi(x)\qquad \big(m\in M,\; x\in M/H_{M}\big).
$$
We then have the obvious inclusion
$$
 C^\infty (M_0/H_{M_0}: \tau_{M_0})\;
 \subseteq \; C^\infty(M/H_{M}: \tau_{M}).
$$
In general, the first of these spaces will
be strictly contained in the second. The first of these spaces enters the definition
of the Eisenstein integral for minimal $\gs\Cartan$-stable parabolic subgroup
from $\parabsgsAq,$ whereas the second is convenient in the context of induction
from a minimal parabolic subgroup from $\parabsA.$ The relation between the
spaces can be clarified as follows. Since $M$ normalizes $M_{0\rmn} \cap K$
it follows that the space $V_\tau^0$ of $M_{0\rmn}\cap K$-invariants in $V_\tau$
is invariant under $\tau(M),$ so that we may define the following representation
$\tau_M^0$ of $M$ by restriction:

\begin{equation}
\label{e: defi tau M zero}
\tau_M^0:= \tau_M|_{V_\tau^0}, \quad {\rm where} \quad
V_{\tau}^0: =( V_\tau)^{M_{0\rmn} \cap K},
\end{equation}
Observe that for every $v \in \cW$ we have
$$
V_\tau^{M_0 \cap K \cap vHv^{-1}} = (V_\tau^0)^{M \cap v H v^{-1}}.
$$
Indeed, this follows from the fact that $M_0 \cap K = M (M_{0\rmn}\cap K)$ and that $\tau(M_{0\rmn}\cap K) = 1$ on $V_\tau^0.$

\begin{lemma}
\label{l: spherical functions on M and Mzero}
Let $(\tau, \Vtau)$ be a finite dimensional unitary representation of $K.$ Then
\begin{equation}
\label{e: spherical functions on M and  Mzero}
C^\infty (M_0/H_{M_0}: \tau_{M_0})\;=\;
C^\infty(M/H_{M} : \tau_{M}^0),
\end{equation}
\end{lemma}

\begin{proof}
We observe that $M_{0\rmn}$ acts trivially on $M_0/H_{M_0}$
by Lemma \ref{l: lemma on deco M zero} (f).
Therefore, every function in the space on the left-hand side of (\ref{e: spherical functions on M and Mzero}) has values in $V_{\tau}^0$ and we see that the space on the left is indeed contained in the space on the right.
For the converse inclusion, let $f: M/H_{M} \to V_\tau^0$ be a function in the space on the right. If $k_0 \in M_0 \cap K$ we may write
$k_0 = k_M k_n$ with $k_M \in M$ and $k_n \in M_{0\rmn}\cap K.$
Let $m_0 \in M_0,$ then $m_0 = m h$ for a suitable $m \in M$ and $h \in H_{M_0}.$ Since $M_{0\rmn} \subseteq H,$ it follows that
\begin{eqnarray*}
f(k_0 m_0) & = &  f(k_M k_n m h)
 = \tau(k_M) f(m(m^{-1}k_n m)h) \\
 &=& \tau(k_M) f(m)   = \tau(k_M)\tau(k_n ) f(m_0) \\
 &=&  \tau(k_0) f(m_0).
 \end{eqnarray*}
  It follows that $f$ belongs to the space on the left.
  \end{proof}

We are now prepared for the definition of the Eisenstein integral related to a fixed parabolic subgroup $P \in \parabsA$.
Given $\psi \in C^\infty(M/H_{M}: \tau_M^0)$ we define
the function  $\psi_{P,\gl}: G \to \Vtau$ by
$$
\psi_{P,\gl} (kman ) = a^{\gl  - \rho_P - \rhoPh}\,\tau(k) \psi(m).
$$
We denote by $C^{\infty}(G/H:\tau)$ the space of smooth functions $\phi:G/H\to V_{\tau}$
 satisfying the rule
$$
\phi(kx)=\tau(k)\phi(x)\qquad(k\in K, \; x\in G/H).
$$
Recall the definition of $\holset_P$ from (\ref{e: defi Omega Q}) with $P$
in place of $Q.$

\begin{prop}
\label{p: first intro Eisenstein}
Let $\omega\in \cD_{\fh/\fh_P}.$
Let $\psi \in C^\infty(M/H_{M}: \tau_M^0)$ and let
$\lambda\in\Omega_P$.
Then the following assertions are valid.
\begin{enumerate}
\itema
For each $x \in G$ the function
$$
h \mapsto \psi_{P, \gl}(xh)\;dl_h(e)^{-1*}\omega
$$
defines a $\Vtau$-valued density on $H/H_P.$ \vspace{-5pt}
\itemb For each $x \in G$ the density in (a) is integrable.\vspace{-5pt}
\itemc The function
$E_H(P\col\psi\col  \gl) : G \to \Vtau$ defined by
$$
E_H(P\col\psi\col  \gl)(x): = \int_{ H/ H_P}\;
 \psi_{P, \gl}(xh)\,dl_h(e)^{-1*}\,\omega, \quad (x \in G),
$$
in accordance with (a) and (b), belongs to $C^\infty(G/H\col \tau).$
\end{enumerate}
\end{prop}

\begin{proof}
Before we start with the actual proof, we note that the condition on $\gl$
implies the existence of parabolic subgroup $P' \in \cP_\gs(A,P)$
such that $\gl \in \Omega_{P', P},$ in view of (\ref{e: defi Omega Q}).

Let $F_\tau\subseteq \widehat M$  denote the finite set of $M$-types in $\tau^\vee$ and let $\cH_\tau$
denote the subspace of $C^\infty(M/H_{M})$ consisting of the left $M$-finite functions of isotype contained in $F_\tau.$
Let $L$ be the left regular representation of  $M$ in $C^\infty(M/H_{M})$
and let $\xi_\tau:= L|_{\cH_\tau}$ be its restriction to the subspace $\cH_\tau.$

Since $C^\infty(M/H_{M}:\tau_M)$ consists
of the $M$-fixed functions in $C^\infty(M/H_{M})\otimes V_\tau,$
it follows that
$$
C^\infty(M/H_{M} : \tau_M) \subseteq \cH_{\tau} \otimes V_\tau.
$$
We define the function $\tilde \psi_{P, \gl}: G \to \cH_\tau \otimes V_\tau$ by
$
\tilde \psi_{P,  \gl}(x)(m) := \psi_{P, \gl}(xm).
$
Then it  readily follows  that
$$
\tilde \psi_{P, \gl}(x man) =   a^{\gl - \rho_P - \rhoPh} (\xi_\tau(m)^{-1} \otimes 1)
\tilde \psi_{P, \gl}(x).
$$
We define $\psi^\vee_{P,\gl}(x): = \tilde \psi_{P,  \gl}(x^{-1}).$  Then 
\begin{equation}
\label{e: space psi vee}
\psi^\vee_{P,\gl} \in C^\infty(P\col \xi_\tau    \col - \gl +\rhoPh)\otimes\Vtau.
\end{equation}
Let $\gf: H /H_P  \to H_P \backslash H $ be the diffeomorphism induced by $h \mapsto h^{-1}.$
Then $d\gf(e)^* \omega = \omega$ and for $x \in G$ we see that
\begin{eqnarray}
\nonumber
\lefteqn{\gf^* [h  \mapsto \psi_{P, \gl}(xh)\;dl_h(e)^{-1*}\omega]}\\
\label{e: density for Eisenstein} \qquad
&=&
[h \mapsto
\psi_{P, \gl}^\vee (h x^{-1})
(e)\; dr_h(e)^{-1*}\omega].
\end{eqnarray}
Let $\eps$ denote the element of $\cH_\tau^{H_{M}}$ such that
$\inp{g}{\eps}= g(e)$ for all $g \in \cH_\tau.$
We may now apply
 Corollary \ref{c: first convergence j H Q} to the first tensor component of the space in (\ref{e: space psi vee}) with $(P',P)$ in place of $(P,Q),$
 with $(\xi_{\tau} , \cH_\tau )$  in place of $(\xi, \cH_\xi),$ with $R_{x^{-1}}(\psi_{P, \gl}^\vee)$ in place of $f$ (where $R$ denotes the right regular representation) and with $\eps$ in place of $\eta.$
From applying the corollary in this fashion, it follows that the expression on the right-hand side of (\ref{e: density for Eisenstein}) is a $V_\tau$-valued density on $H_P\bs H$ which is integrable. This implies (a) and (b).

Using  that $x \mapsto R_x (\psi^\vee_{P, \gl})$ is smooth as a function with values in
the Fr\'echet space
$C^\infty(P : \xi_\tau  \col -\gl +\rhoPh)\otimes \Vtau,$
we find that  $E_{H}(P:\psi:\gl) \in C^\infty(G, V_\tau).$

The right $H$-invariance and the $\tau$-spherical behavior are readily checked.
\end{proof}

\begin{rem}
The above procedure would also work more generally  for functions $\psi \in C^\infty(M/H_{M}: \tau_M).$ However, for generic $\gl \in \faqdc$ 
the map $\psi \mapsto E_H(P:\psi:\gl)\psi$ would then have a (possibly $\gl$-dependent) kernel complementary to $C^\infty(M/H_M:\tau_M^0).$ 
\end{rem}

For $v \in \cW$ the above procedure applies to the data $K, vHv^{-1}, A, \Aq$ in place of $K, H, A, \Aq.$
We thus obtain Eisenstein integrals $E_{vHv^{-1}}(P\col  \psi \col \gl \col x)$ for $\psi$ in the parameter space $C^\infty(M/M \cap v H v^{-1}: \tau_M^0).$
The general Eisenstein integral is defined as follows.
For $v \in \cW$ we equip $L^2(M/M\cap vHv^{-1})$ with the
$L^2$-inner product for the normalized invariant measure,
and $L^2(M/M\cap v H v^{-1}) \otimes \Vtau$ with the
tensor product inner product. The latter restricts to an inner product
on the finite dimensional subspace  $C^\infty(M/M\cap vHv^{-1}:\tau_M^0).$
We define
\begin{equation}
\label{e: defi cA M two}
\cA_{M,2} := \oplus_{v \in \cW} \;C^\infty(M/M\cap vHv^{-1}: \tau_M^0).
\end{equation}
Equipped with the direct sum of the given inner products on the summands,
this space becomes a finite dimensional Hilbert space.

For $\psi \in \cA_{M,2}$ and $\gl\in\holset_{P}$ we define the function $E(P:\psi:\gl): G \to \Vtau$ by
\begin{equation}
\label{e: defi Eisenstein}
E(P\col \psi\col \gl)(x) = \sum_{v \in \cW} E_{vH v^{-1}}(P\col \pr_v \psi \col \gl)(xv^{-1})
\qquad(x\in G).
\end{equation}
It is readily verified that this function belongs to $C^\infty(G/H: \tau).$
We will occasionally write $E(P\col \psi\col \gl\col x)$ for $E(P\col \psi\col \gl)(x)$.

We will now relate the Eisenstein integral thus defined to matrix coefficients
with $H$-fixed distribution vectors. For this we will use a suitable realization of
the space $\cA_{M,2}.$ In analogy with  (\ref{e: defi cA M two}) we define
$$
\cA_{M_0, 2}: = \oplus_{v\in \cW} \; C^\infty(M_0/M_0\cap vH v^{-1}: \tau_{M_0}).
$$
In view of Lemma \ref{l: spherical functions on M and Mzero} applied with
$vH v^{-1}$ in place of $H,$ for $v \in \cW,$ we see that
$$
\cA_{M_0, 2} = \cA_{M,2}.
$$
For $\xi \in \dMzerofu$ and $v \in \cW,$ we denote by $C_\xi^\infty(M_0/M_0\cap vHv^{-1})$
the space of left $M_0$-finite functions in $C^\infty(M_0/M_0 \cap v H v^{-1})$
of isotopy type $\xi.$ Furthermore, we denote by
\begin{equation}
\label{e: type xi spherical functions on Mzero}
C^\infty_\xi(M_0/M_0\cap vH v^{-1}: \tau_{M_0})
\end{equation}
the intersection of $ C^\infty(M_0/M_0\cap vH v^{-1}: \tau_{M_0})$ with
$C_\xi^\infty(M_0/M_0\cap vHv^{-1}) \otimes \Vtau.$ The direct sum of the spaces
(\ref{e: type xi spherical functions on Mzero}) for $v \in \cW$ is denoted by $\cA_{M_0, 2, \xi}.$
Then it follows that
\begin{equation}
\label{e: deco cA Mzero two}
\cA_{M_0, 2} = \oplus_{\xi \in \dMzerofu} \cA_{M_0, 2, \xi},
\end{equation}
as an orthogonal direct sum with finitely many non-zero terms.

Similar definitions, with $M$ in place of
 $M_0,$ lead to spaces
\begin{equation}
\label{e: type xi spherical functions on M}
C^\infty_\xi(M/M\cap vHv^{-1} : \tau_M^0),
\end{equation}
equal to (\ref{e: type xi spherical functions on Mzero}) in view of
(\ref{e: spherical functions on M and Mzero}), for $v \in \cW.$ The orthogonal direct sum of
(\ref{e: type xi spherical functions on M}) over $v \in \cW$ is denoted by
$\cA_{M,2, \xi}.$ Then obviously
\begin{equation}
\label{e: equality cA M and M zero}
\cA_{M_0, 2, \xi} = \cA_{M,2, \xi}.
\end{equation}

For $\xi \in \dMzerofu$ we define  $C(K: \xi: \tau)$ to be the space of functions  $f: K \to \cH_\xi \otimes \Vtau$ transforming according to the rule:
$$
f(m k_0 k ) = [\xi(m) \otimes \tau(k)^{-1}] f(k_0), \qquad (k,k_0 \in K, m \in M_0\cap K).
$$
We recall from \cite[Lemma 3, p.\ 528]{BSft} that there exists a natural linear isomorphism
\begin{equation}
\label{e: defi psi T}
T \mapsto \psi_T,\quad C(K:\xi: \tau) \otimes \bar V(\xi)\;\;{\buildrel\simeq \over \longrightarrow} \;\;\cA_{M_0, 2,\xi } = \cA_{M,2,\xi},
\end{equation}
given by
\begin{equation}
\label{e: formula psi T}
(\psi_T)_v(m)= \inp{f(e)}{\xi(m)\pr_v(\eta)},\quad (v \in \cW),
\end{equation}
for $T = f \otimes \eta \in C(K:\xi: \tau) \otimes \bar V(\xi)$ and $m \in M.$
Moreover, $ T\mapsto \sqrt{\dim \xi}\;\psi_{T}$ is an isometry.

The map  $\inj$ introduced in (\ref{e: i sharp on K}) 
is an isometric embedding $C(K : K\cap M_0:\xi) \to C(K:M: \xi_M).$ Through tensoring with the 
identity map on $\Vtau$ it induces an isometric embedding $C(K:\xi:\tau) \to C(K: \xi_M:\tau_M)$ which we denote by $\inj$ again. 

\begin{thm}
\label{t: eisenstein as matrix coefficient}
Let $\xi \in \dMzerofu$ and let $T = f \otimes \eta  \in C(K\col \xi \col \tau) \otimes \bar V(\xi).$
Then for $x \in G$ and
$\gl \in \convergeset_{P}$,
$$
E(P\col \psi_T \col \gl \col x)
= \inp{\inj \! f}{\pirep_{P, \xi_M, \bar \gl-\rhoPh}(x) j(P\col \xi_M \col \bar \gl)(\eta)},
$$
where $j(P:\xi_M:\bar \gl)(\eta)$ should be viewed 
as an element of $C^{-\infty}(K:M:\xi_M).$ Moreover, the indicated sesquilinear pairing is taken on the first tensor components of $\inj\! f.$ 
\end{thm}

\begin{proof}
The set $\Upsilon_P$ is the intersection of the sets $\Omega_{v,P},$ for $v \in \cW,$ by
(\ref{e: defi Upsilon}) and Lemma \ref{l: holset v Q}.
In view of (\ref{e: defi Eisenstein}) and (\ref{e: defi j without H})
it therefore suffices to restrict to the case that $\eta \in V(\xi, e)$
and prove the result under the (weaker) assumption that $\gl \in \Omega_P.$

Write $\psi = \pr_e \psi_T.$
Then it follows from the proof of Proposition \ref{p: first intro Eisenstein} that,
for $x$ and $\gl$ as specified,
\begin{eqnarray*}
E_H(P\col  \psi_T \col \gl \col x) & =&
\int_{H_P\bs H}  \gf^* [h  \mapsto \psi_{P, \gl}(xh)\;dl_h(e)^{-1*}\omega]\\
&=&
\int_{H_P\bs H} \psi_{P, \gl}^\vee (h x^{-1})(e)\; dr_h(e)^{-1*}\omega.
\end{eqnarray*}
We now calculate the function
$\psi^\vee_{P,\gl} $ in this particular case. As it belongs to
$C^\infty(P\col \xi_M \col -\gl+\rhoPh)\otimes \Vtau$
it is sufficient to calculate its restriction to $K.$
Since $\psi = \psi_T,$ it follows from (\ref{e: formula psi T}) that 
$
\psi(m) = \psi_T(m) = \inp{f(e)}{\xi(m)\eta }.
$
This implies that
$$
\psi_{P,\gl}(k) = \tau(k) \inp{f(e)}{\xi(e)\eta} = \inp{f(k^{-1})}{\eta}.
$$
In turn, this implies that
$$
\psi^\vee_{P,\gl}(k)(e) = \inp{f(k)}{\eta}.
$$
We write $[\inj\!f]_{P,-\gl + \rhoPh }$ for the extension of the function $\inj \!f\in C(K:\xi_M:\tau_M)$ 
to a function in $C^\infty(P:\xi_M: -\gl + \rhoPh) \otimes \Vtau.$ 
Then 
$$ 
\psi^\vee_{P,\gl}(x)(e) = \inp{[\inj\! f]_{P, -\gl +  \rhoPh}(x)}{\eta}.
$$ 
Thus, in view of (\ref{e: integral for j H Q}) we find that 
\begin{eqnarray*}
E_H(P\col  \psi_T \col \gl \col x)
&=&   \int_{H_P\bs H}\big\langle[R(x^{-1})[\inj\! f]_{P,-\gl +\rhoPh}(h),\eta\big\rangle\; dr_h(e)^{-1*}\; \omega
\\
&=& \big\langle [R(x^{-1})[\inj\! f]_{P,-\gl +\rhoPh} ,j_H(P:\xi_M: \bar \gl )(\eta)\big\rangle\\
 &=& \big\langle\pirep_{P,\xi_M , - \gl + \rhoPh}(x^{-1})\inj \! f  ,j_H(P:\xi_M: \bar \gl )(\eta)\big\rangle\\
&=& \big\langle \inj \!  f,\pirep_{P,\xi_M , \bar \gl - \rhoPh}(x)j_H(P:\xi_M: \bar \gl )(\eta)\big\rangle.
\end{eqnarray*}\vspace{-12pt}
The proof is complete.
\end{proof}

\begin{cor}
\label{c: holomorphy Eisenstein integral}
Let $P\in \parabsA$ and let $\psi \in \cA_{M,2}.$ Then the Eisenstein integral $E(P:\psi:\gl)$
depends meromorphically on $\gl\in \faqdc$  as a function with values in
$C^\infty(G/H:\tau).$ As such, it is holomorphic on an open neighborhood of the set $\hatUpsilon_P.$
\end{cor}

\begin{proof}
The assertion about meromorphy follows from the previous result in view of (\ref{e: deco cA Mzero two}) and the linear dependence of the Eisenstein integral on $\psi.$
The statement about holomorphy now follows from Lemma
\ref{l: holomorphy of j} (b).
\end{proof}

 It  will sometimes be convenient to write
$E(P:\gl:x)\psi = E(P:\psi:\gl:x)$ and to adopt the viewpoint that
$E(P:\gl)$ is a meromorphic $\Hom(\cA_{M,2} , C^\infty(G:\tau))$-valued
function
of $\gl \in \faqdc.$

We proceed by relating the Eisenstein integrals defined above to
the Eisenstein integrals introduced earlier in \cite{Banps2} and \cite{BSft}
for minimal $\gs\Cartan$-stable parabolic subgroups.

\begin{cor}
\label{c: equality q extreme Eisenstein integral}
Let $P \in \parabsgsA$ and let $P_0 $ be the unique parabolic subgroup
from $\parabsgsAq$ containing $P.$
Then
\begin{equation}
\label{e: equality eisenstein P and P zero}
E(P: \gl) = E(P_0: \gl)
\end{equation}
as  $\Hom\big(\cA_{M_0,2}, C^\infty(G/H:\tau)\big)$-valued meromorphic functions of $\gl \in \faqdc.$
\end{cor}

\begin{proof}
Let $\xi \in \dMzerofu.$ Then it follows from Corollary \ref{c: equality of jP and jP0} that
$$
j(P : \xi_M: \gl) = \inj \after j(P_0: \xi: \gl).
$$
Let $T = f\otimes \eta \in C(K\col \xi \col \tau) \otimes \bar V(\xi).$ 
Then it follows by \cite[Lemma 4.2]{Banps2} and (\ref{e: i sharp on K}) 
that
\begin{eqnarray*}
E(P_0: \gl: x)\psi_T
&=& \inp{f}{\pirep_{P_0, \xi, \bar \gl}(x) j(P_0: \xi : \bar \gl)(\eta)}\\
&=&\inp{\inj \!f}{ \inj \pirep_{P_0, \xi, \bar \gl}(x) j(P_0: \xi : \bar \gl)(\eta)}\\
&=& \inp{\inj\! f}{\pirep_{P, \xi_M, \bar \gl - \rhoPh}(x) j(P: \xi_M: \bar \gl)(\eta)}\\
&=& E(P: \gl : x)\psi_T.
\end{eqnarray*} 
\vspace{-36pt}

\end{proof}
\medno\medno
The Eisenstein integrals for parabolic subgroups from $\parabsA$ can be related to each other as follows.

\begin{prop}
\label{p: relation eisensteins for dominating parab}
Let $Q \in \parabsA,$ $P \in \cP_\gs(A)$ and $P \succeq Q.$ Then for all $\xi \in \dMzerofu$, all $T \in C(K:\xi: \tau) \otimes \bar V(\xi)$
and generic $\gl \in \faqdc$, we have
\begin{equation}
\label{e: relation Eisenstein Q and P}
E(Q:\gl)\psi_{T}
=
E(P:\gl)\psi_{[\proj \after A(Q:P:\xi_M:-\gl + \rhoPh)^{-1} \after \inj \otimes I]T}.
\end{equation}
Here, $\proj$ is shorthand for the restriction of $\proj \otimes I_{\Vtau}$ 
to the subspace $C(K:\xi_M:\tau_M)$ of $C^\infty(K:M:\xi_M) \otimes \Vtau,$ see also
(\ref{e: defi proj}). 
Likewise, the intertwining operator acts on the first tensor component 
in $C(K: M: \xi_M) \otimes V_\tau.$ 
\end{prop}
\begin{proof}
By linearity it suffices to prove this for  $T = f \otimes \eta,$ with
$f \in C(K:\xi: \tau)$ and $\eta \in \bar V(\xi).$
It follows from Theorem \ref{t: eisenstein as matrix coefficient}
and (\ref{e: j Qprime and j Q}) that for generic $\lambda\in\faqdc$ we have
$$ 
 j(Q: \xi_{M} : \bar \gl )(\eta) = A(P:Q:\xi_M:\bar\gl - \rhoPh)^{-1} j(P:\xi_M: \bar \gl)(\eta).
 $$
For $x \in G$ we now obtain (with the pairing taken on first tensor components)
\begin{eqnarray}
\nonumber\lefteqn{
E(Q: \gl : x) \psi_T =}\\
\nonumber
&=&
\inp{\inj f}{\pirep_{Q, \xi_{M}, \bar \gl - \rhoQh}(x)
 j(Q: \xi_M : \bar \gl )(\eta)}\\
 \nonumber
 &=&  \inp{\inj f}{A(P:Q:\xi_{M}: \bar \gl - \rhoPh )^{-1} \pirep_{P, \xi_{M}, \bar \gl - \rhoPh}(x)
                           j(P: \xi_M : \bar \gl )(\eta)}
\\
\nonumber
&=&
\inp{ A(Q:P:\xi_{M}: - \gl +  \rhoPh )^{-1} \inj  f}{\pirep_{P, \xi_{M}, \bar \gl - \rhoPh}(x)
                          j(P: \xi_M : \bar \gl )(\eta)}\\
               \label{e: is Eisenstein}                       
                          &=&
\inp{p  \after A(Q:P:\xi_{M}: -  \gl +  \rhoPh )^{-1}  \inj f}{\pirep_{P, \xi_{M}, \bar \gl - \rhoPh}(x)
                          j(P: \xi_M: \bar \gl )(\eta)}
                          \end{eqnarray}
The  last identity follows by application of Corollary \ref{c: p and j P}.
Finally, since $p = \inj\after\proj,$ the expression in (\ref{e: is Eisenstein}) 
equals the expression on the right-hand side of (\ref{e: relation Eisenstein Q and P}),
in view of Theorem \ref{t: eisenstein as matrix coefficient}.
\end{proof}

Let $Q,P \in \cP(A)$ be as in Proposition \ref{p: relation eisensteins for dominating parab}. Then motivated by the proposition, 
we define the $C$-function $C(Q:P:\gl)$ to be the unique $\End(\cAMtwo)$-valued 
meromorphic function of $\gl \in \faqdc$ such that 
$$ 
C(Q: P: \gl) \psi_T = \psi_{[\proj \after A(Q:P:\xi_M:-\gl + \rhoPh)^{-1} \after \inj \otimes I]T}
$$ 
for all $\xi \in \dMzerofu$, all $T \in C(K: \xi: \tau) \otimes \bar V(\xi)$
and generic $\gl \in \faqdc.$  Then (\ref{e: relation Eisenstein Q and P})
may be abbreviated as 
\begin{equation}
\label{e: first appearance C}
E(Q: \gl) = E(P:\gl) C(P:Q:\gl).
\end{equation}

The following result is a variation of a result of Harish-Chandra, see 
\cite[Lemma 3, p.\ 47]{HC2}. The proof given below follows a different strategy, which also 
works in the setting of \cite{HC2}.

\begin{prop}
\label{p: det C non zero}
Let $Q \in \cP(A)$ and $P \in \cP_\gs(A)$ such that $P \succeq Q.$  Then the meromorphic
function $\faqdc \ni \gl \mapsto \det \; C(Q:P:\gl)$ is not identically zero.
\end{prop}
 
Before proceeding with the proof of this proposition, we first list a corollary.

\begin{cor}
\label{c: meromorphic inverse C}
Let $Q \in \cP(A)$ and $P \in \cP_\gs(A)$ such that $P \succeq Q.$ 
Then the endomorphism $C(Q:P:\gl) \in \End(\cAMtwo)$ is invertible 
for generic $\gl \in \faqdc$ and $\gl \mapsto C(Q:P:\gl)^{-1}$ is a meromorphic 
$\End(\cAMtwo)$-valued meromorphic function on $\faq.$ 
\end{cor}

\begin{proof}
This follows from Proposition \ref{p: det C non zero}
by application of Cramer's rule for the inversion of a matrix.
\end{proof}
The following lemma will play an important role in the proof of Proposition
\ref{p: det C non zero}.

\begin{lemma}
\label{l: diagram proj A}
 Let $P\in \cP_\gs(A)$ and let $P_0 $ be the unique parabolic subgroup from 
$\cP_\gs(\Aq)$ containing $P.$ Then the following diagram commutes:
$$ 
\begin{array}{ccc}
C(K: K\cap M_0 : \xi) &\buildrel \scriptscriptstyle{A(\bar P_0: P_0: \xi : \gl)}\over  \longrightarrow & C(K: K\cap M_0 : \xi) \\
\scriptstyle{\proj} \uparrow & & \uparrow \scriptstyle{\proj} \\
C(K: M : \xi_M) & \buildrel \scriptscriptstyle{A(\gs(P) : P: \xi_M :  \gl + \rhoPh )}\over  \longrightarrow  & C(K: M : \xi_M) 
\end{array}
$$
for generic $\gl \in \faqdc.$ 
\end{lemma}

\begin{proof}
Let $P_{M_0} = P \cap M_0.$ Then it follows that $P = P_{M_0} N_{P}$ 
and $\gs P = P_{M_0} N_{\gs P}.$ Furthermore, by the assertion at the end of the proof 
of Theorem 4.2 in \cite[p. 373]{Banps1}, with $P_1 = \bar P_0,$ $P_2 = P_0,$ 
$(P_1)_\fp = \gs(P)$ and $(P_2)_\fp = P,$ it follows that the following diagram
commutes for generic $\gl \in \faqdc,$ 
$$ 
\begin{array}{ccc}
C(K: K\cap M_0 : \xi) &\buildrel \scriptscriptstyle{A(P_0: \bar P_0: \xi : - \bar \gl)}\over  \longleftarrow & C(K: K\cap M_0 : \xi) \\
\scriptstyle{\inj} \downarrow & & \downarrow \scriptstyle{\inj } \\
C(K: M : \xi_M) & \buildrel \scriptscriptstyle{A(P: \gs(P) : \xi_M : - \bar \gl - \rho_{Ph})}\over  \longleftarrow  & C(K: M : \xi_M) .
\end{array}
$$ 
The desired result now follows by taking adjoints with respect to the given equivariant sesquilinear pairings on the spaces involved. 
\end{proof}

\begin{proof}[Proof of Proposition \ref{p: det C non zero}]
Let $\xi \in \dMzerofu$ and let $\vartheta \subseteq \widehat K$ be a finite set 
of $K$-types. Then it suffices to show that the restricted operator 
\begin{equation}
\label{e: A proj and inj}
\proj \after A(Q:P:\xi_M : -\gl + \rhoPh)^{-1} \after \inj|_{C(K:K\cap M_0:\xi)_\vartheta}
\end{equation}
has determinant not identically zero. For this it suffices to show that 
the composition of $A(\bar P_0: P_0 :\xi : -\gl)$ 
with (\ref{e: A proj and inj}) has determinant not-identically zero. By Lemma \ref{l: diagram proj A} this composition
may be rewritten as
\begin{equation}
\label{e: composition intertwiners and proj inj} 
\proj \after A(\gs(P): P: \xi : - \gl + \rhoPh) \after 
A(Q:P:\xi_M : -\gl + \rhoPh)^{-1} \after \inj|_{C(K:K\cap M_0:\xi)_\vartheta}.
\end{equation}
Since $\Sigma(\gs(P) \cap \Sigma(P) = \Sigma(P) \cap \fahd \subseteq \Sigma(Q) \cap \Sigma(P),$ 
it follows by the usual product decomposition of the standard intertwining operators,
see \cite[Cor. 7.7]{KSII}, that  (\ref{e: composition intertwiners and proj inj}) equals 
\begin{equation}
\label{e: A gs P Q}
\proj \after A(\gs P : Q : \xi_M : - \gl + \rhoPh) \after \inj|_{C(K:K\cap M_0:\xi)_\vartheta}.
\end{equation}
Thus it suffices to show that the determinant of the linear endomorphism 
of $C(K: K\cap M_0: \xi)_\vartheta$ given in (\ref{e: A gs P Q}) is not identically zero as a meromorphic 
function of $\gl.$  Now this is an immediate consequence of the following result.
\end{proof}

\begin{lemma}
\label{l: asymptotics A R Q}
Let $Q,R \in \cP(A)$ be such that $\gS(\bar R) \cap \gS(Q) \subseteq \gS(Q, \gs\Cartan).$ 
Then there exists an element $\eta \in \faqd$ such that $\inp{\eta}{\ga} > 0$ 
for all $\ga \in \gS(\bar R) \cap \gS(Q).$ Let $\eta$ be such an element
and $d := \dim(N_R \cap \bar N_Q).$ 
Then there exists a constant $c > 0$ such that for every $\mu \in \fadc$ 
and all $f \in C(K: M : \xi_M)$ 
$$ 
\lim_{t \to \infty} \;\; t^{d/2}A(R:Q:  \xi_M : \mu + t\eta) f = c f.
$$ 
in $C(K:M:\xi_M).$ 
\end{lemma}

\begin{proof}
Since $\gS(Q)$ is a positive system for $\gS(\fg, \fa),$ 
there exists $\xi \in \fad$ such that $\inp{\xi}{\ga} >0$ for $\ga \in \gS(Q).$ 
Let $\eta = \xi + \gs\Cartan \xi,$ then $\inp{\eta}{\ga} >0$ for $\ga \in \gS(Q,\ga \Cartan).$ 
Thus, $\eta$ satisfies the requirements. Let $\eta$ be any such element.

Replacing $\mu$ by $\mu + t_0 \eta$ for a suitable $t_0 > 0$ we see that we may 
as well assume that $\Re \inp{\mu}{\ga} > 0$ for all $\ga \in \gS(\bar R) \cap \gS(Q).$ 
In this case, we see that for all $t \geq 0$ the intertwining operator is given by the absolutely convergent integral
\begin{eqnarray*}
[A(R:Q:  \xi_M : \mu + t\eta) f ] (k)  & =& \int_{N_R \cap \bar N_Q} f_t(\bar n^{-1} k) d\bar n\\
&=& 
\int_{N_R \cap \bar N_Q} e^{- t \eta \cH_Q(\bar n)} e^{(-\mu - \rho_Q) \cH_Q(\bar n) }
f(\kappa_Q(\bar n)^{-1}k)  \; d\bar n.
\end{eqnarray*}
Here $f_t$ denotes the extension of $f$ to an element of $C^\infty(Q: \xi_M: \mu + t \eta).$ 
Moreover, the analytic maps $\cH_Q: G \to \fa$ and $\kappa_Q: G \to K$ are defined by 
$$ 
x \in \kappa_Q (x) \exp \cH_Q(x) N_Q,\qquad (x \in G),
$$ 
in accordance with the Iwasawa decomposition $G = K A N_Q.$ 

By using the properties of the function $h = \eta\cH_Q|_{N_R\cap \bar N_Q}$ 
stated in Lemma \ref{l: stationary phase} below we will be able to determine the asymptotic behavior for $t \to \infty$ by using the real version of the method of stationary phase. 

It follows from Lemma \ref{l: stationary phase} (a) that $h \geq 0.$ 
Hence, the intertwining operator is a continuous linear endomorphism of 
$C(K: M: \xi_M),$ with operator norm bounded by 
$$ 
\| A(R:Q:  \xi_M : \mu + t\eta)\| \leq \int_{N_R \cap \bar N_Q} e^{(-\Re \mu - \rho_Q) \cH_Q(\bar n)} \; d \bar n.
$$ 

It follows from Lemma \ref{l: stationary phase} (b) 
that there exists an open neighborhood $V$ of $0$ in $\R^d$ 
and an open embedding $\gf: \R^d \to  N_R\cap \bar N_Q,$ sending $0$ to $e$ such that 
$$ 
h(\gf(x)) = \inp{Sx}{x}
$$ 
with $S$ a positive definite matrix. Let $U$ be an open neighborhood of $e$ 
in $N_R \cap \bar N_Q$ with closure contained in  $\gf(V),$ and let $r >0$ be as in condition (c) of the mentioned lemma.
Fix $\chi \in C_c^\infty(\gf(V))$ such that $\chi = 1 $ on a neighborhood of the closure of $U.$ 
Then 
$$ 
A(R: Q : \xi_M: \mu + t\eta) f = I_t(f) + R_t(f)
$$ 
with 
$$ 
I_t(f) = \int_{N_R \cap \bar N_Q} e^{- t h(\bar n)} e^{(-\mu - \rho_Q) \cH_Q(\bar n) }
f(\kappa(\bar n)^{-1}k) \chi(\bar n)  \; d\bar n.
$$ 
The remainder term $R_t(f)$ is given by the same integral but with $\chi(\bar n)$ replaced
by $1 - \chi(\bar n).$ As the latter function is zero on $U$, it follows from the estimate in (b) that 
$$ 
\|R_t(f)\| \leq C_1 e^{-tr} \|f\|,
$$ 
with $C_1$ a positive constant independent of $f$ and $t.$ 
Accordingly, we may ignore this term and concentrate on $I_t(f).$ By substituting 
$\gf(x)$ for $\bar n$ we obtain
$$ 
I_t(f) = \int_{\R^n} e^{- t \inp{Sx}{x}} e^{(-\mu - \rho_Q) \cH_Q(\gf(x) )}
f(\kappa(\gf(x))^{-1}k) \chi(\gf(x)) J(x)\; dx,
$$ 
where $J(x)$ is a Jacobian. Substituting $t^{-1/2} x$ for $x$ and taking the limit for $t \to \infty,$ we see that
$$ 
t^{d/2} I_t(f) \to  \int_{\R^d} e^{-\inp{Sx}{x}}  f(k) \, J(0) \;dx,  
$$ 
uniformly in $k.$ This establishes the result with 
$$ 
c = J(0) \int_{\R^d} e^{-\inp{Sx}{x}} \;dx.
$$ 
\end{proof}

\begin{lemma}
\label{l: stationary phase}
 Let $Q,R \in \cP(A)$ and let $\eta \in\fad$ be such that $\inp{\eta}{\ga} >0$ 
for all $\ga \in \gS(\bar R) \cap \gS(Q).$ Let $\cH_Q: G \to \fa$ be the Iwasawa map determined by $x \in K \exp \cH_Q(x) N_Q,$ for 
$x \in G.$ Then the function $h = \eta \cH_Q|_{N_R \cap \bar N_Q}$ has the following properties,
\begin{enumerate}
\itema
$h \geq 0;$ \vspace{-5pt}
\itemb $h$ has an isolated critical point at $e$ with 
positive definite Hessian;\vspace{-5pt}
\itemc for each open neighborhood $U$ of $e$ in $N_R \cap \bar N_Q$ there exists a constant 
$r > 0$ such that
$$ 
\bar n\in (N_R\cap \bar N_Q) \setminus U \;\; \Longrightarrow  \;\; h(\bar n) > r.
$$ 
\end{enumerate}
\end{lemma}

\begin{proof}
For each indivisible root $\ga \in \gS(\fg, \fa)$ we write $\fn_\ga = \fg_\ga + \fg_{2\ga},$ 
$N_\ga = \exp(\fn_\ga),$ 
$\bar \fn_\ga = \Cartan \fn_\ga$ and $\bar N_\ga := \exp (\bar \fn_\ga).$ 
Furthermore, we write $\fg(\ga)$ for the split rank one subalgebra generated by 
$\fn_\ga + \bar \fn_{\ga}$  and $G(\ga)$ for the associated analytic subgroup of $G.$ 
Let $\cH_\ga: G(\ga) \to \fa\cap \fg(\ga) = (\ker\ga)^\perp$ be the Iwasawa projection
associated with the Iwasawa decomposition $G(\ga) = K\cap G(\ga) )(A \cap G(\ga)) N_\ga.$ 
Then $\cH_\ga = \cH_Q|_{G(\ga)}.$ 

Let $\ga_1, \ldots, \ga_k$ 
be the indivisible positive roots in $\gS(\bar R) \cap \gS(Q)$.
Then by the method of S.G.\ Gindikin and F.I.\  Karpelevic \cite{GK} (see \cite[Thm. 7.6]{KSII}
for the version for intertwining operators), there exists a diffeomorphism $\psi: \bar N_{\ga_1} \times \cdots 
\bar N_{\ga_k} \to N_R \cap \bar N_Q$ such that 
$$ 
h (\psi(\bar n_1 , \ldots, \bar n_k)) = \sum_{j=1}^k \eta \cH_{\ga_j}(\bar n_j).
$$ 
To show that $h$ has properties (a), (b) and (c), it suffices to show that each 
of the functions $h_j = \eta \after \cH_{\ga_j}|_{\bar N_{\ga_j} }$ has these properties,
with $\bar N_{\ga_j}$ in place of $N_R \cap  \bar N_Q.$ 

Let $\ga \in \gS(\fg, \fa)$ be any indivisible root such that $\inp{\eta}{\ga} > 0.$ Then it suffices
to show that $h_\ga: \bar N_\ga \to \R, \bar n \mapsto \eta \cH_\ga(\bar n)$ has properties 
(a) and (b) with $\bar N_\ga$ in place of $N_R \cap \bar N_Q.$ The function $h_\ga$ can 
be explicitly computed through ${\rm SU(2,1)}$-reduction, see \cite[Thm.\ IX.3.8]{HelDS}. 
From the explicit expression given in \cite{HelDS}, 
properties (a), (b) and (c) are readily verified.
\end{proof}

We will now describe the asymptotic behavior of the various Eisenstein integrals,
using the established relations (\ref{e: equality eisenstein P and P zero}) and (\ref{e: first appearance C}) between them.

For a parabolic subgroup $R \in \parabsgsAq$ and for $v \in \cW$ we define the functions
$$
\Phi_{R, v}(\gl : \cdot): \Aqp(R) \to \End(V_\tau^{M_0 \cap K \cap v H v^{-1}})
$$
as in \cite[Lemma 10.3]{BSexpans}. These functions are smooth on the chamber $\Aqp(R)$ and as such depend meromorphically on the parameter $\gl \in \faqdc.$ Moreover, for generic $\gl \in \faqdc$ they have an absolutely converging series expansion of the form
$$
\Phi_{R,v}(\gl: a) = a^{-\rho_R} \sum_{\mu \in \N\gS^+(R)}  a^{-\mu} \Gamma_{R,\mu}(\gl) ,
$$
where the $\Gamma_{R,\mu}$ are meromorphic $\End (V_\tau^{M_0 \cap K \cap vHv^{-1}})$-valued functions and $\Gamma_{R, 0} = I.$

Let $P_0 \in \parabsgsAq$ then by \cite[Thm.\ 11.1]{BSexpans}  and
(\ref{e: equality cA M and M zero}),
there exist unique $\End(\cAMtwo)$-valued meromorphic functions $C_{R|P_0}(s: \dotvar )$ on $\faqdc$ such
that for all $\psi \in \cAMtwo$ and each $v \in \cW$ and generic
$\gl \in \faqdc$ we have
$$
E(P_0: \gl: av)\psi = \sum_{s \in W(\faq)} \Phi_{R, v}(s \gl: a)[C_{R|P_0}(s:\gl)\psi]_v(e),  \qquad (a \in \Aqp(R)).
$$
Here $W(\faq)$ denotes the Weyl group of the root system $\gS(\fg, \faq).$

\begin{thm}
\label{t: asymptotics Eisenstein integrals}
Let $Q \in \parabsA$ and $R \in \parabsgsAq.$ Then there exist unique meromorphic $\End(\cA_{M,2})$-valued meromorphic functions
$C_{R|Q}(s:\dotvar)$ on $\faqdc,$ for $s \in W(\faq),$
such that for all $\psi \in \cAMtwo,$ each $v \in \cW$ and generic
$\gl \in \faqdc$ we have
$$
E(Q: \gl: av)\psi = \sum_{s \in W(\faq)} \Phi_{R, v}(s \gl: a)[C_{R|Q}(s:\gl)\psi]_v(e),  \qquad (a \in \Aqp(R)).
$$
These meromorphic $C$-functions are generically pointwise invertible,
with meromorphic inverses.
\end{thm}

\begin{proof}
Uniqueness follows by uniqueness of asymptotics, see, e.g.,
\cite[p.\ 305, Cor.]{HCsf1}
for details.

For the remaining statements on existence and invertibility, 
we first consider the case that $Q$ is $\fq$-extreme,
i.e., $Q\in \parabsgsA.$ Then there exists a unique $Q_0 \in \cP_\gs(\Aq)$ 
containing $Q.$ By applying Corollary \ref{c: equality q extreme Eisenstein integral} and the preceding discussion we find that 
$$ 
C_{R|Q}(s: \gl) = C_{R|Q_0}(s:\gl) 
$$ 
satisfies the asymptotic requirements. Invertibility follows from \cite[Cor.\ 15.11]{Banps2}.

We now assume that $Q \in \parabsA$ is general. 
Then there exists a
$P\in \parabsgsA$ such that $P \succeq Q.$ 

By Proposition \ref{p: relation eisensteins for dominating parab}  and 
(\ref{e: relation Eisenstein Q and P}) we have
$$
E(P:\gl:x) = E(Q:\gl:x) \after C(Q: P:\gl).
$$
\medbreak
In view of Proposition \ref{p: det C non zero} we see that 
$$
C_{R|Q}(s:\gl) = C_{R|P}(s: \gl) \after C(Q:P:\gl)^{-1}
$$
satisfies the asymptotic requirements. The invertibility requirements now follow from
the invertibility of $C_{R|P}(s:\gl),$ established earlier in this proof.
\end{proof}

\begin{cor}
\label{c: functional c functions}
Let $P,Q\in \parabsA.$
Then there exists a unique
meromorphic $\End(\cAMtwo)$-valued function
$C(P:Q:\dotvar)$ on $\faqdc$ such that
\begin{equation}
\label{e: general relation Eisensteins}
E(P:\gl:x) = E(Q:\gl: x) \after C(Q:P:\gl)
\end{equation}
for all $x \in G/H$ and generic $\gl \in \faqdc.$ Furthermore, the following identities 
are valid as identities of meromorphic $\End(\cAMtwo)$-valued functions in $\gl \in \faqdc.$
\begin{enumerate}
\itema
$
C(Q:P:\gl) = C_{R|Q}(s:\gl)^{-1} C_{R|P}(s:\gl),
\qquad (s \in W(\faq), \;R\in \parabsgsAq);$ \vspace{-5pt}
\itemb
$
C(P_1:P_2:\gl) \,C(P_2:P_3:\gl) = C(P_1:P_3:\gl),
\qquad (P_1,P_2, P_3 \in \parabsA);$\vspace{-5pt}
\itemc
$
C(P:Q:\gl) C(Q:P:\gl) = C(Q:P:\gl)C(P:Q:\gl) = I.
$
\end{enumerate}
\end{cor}

\begin{proof}
Uniqueness follows from Theorem \ref{t: asymptotics Eisenstein integrals} combined
with uniqueness of asymptotics. We will first establish the existence 
for $P, Q \in \cP_\gs(A).$ Let $P_0, Q_0$ be the unique minimal $\gs\Cartan$-stable 
parabolic subgroups  in $\cP_\gs(\Aq)$ with $P_0 \supset P$ and $Q_0 \supset Q.$ 
Then by \cite[(42) \& (70)]{BSft} there exists a meromorphic function 
$a: \faqdc \to \End(\cAMtwo)$ 
such that $E(P_0: \gl) = E(Q_0: \gl) a(\gl).$ In view of Corollary 
\ref{c: equality q extreme Eisenstein integral} 
it follows that (\ref{e: general relation Eisensteins}) is valid with $C(Q:P:\gl) = a(\gl).$ 

By using Proposition  \ref{p: relation eisensteins for dominating parab}, (\ref{e: relation Eisenstein Q and P}) and Corollary \ref{c: meromorphic inverse C}
the existence of $C(Q:P:\gl)$ can now be inferred for arbitrary $P,Q \in \cP(A).$ 

Now that the existence has been established, (a) follows from Theorem 
\ref{t: asymptotics Eisenstein integrals} combined with uniqueness of asymptotics. Finally, (b) and (c) follow
from the established uniqueness of the $C$-functions involved.
\end{proof}
\section{The case of the group}
\label{s: the case of the group}
In this section we will consider the case of the group, viewed as a symmetric space,
and compare our definition of the Eisenstein integral for a minimal parabolic subgroup
with the one given by Harish-Chandra \cite{HCha3}.

Let $\bp G$ be a group of the Harish-Chandra class, and let $G = \bp G \times \bp G$
and $H$ the diagonal in $G.$  Then $H$ equals the fix point group of the involution
$\gs: G \to G$ given by $\gs (x,y) = (y,x).$
The map $m: (x,y)\mapsto x y^{-1}$ induces a diffeomorphism
$G/H \to \bp G$ which is equivariant for the action of $G$ on $G/H$ by left translation
and the action on $\bp G$ by left times right translation. Accordingly, pull-back by $m$ induces
a $G$-equivariant topological linear isomorphism $m^*: C^\infty(\bp G) \to C^\infty(G/H).$

We fix a Cartan involution $\bp \theta$ for $\bp G.$ Let $\bp \fg = \bp \fk \oplus \bp \fp$ be
the associated infinitesimal Cartan decomposition and let $\bp \fa$ be a fixed choice
of a maximal abelian subspace of $\bp \fp$.
Then $\theta = \bp \theta \times \bp \theta$ is a Cartan involution for $G$ which commutes with $\gs.$
The associated Cartan decomposition is given by
$\fg = \fk \oplus \fp,$ where $\fk = \bp\fk \times \bp \fk$ and
$\fp = \bp \fp \times \bp \fp.$  Furthermore,
$\fa = \bp \fa \times \bp \fa$ is a maximal abelian subspace of $\fp.$

The infinitesimal involution $\gs$ on $\fg = \bp \fg \times \bp \fg$ is given by $(X,Y) \mapsto (Y,X),$
so that its $+1$ eigenspace $\fh$  equals the diagonal of $\fg,$ whereas the $-1$-eigenspace $\fq$
consists of the elements $(X, -X),$ $X \in \bp \fg.$ It follows that
$\fp \cap \fq = \{(X,-X)\mid X \in \bp \fp\},$ and that the subspace
$$
\faq: = \{(X, -X)  \mid X \in \bp \fa\}
$$
is maximal abelian in $\fp \cap \fq.$ Furthermore, $\fa = \fah \oplus \faq,$ where
$\fah = \{(X,X)\mid X \in \bp \fa\} = \fa \cap \fh.$ At the level of groups we accordingly have
$A = \Ah \Aq,$ where $\Ah = A \cap H = \{(a,a)\mid a \in \bp A\}$ and $\Aq=\{(a, a^{-1})\mid a \in \bp A\}.$
The root system $\gS$ of $\fa$ in $\fg$ equals $(\bp \gS \times \{0\} )\cup (\{0\} \times \bp\gS)$,
where $\bp \gS$ denotes the root system of $\bp  \fa$ in $\bp \fg.$ The associated root spaces are given by
$$
\fg_{(\ga, 0)} = \fg_\ga \times \{0\}, \quad{\rm and} \quad \fg_{(0, \gb)}= \{0\} \times \fg_\gb,\qquad (\ga, \gb \in \bp \gS).
$$
The positive systems for $\gS$ are the sets of the form $(\Pi_1 \times \{0\}) \cup (\{0\}\times \Pi_2$ where
$\Pi_1, \Pi_2$ are positive systems for $\bp \gS.$ Accordingly,
$$
\cP(A) = \{\bp P \times \bp Q \mid \bp P , \bp Q \in \cP(\bp A)\}.
$$
Let $\bp M$ denote the centralizer of $\bp A$ in $\bp K.$ Then the centralizer of
$A$ in $K$ is given by $M = \bp M \times \bp M$ and we see that the $\theta$-stable
Levi component of any parabolic in $\cP(A)$ is equal to $M A.$

Our first objective is to give a suitable description of the $H$-fixed distribution vector
$j(R:\xi:\gl)(\eta),$ for $R = \bp P \times \bp Q$ a minimal parabolic subgroup from
$\parabs(A),$ for $\gl \in \faqdc,$ and for $\xi \in \dM$ such that the space $V(\xi),$ defined
as in (\ref{e: defi V xi}), is non-trivial.

We observe that $N_K(\faq)$ and $N_{K\cap H}(\faq)$ have the same image in $\GL(\faq),$ so that
$\cW,$ defined as in (\ref{e: defi cW}), consists of the identity element $e = (\bp e, \bp e).$
It follows that $V(\xi) = V(\xi, e)$
as in (\ref{e: defi V xi}), so that
\begin{equation}
\label{e: defi V xi group case}
V(\xi) = \cH_\xi^{H_{M}}.
\end{equation}
Thus, $V(\xi)\neq 0$ if and only if $\xi$ has a non-trivial $H_{M}$-fixed vector.
The set of such (classes of) irreducible representations of $M$ is denoted by $\dM_{H_{M}}.$

 If $\xi \in \dM_{H_{M}},$ then
\begin{equation}
\label{e: xi as tensor prod}
\xi \simeq \bp \xi \hatotimes \bp \xi^\vee,
\end{equation}
for an irreducible unitary representation
$\bp \xi$ of $\bp M$ in a finite dimensional Hilbert space $\cH_{\bp\xi}.$
Using the canonical identification
\begin{equation}
\label{e: natural identification End H xi}
\cH_{\bp \xi} \otimes \cH_{\bp \xi}^* \simeq
\End(\cH_{\bp \xi})
\end{equation}
 we shall model $\xi$ as the representation in
$\cH_\xi:= \End(\cH_{\bp \xi})$ given by
$$
\xi(m_1, m_2) T = \bp\xi(m_1) \after T \after \bp\xi(m_2)^{-1},
$$
for $T \in \End(\cH_{\bp \xi})$ and $m_1, m_2 \in \bp M.$
In particular, we see that with this convention,
$$
V(\xi) = \C I_{\bp \xi}.
$$
The space $\faqdc$ is identified with the subspace of $\fadc$ consisting of linear functionals
on $\fadc$ of the form $(\bp \gl, -\bp\gl) : (X,Y) \mapsto \bp\gl(X) - \bp \gl (Y).$ We agree to write
$$
\gl = (\bp \gl, - \bp \gl), \qquad (\bp \gl \in \bp \fadc).
$$
As in Section \ref{s: induced reps and densities},
we write $C^{\pm \infty}(K:\xi)$ for $C^{\pm \infty}(K:M:\xi)$
and $C^{\pm \infty}(\bp K: \bp \xi)$ for $C^{\pm \infty}(\bp K: \bp M: \bp \xi). $ Then as indicated in Section \ref{s: induced reps and densities},
we have topological linear isomorphisms
$$
C^{-\infty}(K: \xi) \simeq C^\infty(K:\xi)' \quad {\rm and} \quad
C^{-\infty}(\bp K: \bp \xi) \simeq C^\infty(\bp K: \bp \xi)',
$$
which restricted to the subspaces of smooth functions are induced
by the pairings (\ref{e: bilinear pairing on K}) for $(K, \xi)$ and $(\bp K, \bp \xi).$

We now consider the topological linear isomorphism
$$
\Phi: \Cminf(K: \xi) \;\;
\buildrel\simeq \over \longrightarrow
\;\;
\Hom(C^\infty(\bp K: \bp \xi), \Cminf(\bp K: \bp \xi))
$$
determined by the Schwartz kernel theorem. It is given by
$$
\inp{\Phi(h)(f)}{g} = \inp{h}{g \otimes f},
$$
for $h \in C^{-\infty}(K:\xi),$ $f \in C^\infty(\bp K: \bp \xi)$ and
$g \in C^\infty(\bp K : \bp \xi^\vee),$ with $g \otimes f$ viewed as an
element of $C^\infty(K: \xi^\vee).$

According to the compact picture explained in Section \ref{s: induced reps and densities}, we may identify $\Phi$
with a uniquely determined topological linear isomorphism
$$
\Phi_\gl : \Cminf(R:\xi:\gl) \buildrel\simeq \over \longrightarrow
\Hom(C^\infty(\bp Q,\bp \xi, \bp \gl),
\Cminf(\bp P,\bp \xi, \bp \gl)).
$$
The isomorphism $\Phi_\gl$ is readily seen to be $G$-equivariant, by $G$-equivariance of the pairings involved in the definition of $\Phi,$ for the appropriate principal series representations.
It maps the $H$-invariants in the space on the left to the subspace of $\bp G$-intertwining operators on the right.

We write $\inp{\dotvar}{\dotvar}_\xi$ for the $K$-equivariant pre-Hilbert structure on $C^\infty(K:\xi)$
given by (\ref{e: sesquilinear pairing on K})
 and $\inp{\dotvar}{\dotvar}_{\bp \xi}$ for the similar $\bp K$-equivariant pre-Hilbert structure on
$C^\infty(\bp K : \bp \xi). $ The latter structure extends to continuous sesquilinear pairings
$
C^{\pm \infty}(\bp K:\bp \xi) \times C^{\mp \infty }(\bp K: \bp \xi) \to \C, $ also denoted by
$\inp{\dotvar}{\dotvar}_{\bp \xi}.$  As $C^\infty(\bp K : \bp \xi)$ is a Montel space, it is reflexive, and we may take adjoints with respect to these pairings. Accordingly, given
a continuous linear operator $T: C^\infty(\bp K: \bp \xi) \to \Cminf(\bp K :\bp \xi)$
we define  the continuous linear operator $T^*: C^\infty(\bp K: \bp \xi) \to \Cminf(\bp K :\bp \xi)$ by
$$
\inp{T^* f}{g}_{\bp \xi} = \inp{f}{Tg}_{\bp \xi},\qquad (f,g \in C^\infty(\bp K: \bp \xi)).
$$

\begin{lemma}
\label{l: Phi inv T star and trace}
Let $F \in C^\infty(K: \xi)$ and let $T: C^{\infty}(\bp K:\bp \xi) \to C^{\infty}(\bp K: \bp\xi)$ be a
continuous linear operator. Then
\begin{equation}
\label{e: Phi inv T star and trace}
\inp{F}{\Phi^{-1}(T^*)}_\xi = \int_{\bp K}\tr_{\bp \xi}[(T \otimes I) F (\bp k,\bp k) ]\; d\bp k.
\end{equation}
\end{lemma}

\begin{proof}
We first consider the isomorphism  $\gf: \Cminf(K) \to \Hom(C^{\infty}(\bp K), \Cminf(\bp K))$ given by the Schwartz kernel isomorphism. Let ${f_j}$ denote an $L^2(\bp K)$-orthonormal basis subordinate to the decomposition into the finite dimensional
$\bp K$-isotypical components with respect to the left regular representation.
Then for each smooth function $f \in C^\infty(\bp K)$ we have
$
f = \sum_j \inp{f}{ f_j}_{2} f_j = \sum_j \inp{f}{\bar f_j} f_j
$
with convergence in $L^2(\bp K).$ Here index $2$ indicates that the pairing corresponds to the
sesquilinear $L^2$-inner product.
It follows that for each $K$-finite function $F \in C^{\infty}(K)$ we have
$$
\inp{F}{\gf^{-1} (I)}_{2} = \inp{F}{\sum_j f_j \otimes \bar f_j}_{2}.
$$
For $F = f_k \otimes f_l$ this gives
$
\inp{F}{\gf^{-1}(I)}_{2} =   \inp{f_l}{f_k}= \int_K f_k(k) f_l(k) \; dk.
$
By continuous linearity and density this implies that
$$
\inp{F}{\gf^{-1}(I)}_2 = \int_K F(k,k)\; dk, \qquad (F \in C^\infty(\bp K\times \bp K)).
$$
We next consider the natural isomorphism $\psi$ from $\cH_{\xi} = \cH_{\bp \xi} \otimes \cH_{\bp\xi}^\vee$ onto $\End(\cH_{\bp \xi}). $ Then it is readily verified that
$$
\inp{U}{\psi^{-1}(I_{\bp \xi})}_{\xi} = \tr_\xi( \psi(U))
\qquad(U\in\cH_{\xi}).
$$
Here the index $\xi$ indicates that the natural sesquilinear inner product induced
by the inner product on $\cH_\xi$ is taken. We now consider the Schwartz kernel isomorphism $\widetilde \Phi$ from $\Cminf(K, \cH_\xi)$ onto $\Hom(C^{\infty}(\bp K, \cH_{\bp \xi}), \Cminf(\bp K, \cH_{\bp \xi})).$ Then $\widetilde \Phi$ is  identified with $\gf \otimes \psi$
in a natural way. Thus, for $F \in C^\infty(K, \cH_\xi))$ we have
\begin{equation}
\label{e: pairing F with Phi tilde inverse}
\inp{F}{\widetilde \Phi^{-1} (I)}_\xi = \int_{\bp K}{ \tr_\xi (\psi(F(\bp k, \bp k))}\; d\bp k.
\end{equation}
Identifying $\cH_\xi$ with $\End (\cH_{\bp \xi})$ via $\psi$ we agree to rewrite the above expression
without the $\psi.$ We view $C^\infty (K :\xi)$ as the space of $M = \bp M \times \bp M$-invariants in $C^\infty (K , \cH_\xi).$ Likewise we view $C^{\pm \infty}(\bp K : \bp \xi)$ as the space
of $\bp M$-invariants in $C^{\pm \infty}(\bp K, \cH_{\bp \xi})$ (for the right action of $\bp M$
on $C^{\pm \infty}(\bp K)$). The $\bp M$-equivariant inclusion maps and projection maps will
be denoted by $\rmi$ and $P$ respectively. Then $\Phi = \widetilde \Phi\after \rmi = P \after \widetilde \Phi \after \rmi ,$ and we find that for $F \in C^\infty(K: \xi)$
\begin{equation}
\label{e: equality pairing Phi and Phi tilde}
\inp{F}{\Phi^{-1}(I)}_2 = \inp{F}{\widetilde \Phi^{-1} (I) }.
\end{equation}
This implies (\ref{e: Phi inv T star and trace}) with $T = I.$ To obtain the general formula, we note that for a continuous linear operator $T \in \End( C^{\infty}(\bp K: \bp \xi))$ the Hermitian adjoint $T^*$ is a continuous linear operator in $\End(\Cminf(\bp K:\bp \xi))$ and
$$
\Phi\big((T^* \otimes I) u\big) = T^*\after \Phi(u)
\qquad\big(u\in \Cminf(K:\xi)\big).
$$
For $u=\Phi^{-1}(I)$ this yields
$$
(T^* \otimes I)\Phi^{-1}(I)= \Phi(T^*).
$$
It follows that
$$
\inp{F}{\Phi^{-1}(T^*)} = \inp{F}{(T^*\otimes I)\Phi^{-1}(I)} = \inp{(T \otimes I)(F)}{\Phi^{-1}(I)}.
$$
Hence, (\ref{e: Phi inv T star and trace}) follows by application of (\ref{e: pairing F with Phi tilde inverse}) and  (\ref{e: equality pairing Phi and Phi tilde}).
\end{proof}

\begin{lemma}
\label{l: j P Q and Phi intertwiner}
Let $\bp P,
\bp Q \in \cP(\bp A).$ Then for generic $\bp \gl \in \bp \fadc,$
\begin{equation}
\label{e: j P Q and Phi intertwiner}
j(\bp P \times \bp Q: \xi: \gl) (I_{\bp\xi})= \Phi^{-1}_\gl(A(\bp P : \bp  Q: \bp \xi : \bp \gl )).
\end{equation}
\end{lemma}

\begin{proof}
Put $R = \bp P \times \bp Q$ as before. Then in the present case of the group,
$\rho_{Rh} = 0,$ so that the distribution vector on the left-hand side  of
(\ref{e: j P Q and Phi intertwiner}) belongs to $\Cminf(R:\xi: \gl).$

It follows from
(\ref{e: j Qprime and j Q}) applied with $Q' = \bp P \times \bp \bar P$ and
$Q = \bp P \times \bp Q$ that
$$
j(\bp P \times \bp \bar P: \xi:  \gl)
= [I \otimes A(\bp \bar P : \bp Q : \bp \xi^\vee:-\bp\gl ) ]\after j(\bp P \times \bp Q : \xi:  \gl).
$$
Since $A(\bp \bar P : \bp Q : \bp \xi ^\vee: - \bp \gl )$ has transpose
$A(\bp Q : \bp \bar P: \bp \xi : \bp \gl)$ relative to the bilinear pairing
$C^\infty(\bp K:\bp\xi) \otimes C^\infty(\bp K: \bp\xi^\vee) \to \C,$
it follows that
\begin{equation}
\label{e: two times j and Phi}
\Phi_{\lambda}\big(j(\bp P \times \bp \bar P : \xi: \gl) (I_{\bp\xi})\big) =
\Phi_{\lambda}\big(j(\bp P \times \bp Q : \xi: \gl) (I_{\bp\xi})\big)\after A(\bp Q :\bp \bar P : \bp \xi : \bp\gl )
\end{equation}
For $\bp Q = \bp \bar P$ the equality (\ref{e: j P Q and Phi intertwiner}) has been established in
\cite[Lemma 1]{BSmulti}.
Combining this with (\ref{e: two times j and Phi}) we find that
\begin{equation}
\label{e: j as Phi inverse of A}
A(\bp P : \bp \bar P : \bp \xi : \bp\gl ) = \Phi_{\lambda}(j(\bp P \times \bp Q : \xi: \gl) (I_{\bp\xi}))\after
A(\bp Q :\bp \bar P : \bp \xi :\bp\gl ).
\end{equation}
The intertwining operator on the left-hand side of (\ref{e: j as Phi inverse of A})
decomposes as  the composition
$$
A(\bp P : \bp Q : \bp \xi : \bp\gl ) \circ A(\bp Q : \bp \bar P : \bp \xi : \bp\gl ),
$$
as an $\End(C^\infty(K:\xi))$-valued meromorphic function of
$\bp \gl\in \bp\fadc.$
Using the invertibility of the second intertwining operator for generic
$\gl \in \bp \fadc$ we obtain (\ref{e: j P Q and Phi intertwiner}).
\end{proof}

\begin{cor}
\label{c: first expression Eisenstein integral group case}
Let $f \in C^\infty(K: \xi).$ Then for generic $\bp \gl \in \bp \fadc,$
\begin{equation}
\label{e: j in pairing is trace}
\inp{f}{j(\bp P \times \bp Q : \xi: -\bar \gl) (I_{\bp\xi})} =
\int_{\bp K}\tr_{\bp \xi} \left([ A(\bp Q : \bp P :\bp \xi : \bp\gl)\otimes I)f ] (\bp k,\bp k) \right)\; d\,\bp k.
\end{equation}
\end{cor}

\begin{proof}
For generic $\bp \gl \in \bp\fadc,$ the continuous linear endomorphism
$T:= A(\bp Q : \bp P: \bp\xi : \bp\gl)$ of $C^\infty(\bp K:\bp\xi)$
has Hermitian adjoint $T^*= A(\bp P : \bp Q : \bp \xi: - \bp\bar \gl).$
The result now follows by combining
 Lemma \ref{l: j P Q and Phi intertwiner}, with $- \bp \bar  \gl$ in place
of $\bp \gl,$  and
Lemma \ref{l: Phi inv T star and trace}.
\end{proof}

The expression on the left-hand side of (\ref{e: j in pairing is trace}) is very closely
related to an Eisenstein integral for the parabolic subgroup $R = \bp P \times \bp Q,$
defined as in Definition \ref{t: eisenstein as matrix coefficient}.
This will allow us to express the Eisenstein integral in terms of the group structure of $\bp G.$

To be more precise, let $\xi$ be as in (\ref{e: xi as tensor prod}) and let $(\tau, \Vtau)$ be a finite dimensional unitary representation of $K.$  We recall the definition of the space $C(K:\xi:\tau)$ and the definition of the linear isomorphism $T\mapsto \psi_T$ from
$C(K:\xi:\tau) \otimes V(\xi)$ onto $\cA_{2, M, \xi}$ from (\ref{e: defi psi T}) and the surrounding text (note that $M_0 = M$).

Since $\cW = \{e\},$ we have
$$
\cA_{2, M, \xi} = C^\infty_\xi(M/H_{M}: \tau_M).
$$
Since $V(\xi) = \C I_{\bp \xi},$ it follows that the following map is a linear isomorphism;
\begin{equation}
\label{e: first iso f to psi f}
f \mapsto \psi_{f \otimes I_{\bp \xi}}, \quad C(K:\xi:\tau)
\;\;
\buildrel \simeq \over \longrightarrow \;\; 
 C^\infty_\xi(M/H_{M}: \tau_M).
\end{equation}
It follows from (\ref{e: formula psi T}) that
$$
\psi_{f \otimes I_{\bp \xi}}(m^{-1}) = \inp{f(m)}{I_{\bp \xi}}_{\rm HS} = \tr_{\bp \xi} (f(m))
\qquad(m\in M),
$$
where the subscript $HS$ means that the Hilbert-Schmidt inner product is taken.

\begin{cor}
With notation as in Corollary \ref{c: first expression Eisenstein integral group case},
let $f \in C^\infty(K:\xi:\tau).$ Then
\begin{eqnarray}
\label{e: second expression Eisenstein integral group case}
\lefteqn{E(\bp P \times \bp Q : \psi_{f \otimes I_{\bp \xi}}: \gl) (\bp x, e) =}\\
\nonumber
\qquad
&=&
\int_{\bp K}\tr_{\bp \xi} \left(
\Big[
\Big(A(\bp Q : \bp P :\bp\xi : - \bp\gl)
\otimes
\pirep_{\bp Q, \bp \xi^\vee , \bp \gl} (\bp x) \Big)
f
\Big]
(\bp k,\bp k) \right)\; d\bp k,
\end{eqnarray}
for $\bp x \in \bp G$ and generic $\gl \in \faqdc.$
\end{cor}

\begin{proof}
We note that $M_{0}=M,$ so that $K\cap M_0 = M,$ $\xi_M = \xi$ and the map 
$\inj$ introduced in (\ref{e: i sharp on K}) is just the identity map in the present setting.
By application of Theorem \ref{t: eisenstein as matrix coefficient} with $R = \bp P \times \bp Q$
in place of $P,$ we now  find, taking into account that $\rho_{Rh} = 0,$ that the Eisenstein integral
on the left-hand side of (\ref{e: second expression Eisenstein integral group case}) equals
\begin{equation}
\naam{e: eisenstein as pairing}
\inp{f}{\pirep_{R, \xi, \bar \gl} (\bp x, e) j(R: \xi :\bar \gl)(I_{\bp \xi})}
=
\inp{f}{\pirep_{R, \xi, \bar \gl} (e, \bp x^{-1}) j(R: \xi :\bar \gl)(I_{\bp \xi})},
\end{equation}
by $H$-invariance of $j.$
Here $\inp{\dotvar}{\dotvar}$ stands for the sesquilinear map
$C^{\infty}(K:\xi:\tau) \times \Cminf(K: \xi) \to \Vtau$ induced by
the sesquilinear pairing $C^{\infty}(K: \xi) \times \Cminf(K: \xi) \to \C.$
By equivariance of the pairing, (\ref{e: eisenstein as pairing}) equals
$$
\inp{\pirep_{R, \xi, - \gl}(e, \bp x) f}{j (R: \xi :\bar \gl)(I_{\bp \xi})}
=
\inp{[ I \otimes \pirep_{\bp Q, \bp \xi^\vee ,  \bp\gl}(\bp x) ] f}{j (R: \xi :\bar \gl)(I_{\bp \xi})}
$$
By application of (\ref{e: j in pairing is trace}) we infer that the last displayed expression
equals the integral on the right-hand side of (\ref{e: second expression Eisenstein integral group case}).
\end{proof}

We shall now relate the Eisenstein integral in (\ref{e: second expression Eisenstein integral group case})
to Harish-Chandra's Eisenstein integral for the group.
We agree to write
$$
\tau_1(k) v =  \tau(k, e) v, \quad {\rm and}\quad  v \tau_2(k) : = \tau(e, k^{-1}) v,\qquad (v \in \Vtau, k \in \bp K).
$$
Then $(\tau_1,$ $\tau_2)$ is a unitary bi-representation of  $\bp K$ in $\Vtau$ in the sense
that $\tau_1$ is a unitary left representation and $\tau_2$ a unitary right representation of $\bp K$
in $\Vtau$ and these two representations commute.
Clearly any such bi-representation $(\tau_1, \tau_2)$  of $\bp K$
comes from a unique unitary representation $\tau$ as above, and $ \tau(k_1, k_2) v = \tau(k_1) v \tau(k_2^{-1}),$ for $v \in \Vtau$ and $(k_1, k_2) \in K.$
Given $\tau$ as above, we agree to write $\tau_M$ for the restriction of $\tau$ to $M.$
Furthermore, we agree to write $\tau_{j\bp M}$ for the restriction of $\tau_j$ to $\bp M,$ for
$j =1,2.$  Then $\tau_M$ corresponds to the bi-representation $ (\tau_{1 \bp M}, \tau_{2\bp M})$
of $\bp M.$

Let $C^\infty(\bp M: \tau_M)$ denote the space of smooth functions $\gf : = \bp M \to \Vtau$
transforming according to the rule
$$
\gf(m_1 m m_2) = \tau_1(m_1) \gf(m) \tau_2(m_2),\qquad (m, m_1, m_2 \in \bp M).
$$
Then it is readily verified that pull-back under the map $m: (x,y) \mapsto xy^{-1}$ induces
a linear isomorphism
\begin{equation}
\label{e: m star for M}
m^*: \; C^\infty(\bp M: \tau_M)   \;\;{\buildrel \simeq \over \longrightarrow}
\;\;   C^\infty(M/H_{M}:\tau_M).
\end{equation}
The inverse of this isomorphism will be denoted by
\begin{equation}
\label{e: psi to bp psi}
\psi \mapsto \bp \psi, \quad    C^\infty(M/H_{M}:\tau_M)
\;\;{\buildrel \simeq \over \longrightarrow}
\;\;   C^\infty(\bp M: \tau_M).
\end{equation}
By $\bp M \times \bp M$-equivariance, it follows that the isomorphism (\ref{e: psi to bp psi})
restricts to an isomorphism
\begin{equation}
\label{e: m star for M, types}
 C^\infty_\xi(M/H_{M}:\tau_M) \simeq C^\infty_{\xi}(\bp M: \tau_M).
\end{equation}
Here the space on the right-hand side is defined as the intersection
of $C^\infty(\bp M: \tau_M)$ with the space $C_\xi(\bp M) \otimes \Vtau,$
where $C_\xi(\bp M)$ denotes the isotypical component of type $\xi$ for the representation
$L \times R$ of $M$ in $C(\bp M).$ Furthermore, the space on the left-hand side of
(\ref{e: m star for M, types}) is defined similarly.

Since (\ref{e: first iso f to psi f}) is an isomorphism, it now follows that the following
map is a linear isomorphism as well,
$$
f \mapsto \bp \psi_{f \otimes I_{\bp \xi}},  \quad  C(K:\xi: \tau) \;\;
\buildrel\simeq \over \longrightarrow\;\;
C^\infty_{\xi}(\bp M: \tau_M).
$$
We now recall the definition of Harish-Chandra's Eisenstein integral associated
with a parabolic subgroup $\bp Q \in \cP(\bp A).$
Given $\bp \psi \in C^\infty(\bp M: \tau_{\bp M})$ and $\bp \gl \in \bp \fadc,$
we define the function $\bp \psi_{\bp\gl} : \bp M \to \Vtau$ by
\begin{equation}
\label{e: definition bp psi gl}
\bp \psi_{\bp\gl} (\bp n \bp a \bp m \bp k) =
\bp a^{\bp \gl + \rho_{\bp Q}}  \bp \psi(\bp m) \tau_2(\bp k),
\end{equation}
for $\bp k \in \bp K,$ $ \bp m \in \bp M,$ $\bp a
\in \bp A$ and $\bp n \in N_{\bp Q}. $
The Harish-Chandra Eisenstein integral for the group $\bp G$ is now defined by
\begin{equation}
\label{e: first defi Eisenstein integral}
E_\HC(\bp Q : \bp \psi : \bp \gl)(\bp x) = \int_{\bp K}
\tau_1(\bp k)^{-1} \bp \psi_{\bp \gl} (\bp k \bp x)\; d\bp k,
\end{equation}
for $\bp \gl \in \bp \fadc$ and $\bp x \in \bp G.$ We will derive a formula for the present
type of Eisenstein integral, which will allow comparison with (\ref{e: second expression Eisenstein integral group case}). In the formulation of the
following lemma, we will use the natural identifications (\ref{e: natural identification End H xi}) and
$$
C(K:\xi:\tau) =
\left(
C^\infty(K:\xi) \otimes \Vtau
\right)^K
\simeq
\left(
C^\infty(\bp K: \bp \xi) \otimes C^\infty(\bp K : \bp \xi^\vee) \otimes \Vtau
\right)^K.
$$
Furthermore, we will write $\tr_{\bp \xi}$ as shorthand for the map
$$
\tr_{\bp \xi} \otimes I_{\Vtau}:
\End(\cH_{\bp \xi}) \otimes \Vtau \to \Vtau.
$$

\begin{lemma}
Let $\bp \xi \in \bp\dM$ and put $\xi  = \bp \xi \otimes \bp \xi^\vee.$
Furthermore, let $f \in C^\infty(K: \xi : \tau)$ and put $\bp \psi = \bp \psi_{f \otimes I_{\bp \xi}}.$
Then for all $\bp \gl \in \bp \fadc,$
\begin{equation}
\label{e: second expression HC Eisenstein}
E_\HC(\bp Q : \bp \psi_{f \otimes I_{\bp \xi}}: \bp \gl)(\bp x) = \int_{\bp K} \tr_{\bp \xi}\left( [\,
(I \otimes \pirep_{\bp Q, \bp \xi^\vee , \bp \gl}(\bp x)) f ](\bp k, \bp k)\right)\,
\; d\,\bp k.
\end{equation}
\end{lemma}

\begin{proof}
We agree to write $f_{\gl}$ for the unique function in $C^\infty(G: \bp Q \times \bp Q: \xi: \gl)\otimes \Vtau$  whose restriction to $K$ equals $f.$

The function
$\bp \psi := \bp
\psi_{f \otimes I_{\bp\xi}}\in C^\infty_\xi(\bp M:\tau^0_M)$ is completely determined by
$$
\bp \psi(e) = \inp{f(e,e)}{I_{\bp \xi}}_{{\rm HS}}= \tr_{\bp \xi} [f(e,e)].
$$
In the second expression, we have used the bilinear map
$(\cH_\xi \otimes \Vtau) \times \bar \cH_\xi \to \Vtau$
induced by the Hilbert-Schmidt inner product
on $\cH_\xi = \End(\cH_{\bp \xi}).$

We now observe that the function $\bp \psi_\gl$ defined by (\ref{e: definition bp psi gl})
can be expressed in terms of $f_{-\gl}$ in the following fashion;
\begin{equation}
\label{e: psi in terms of f}
\bp\psi_\gl(\bp x) = \tr_{\bp \xi} [ f_{-\gl} (e,\bp x)], \qquad (\bp x \in \bp G).
\end{equation}
It follows from the sphericality of $f$ that
$$
\tr_{\bp \xi}[f_{-\gl} (\bp x \bp k_1, \bp y \bp k_2)] = \tau_1(\bp k_1)^{-1} \tr_{\bp \xi} [f_{-\gl}(\bp x, \bp y)] \tau_2(\bp k_2),
$$
for $\bp x, \bp y \in \bp G$ and $\bp k_1, \bp k_2 \in \bp K. $ We thus obtain from
(\ref{e: psi in terms of f}) that
$$
\tau_1(\bp k)^{-1}\bp\psi_{\bp \gl} (\bp k \bp x) =  \tr_{\bp \xi}[ f_{-\gl} (\bp k, \bp k \bp x) ]
= \tr_{\bp \xi}\big(\big[ (I \otimes \pirep_{\bp Q, \bp \xi^\vee , \bp \gl} (\bp x))f\big](\bp k, \bp k ) \big).
$$
Equation (\ref{e: second expression HC Eisenstein}) now follows from
(\ref{e: first defi Eisenstein integral}).
\end{proof}

The $\fh$-extreme parabolic subgroups in $\cP(A)$ are the parabolic
subgroups of the form  $\bp P \times \bp P$ with $\bp P \in \cP(\bp A).$ For these parabolic
subgroups our Eisenstein integrals essentially coincide with the unnormalized Eisenstein
integrals of Harish-Chandra. More precisely, the following result is valid.

\begin{cor}
\label{c: char of HC Eisenstein integral}
Let $\bp P \in \cP(\bp A)$ and $\psi \in C^\infty(M/H_{M}:\tau^0_M).$
Then for all $\bp x, \bp y \in \bp G$ we have
\begin{equation}
\label{e: final equality eisenstein}
E(\bp P \times \bp P : \psi: \gl) (\bp x, \bp y )
=E_\HC(\bp P : \bp\psi : \bp \gl)(\bp x\, \bp y^{-1}),
\end{equation}
with $\gl = (\bp \gl, - \bp \gl),$
as an identity of meromorphic $\Vtau$-valued functions of $\;\bp\gl \in \bp\fadc.$
\end{cor}

\begin{proof}
The space $C^\infty(M/H_{M}: \tau_M^0)$ is spanned by the functions of the
form $\psi_{f \otimes I_{\bp \xi}},$ where
$\bp \xi \in \bp M^\wedge,$ $\xi = \bp \xi \otimes \bp \xi^\vee$ and $f \in C(K:\xi:\tau).$
By linearity it therefore
suffices to establish (\ref{e: final equality eisenstein})
for $\psi = \psi_{f \otimes I_{\bp \xi}},$ with $\bp \xi$ and $f$ as mentioned.
Moreover, by right $H$-invariance of the Eisenstein integral on the left-hand side,
it suffices to prove the result for $\bp y =e.$ The claim now follows by comparison of (\ref{e: second expression Eisenstein integral group case}) and (\ref{e: second expression HC Eisenstein}).
\end{proof}

\begin{rem}
\label{r: holomorphy HC Eisenstein}
In particular, we see that the Eisenstein integral on the left is holomorphic as a function of
$\gl \in \faqdc.$ As $\gS(\bp P \times\bp P)_- = \emptyset,$ this can also be derived by combining Theorem \ref{t: eisenstein as matrix coefficient} with Remark \ref{r: holomorphy j in fh extreme case}.
\end{rem}

\begin{cor}
\label{c: general cor Eisenstein integral in group case}
With notation as in Corollary \ref{c: first expression Eisenstein integral group case},
let $f \in C^\infty(K:\xi:\tau).$  Let $\psi_{f \otimes I_{\bp \xi}} \in C^\infty_\xi(M/H_{M}; \tau_M)$ be defined as in (\ref{e: first iso f to psi f}). Then
\begin{eqnarray}
\label{e: third expression Eisenstein integral group case}
\lefteqn{\!\!\!\!\!\!\!\!E(\bp P \times \bp Q : \psi_{f \otimes I_{\bp \xi}}: \gl) (\bp x, \bp y ) =\qquad }\\
\nonumber
\qquad &=&
E_\HC(\bp Q : \bp\psi_{[(A(\bp Q : \bp P :\bp\xi : - \bp\gl)\otimes I)f ]\otimes I_{\bp \xi}}: \bp \gl)(\bp x\, \bp y^{-1}),
\end{eqnarray}
for generic $\bp \gl \in \bp \fadc,$ $\gl = (\bp \gl, - \bp \gl)$ and all $\bp x, \bp y \in \bp G.$
\end{cor}

\begin{proof}
By right $H$-invariance of the Eisenstein integral on the left-hand side,
it suffices to prove the result for $\bp y =e.$
It follows from (\ref{e: second expression Eisenstein integral group case}) that
$$
E(\bp P \times \bp Q : \psi_{f \otimes I_{\bp \xi}}: \gl) (\bp x, e)
=E(\bp Q \times \bp Q : \psi_{[(A(\bp Q : \bp P :\bp\xi : - \bp\gl)\otimes I)f ]\otimes I_{\bp \xi}}: \gl) (\bp x, e).
$$
The identity now follows from (\ref{e: final equality eisenstein}).
\end{proof}

\section*{Appendix: Fubini's theorem for densities}
\label{s: appendix}
\addcontentsline{toc}{section}{Appendix: Fubini's theorem for densities}
\setcounter{section}{1}
\setcounter{Thm}{0}
\setcounter{equation}{0}
\renewcommand{\thesection} {\Alph{section}}
In this appendix our purpose is to establish a Fubini type theorem for repeated integration in the setting of
a  Lie group  $G$ with two closed subgroups $H$ and $L$ such that $H \subseteq L.$ The Fubini theorem concerns repeated integration for densities on the total space of the natural fiber bundle
\begin{equation}
\label{e:  fiber bundle G H L}
\pi: L \bs G \to H \bs G,
\end{equation}
with fibers diffeomorphic to $H\bs L.$
It expresses the integral over the total space as an iterated integration, first over the fibers and then
over the base space. In case of unimodular groups there is a well known version
of such a Fubini theorem involving invariant densities on the quotient spaces.
In the case of non-unimodular groups such densities do not exist. Nevertheless,
in this setting an appropriate formulation of iterated integration can be given as well.

To describe it, we will first formulate and establish a Fubini theorem for general fiber bundles,
and then specialize to the above situation.

If $V$ is real linear space of finite dimension $n$, then by $\cD_V$ we denote
the space of complex-valued densities on $V,$ i.e., the (complex linear) space of functions
$\gl: \wedge^{n}(V) \to \C$ transforming according to the rule
$\gl(t \xi) = |t| \gl(\xi),$ for all $t \in \R$ and $\xi \in \wedge^{n} V.$ A density $\gl$ is said
to be positive if $\gl(\xi) > 0$ for all non-zero $\xi \in \wedge^n V.$
By pull-back under the natural map $V^n \to \wedge^n V$ we see that
a density may also be viewed as a map $V^n \to \C$ transforming according to the
rule
$\gl \after T^n = |\det T | \gl,$ for all $T \in \End(V).$ This will be our viewpoint from
now on. Note that $\cD_V$
has dimension $1$ over $\C.$
If $W$ is a second real linear space of the same dimension $n$ and $A: V \to W$
a linear map, then pull-back under $A$ is the map $A^*: \cD_W \to \cD_V $ defined by
$$
A^*\mu = \mu\after A^n, \qquad (\mu \in \cD_W).
$$

\begin{lemma}
Let $E, F$ be finite dimensional real linear spaces.
 Then $\cD_{E\oplus F} \simeq \cD_E \otimes \cD_F$ naturally.
\end{lemma}

\begin{proof}
Let $p$ and $q$ be the dimensions of $E$ and $F$ respectively and put $n = p+q.$
We consider the natural isomorphism $\mu: \wedge^p E \otimes \wedge^q F \to \wedge^n(E\oplus F).$
Given $\ga \in \cD_{E}$ and $\gb \in \cD_F,$ we define $\ga \boxtimes \gb: \wedge^p E \otimes \wedge^q F \to \C$
by $\ga\boxtimes \gb (\xi\times \eta) = \gl(\xi) \mu(\eta).$ We note that this definition is
unambiguous, and that $(\ga \boxtimes \gb) \after (t\dotvar) = |t|(\ga \boxtimes \gb),$
so that $(\ga, \gb) \mapsto \ga \boxtimes \gb \after \mu^{-1}$ defines a bilinear map
from $\cD_E \times \cD_F$ to $\cD_{E \oplus F}.$ The induced map $\cD_E \otimes \cD_F \to \cD_{E \oplus F}$
is a non-trivial linear map between one dimensional complex linear spaces, hence a linear isomorphism.
\end{proof}

From now on we shall identify $\cD_{E\oplus F}$ with $\cD_E \otimes \cD_F$ via
the isomorphism given in the proof of the above lemma.

The lemma can be generalized to the setting of short exact sequences as follows.
Let
\begin{equation}
\label{e: short exact sequence}
0 \to E'\; {\buildrel i \over \longrightarrow}\; E \;{\buildrel p \over \longrightarrow}  \; E'' \to 0
\end{equation}
 be a short exact sequence of finite dimensional real linear spaces of dimensions
$k, n$ and $n-k.$ We recall that a linear map $f: E'' \to E$ is said to be splitting if $p\after f = {\rm id}_{E''}.$
Associated with $f$ is an isomorphism $i \oplus f: E \oplus E'' \to E,$ which by pull-back
induces a natural isomorphism
\begin{equation}
\label{e: iso induced by splitting}
(i \oplus f)^*: \cD_{E'} \otimes \cD_{E''} = \cD_{E' \oplus E''} \;{\buildrel \simeq \over \longrightarrow}\; \cD_E.
\end{equation}

\begin{lemma}
\label{l: iso independent of splitting}
The isomorphism (\ref{e: iso induced by splitting})
is independent of the splitting map $f.$
\end{lemma}

\begin{proof}
Let $g$ be a second splitting map. Then
$(i \oplus f)^* - (i \oplus g)^* = (i \oplus (f-g))^*.$  Now $f - g$ maps $E''$ into $\ker p = i(E')$ so that
$i \oplus (f-g)$ maps $E' \oplus E''$ into the subspace $i(E')$ of $E.$ It follows that $(i \oplus (f-g))^* = 0$
so that $(i \oplus f)^* = (i \oplus g)^* .$
\end{proof}

From now on, given a short exact sequence of the form (\ref{e: short exact sequence}) we shall identify elements of
the spaces $\cD_{E'} \otimes \cD_{E''} $ and $\cD_{E}$ via the isomorphism $(i \oplus f)^*,$ which
is independent of the choice of $f.$

We now turn to manifolds. Let $M,N$ be smooth manifolds and  $\gf: M  \to N$ a smooth map.
Then by $T\gf: TM \to TN$ we denote the induced map between the tangent bundles.
For a given $x \in M,$ this map restricts to the tangent map $T_x \gf: T_x M \to T_{\gf(x)} N,$
which will also be denoted by $d\gf(x).$

By $\cD_M$ we denote the complex line bundle of densities on $M.$
The fiber of this bundle at a point $x \in M$ is equal to $\cD_{T_x M}.$
The space of continuous densities is denoted by $\Gamma(\cD_M).$
If $\dim M = \dim N$ then the smooth map $\gf: M \to N$ induces a pull-back map $\gf^*: \Gamma(\cD_N) \to \Gamma(\cD_M),$ given by
$$
\gf^*(\mu)_x = d\gf(x)^* \mu_{\gf(x)},\qquad (\mu \in \Gamma(\cD_N),\; x \in M).
$$
There is notion of integration of compactly supported continuous densities on manifolds
for which the substitution of variables theorem is valid. More precisely, if $\gf: M \to N$ is a diffeomorphism of smooth manifolds, then
\begin{equation}
\label{e: integral of density}
\int_N \mu = \int_M \gf^*(\mu), \qquad (\mu \in \Gamma(N)).
\end{equation}

Let $\pi: F\to B$ be a smooth fiber bundle. Let $\cD_F$ denote the density bundle on $F.$
We may introduce a bundle of fiber densities on $F$ as follows.
The map $\pi$ induces the homomorphism $T\pi: TF \to TB$ of vector bundles. The kernel $K = \ker T\pi$
of this bundle is a subbundle of $TF.$ Obviously, the fiber of $K$ at $p \in F$ may be viewed
as the tangent space of the fiber $F_{\pi(p)}$ at the point $p.$
The associated bundle $p \mapsto \cD_{K_p}$ is a smooth complex line bundle on $F,$ which we shall
call the bundle of fiber densities on $F.$ We shall denote this bundle by $\cD^B_F.$

On the other hand, the fiber product or pull-back bundle
 $\pi^*(\cD_B) := F \times_\pi \cD_B$ of $\cD_B$ under $\pi$ is
a complex line bundle on $F.$ We shall denote the associated canonical line bundle homomorphism
$\pi^*(\cD_B) \to \cD_B$ by $\tilde \pi.$

The short exact sequence  $0 \to K \to TF \to \pi^*(TB) \to 0$ of vector bundles on $F$
naturally induces a line bundle isomorphism
\begin{equation}
\label{e: natural isomorphism density bundle for fibration}
\cD^B_F \otimes \pi^*(\cD_B) \simeq \cD_F,
\end{equation}
via which we shall identify elements of these spaces. Here naturality means
that for a fiber bundle morphism $\gf$ from $\pi$ to a bundle $\pi': F' \to B'$ with
$\dim F' = \dim F$ and $\dim B' = \dim B$ the following diagram commutes:
\begin{equation}
\label{e: naturality tensor product}
\begin{array}{ccc}
\cD^{B}_{F}\otimes \pi^*(\cD_{B})\!\!\!\!\!\! \!&  \buildrel \simeq \over \longrightarrow & \!\!\!\cD_{F}\\
\scriptstyle{(T\gf)^* \otimes (T\gf)^*}\; \downarrow \qquad\qquad\qquad&& \qquad\downarrow \; \scriptstyle{(T\gf)^*}\\
 \cD^{B'}_{F'}  \otimes  (\pi')^*(\cD_{B'})&  \buildrel \simeq \over \longrightarrow & \cD_{F'}
\end{array}
\end{equation}
Let now $b \in B$ and let $F_b$ the fiber $\pi^{-1}(b)$ of $F$ above $b.$
Obviously, the restriction of $\cD^B_F$ to this fiber is naturally isomorphic
to $\cD_{F_b},$ the density bundle of the fiber.
On the other hand, via $\tilde \pi$ the restriction of the bundle $\pi^*(\cD_B)$ to $F_b$  may be identified with the trivial bundle $F_b \times \cD_{T_bB}.$
Accordingly, we obtain natural isomorphisms
$
\cD_{F_b} \otimes \cD_{T_bB} \simeq \cD_F|_{F_b},
$
and
$$
\Gamma(\cD_F|_{F_b}) \simeq \Gamma(\cD_{F_b}) \otimes_\C \cD_{T_b B}.
$$
Integration over the fiber gives a natural linear map
$$
I_b: \Gamma_c(\cD_{F_b}) \to \C, \qquad\mu \mapsto \int_{F_b} \mu.
$$
By transfer we obtain a natural map $I_b \otimes {\rm id}: \Gamma_c(\cD_F|_{F_b}) \to \cD_{T_bB}.$
We now define the push-forward map $\pi_*: \Gamma_c(\cD_F) \to {\rm sect}(\cD_B),$ by
\begin{equation}
\label{e: defi push forward}
\pi_*(\mu)(b) := (I_b \otimes {\rm id})(\mu|_{F_b}).
\end{equation}
Here ${\rm sect}(\cD_B)$ denotes the space of all (not-necessarily continuous) sections of $\cD_B.$

By the naturality of the constructions and the invariance of integration as formulated in
(\ref{e: integral of density}), one readily checks that the notion of push-ward of compactly supported densities is invariant under isomorphisms of bundles.

\begin{lemma}
\label{l: naturality push forward}
Let $\gf $ be an isomorphism from the fiber bundle $\pi: F \to B$ to a second fiber bundle
$\pi': F' \to B'$ and let $\gf_\circ$ denote the induced diffeomorphism $B \to B'.$ Then the following diagram commutes:
$$
\begin{array}{ccc}
\Gamma_c(\cD_{F'}) &{\buildrel \gf^* \over \longrightarrow }& \Gamma_c(\cD_F)\\
\pi'_*\downarrow&& \downarrow \pi_*\\
\Gamma_c(\cD_{B'}) & {\buildrel \gf_\circ^* \over \longrightarrow }& \Gamma_c(\cD_B)
\end{array}
$$
\end{lemma}
We can now establish the following Fubini type theorem for the
integration of densities over fiber bundles.

\begin{lemma}
\label{l: fubini for continuous densities}
The map $\pi_*$ maps $\Gamma_c(\cD_F)$ {\rm  (}respectively $\Gamma_c^\infty(\cD_F)${\rm )} continuous linearly to $\Gamma_c(\cD_B)$
{\rm (}respectively $\Gamma_c^\infty(\cD_B)${\rm)}.
Moreover, for all $\mu \in \Gamma_c(\cD_F),$
\begin{equation}
\label{e: fubini for densities}
\int_F \mu = \int_B \pi_*(\mu).
\end{equation}
\end{lemma}

\begin{proof}
By using partitions of unity, and invoking invariance of integration, cf.\ (\ref{e: integral of density}),
and
Lemma \ref{l: naturality push forward},
we may reduce the proof to the case that $B$ is open in $\R^n$ and
that $F = B \times V,$ with $V$ an open subset of Euclidean space $\R^k.$ In that case the result
comes to down to continuous and smooth parameter dependence and Fubini's theorem for Riemann integrals
of continuous functions.
\end{proof}

\begin{cor}
\label{c: Fubini for integrable density}
Let $\mu$ be a measurable section of $\cD_F.$ Then the following statements are equivalent.
\begin{enumerate}
\itema
The density $\mu$ is absolutely integrable.\vspace{-5pt}
\itemb
For almost every $b \in B$ the integral for $\pi_*(\mu)_b$ is absolutely convergent and the resulting density $\pi_*(\mu)$ is absolutely integrable over $B.$
\end{enumerate}
If any of these conditions is fulfilled, then (\ref{e: fubini for densities}) is valid.
\end{cor}

\begin{proof}
This follows by reduction to Fubini's theorem through the use of
partitions of unity, as in the proof of Lemma \ref{l: fubini for continuous densities}
\end{proof}

We will now apply the above result to the particular setting of a Lie group
$G$ with closed subgroups $H$ and $L$ such that $H \subseteq L.$ As said at the start of this
appendix, this setting gives rise to the natural fiber bundle (\ref{e: fiber bundle G H L})
with fiber diffeomorphic to $L\bs H.$

Let $\Delta_{\scriptscriptstyle L\bs G}: L \to \R_+$ be the positive character given by
\begin{equation}
\label{e:  defi Delta sub}
\Delta_{\scriptscriptstyle L\bs G}(l) = |\det \Ad_G(l)_{\fg/\fl}|^{-1},\qquad (l \in L),
\end{equation}
where $\Ad_G(l)_{\fg/\fl} \in \GL(\fg/\fl)$ denotes the map induced by the adjoint map $\Ad_G(l) \in \GL(\fg).$
Given a character $\xi$ of $L$ we denote by  $C(G:L:\xi)$ the space of continuous  functions
$f: G \to \C$ transforming according to the rule
$$
f(l x) = \xi(l) f(x),
$$
for $x \in G$ and $l \in L.$
We denote by $\cM(G)$ the space of  measurable functions $G \to \C$ and by $\cM(G:L:\xi)$ the space of $f \in \cM(G)$ transforming according to the same rule.

Given  $f \in \cM(G)$  and
$\omega \in \cD_{\fg/\fl},$ we denote by $f_\omega$ the function $G \to \cD_{L\bs G}$ defined
by
$$
f_\omega(x) = f(x) \, dr_x(e)^{-1*} \omega.
$$

\begin{lemma}
Let $\omega  \in \cD_{\fg/\fl}\setminus \{0\}.$ Then the map $f \mapsto f_\omega$ defines
a continuous linear isomorphism from $C(G:L:\Delta_{\scriptscriptstyle L\bs G})$ onto $\Gamma(\cD_{L\bs G}).$
\end{lemma}

\begin{proof}
Write $\Delta = \Delta_{\scriptscriptstyle L\bs G}.$ In the proof we will use the notation $[e]$ for the image of $e$ in $L\bs G.$
Moreover, we will use the canonical identification $T_{[e]}(L\bs G) \simeq \fg/\fl.$
Let $\omega$ be as stated, and let $f \in C(G:L:\Delta).$ Then for $x \in G$ we have
$f_\omega(x) \in \cD_{T_{[x]}(L\bs G)}.$ Let $l \in L,$ then
\begin{eqnarray*}
f_\omega(lx) & = & \Delta(l) f(x) \,dr_{lx}([e])^{-1*}\omega\\
& =& \Delta(l) f(x) dr_x([e])^{-1*}\, dr_l([e])^{-1*}\omega\\
&=& \Delta(l) f(x) \,dr_x([e])^{-1*} \Ad(l)^* \omega\\
&=& f_\omega(x)
\end{eqnarray*}
It follows that $f_\omega$ factors through a smooth map $L\bs G \to \cD_{L\bs G},$ with
$f_\omega(x)$ a density on $T_{[x]}(L\bs G).$ Accordingly, $f_\omega$ defines a section
of $\cD_{L\bs G},$ which clearly is continuous. The bijectivity of the map $f \mapsto f_\omega$
from $C(G:L:\Delta)$ onto $\Gamma(\cD_{L \bs G})$ is obvious.
\end{proof}

Our next goal is to calculate the push-forward $\pi_*(f_\omega),$ for
$\omega \in \cD_{\fg/\fl}$ and $f \in C(G:L: \Delta_{\scriptscriptstyle L\bs G}),$ and $\pi: L \bs G \to H \bs G$
the canonical projection.

We note that $\pi$ is a fiber bundle with total space $F = L\bs G,$ base space $B = H\bs G$ and fiber diffeomorphic to $L \bs H.$
Thus, we have the natural isomorphism (\ref{e: natural isomorphism density bundle for fibration}).

If $x \in G,$ then the diffeomorphism $r_x^F: F \to F, z \mapsto zx$ defines an isomorphism of fiber bundles over the diffeomorphism $r_x^B$ defined by right multiplication on $B,$ i.e.,
the following diagram commutes,
$$
\begin{array}{ccc}
F&  {\buildrel r^F_x \over \longrightarrow }&  F\\
\downarrow && \downarrow\\
B & {\buildrel r^B_x \over  \longrightarrow} & B.
\end{array}
$$
In the sequel we shall use the
commutativity of the diagram (\ref{e: naturality tensor product}) with $F = F',$ $B= B'$ and $\gf = r^F_x.$

In particular, it follows that $(dr_x^F)^* \otimes (dr_x^F)^*\in \End(\cD^B_F\otimes \pi^*\cD_B)$
corresponds to the naturally induced automorphism
$(dr_x^F)^*$ of $\cD_F.$

We fix non-zero elements $\omega_{\scriptscriptstyle L \bs G}\in \cD_{\fg/\fl},$ $\omega_{\scriptscriptstyle L \bs H} \in \cD_{\fh/\fl}$
and $\omega_{\scriptscriptstyle H\bs G} \in \cD_{\fg/\fh}$ such that
\begin{equation}
\label{e: tensor product of omegas}
\omega_{\scriptscriptstyle L\bs H}  \otimes \omega_{\scriptscriptstyle H\bs G}= \omega_{\scriptscriptstyle L\bs G}
\end{equation}
with respect to the identification determined by the short
exact sequence $0 \to \fh/\fl \to \fg/\fl \to \fg/\fh \to 0.$
This short exact sequence may be identified with the
short exact sequence of tangent spaces
$$
0 \to T_{[e]}(L\bs H) \;{\buildrel di([e]) \over\longrightarrow}\; T_{[e]}(L\bs G)\; {\buildrel d\pi([e]) \over \longrightarrow} \;
T_{[e]}(H\bs G) \to 0,
$$
where $i: L \bs H \embeds L\bs G $ denotes the natural embedding
of $L\bs H$ onto the fiber $\pi^{-1}([e]).$
Accordingly, formula (\ref{e: tensor product of omegas}) may be viewed as an identity of elements associated with
the decomposition
$$
(\cD^B_F)_{[e]} \otimes (\cD_B)_{[e]} = (\cD_F)_{[e]}.
$$
\begin{lemma}
Let $\omega_{\scriptscriptstyle L \bs H} ,\omega_{\scriptscriptstyle H\bs G}$ and $\omega_{\scriptscriptstyle L\bs H}$ satisfy (\ref{e: tensor product of omegas}). Then for all $h \in H$ and $x \in G,$
\begin{equation}
\label{e: decomposition translate omega L}
dr_{hx}([e])^{-1*}\omega_{\scriptscriptstyle  L\bs H} = \Delta_{\scriptscriptstyle H\bs G}(h)^{-1} \left( dr_{hx}([e])^{-1*}\omega_{\scriptscriptstyle L\bs H}  \otimes dr_{x}([e])^{-1*}\omega_{\scriptscriptstyle H\bs G} \right),
\end{equation}
in accordance with the decomposition
$(\cD^B_F)_{[hx]} \otimes (\cD_B)_{[x]} = (\cD_F)_{[hx]},$ corresponding to (\ref{e: natural isomorphism density bundle for fibration}).

\end{lemma}
\begin{proof}
Let $h \in H,$ then
$
dr_h(e)^{-1*}(\omega_\ssHG) = \Ad(h)^*\omega_\ssHG = \Delta_{\scriptscriptstyle H\bs G }(h)^{-1} \omega_\ssHG
$
and we see that
\begin{equation}
\label{e: tensor product deco at e}
dr_h([e])^{-1*}\omega_\ssLG = \Delta_{\scriptscriptstyle H\bs G}(h)^{-1} \left( dr_h([e])^{-1*}\omega_\ssLH \otimes \omega_\ssHG \right).
\end{equation}
Let now $x \in G,$ then in view of the $G$-equivariance of
the fiber bundle $F \to B$ formula  (\ref{e: decomposition translate omega L}) follows by application of $dr_x([h])^{-1*}$ to both sides of the identity
 (\ref{e: tensor product deco at e}).
\end{proof}

\begin{thm}
\label{t: fubini for group fibering}
Let $\omega_\ssLG ,\omega_\ssHG $ and $\omega_\ssLH$ satisfy (\ref{e: tensor product of omegas}). Let $\gf \in \cM(G:L:\Delta_{\scriptscriptstyle L\bs G})$ and let $\gf_{\omega_\ssLG}$ be the associated measurable density on $L\bs G.$
Then the following assertions (a) and (b) are equivalent
\begin{enumerate}
\itema
The density $\gf_{\omega_\ssLG}$ is absolutely integrable.\vspace{-5pt}
\itemb
There exists a left $H$-invariant subset $\cZ$ of measure zero in $G$ such that
\begin{enumerate}
\item[{\rm (1)}]
for every $x \in G\setminus \cZ,$ the integral
\begin{equation}
\label{e: fiber integral Ix}
I_x(\gf) = \int_{L\bs H\ni [h]} \Delta_{\scriptscriptstyle H\bs G}(h)^{-1}\gf(hx)\, d r_{h}([e])^{-1*} \omega_\ssLH,
\end{equation}
is  absolutely convergent;
\item[{\rm (2)}]
the function $I(\gf): x \mapsto I_x(\gf)$ belongs to
$\cM(G:H:\Delta_{\scriptscriptstyle H\bs G});$
\item[{\rm (3)}]
the associated density
$I(\gf)_{\omega_{\scriptscriptstyle H\bs G}}$ is absolutely integrable.
\end{enumerate}
\end{enumerate}
Furthermore, if any of the conditions (a) and (b) are fulfilled, then
\begin{equation}
\label{e: fubini identity G L H}
\int_{L\bs G} \gf_{\omega_{\scriptscriptstyle L \bs G}} = \int_{H \bs G} I(\gf)_{\omega_\ssHG }.
\end{equation}
\end{thm}
\begin{proof}
We retain the notation introduced before the statement of the theorem.
Then for $x \in G$ and $h \in H$ the associated density at $Lhx$ is given by
\begin{equation}
\label{e: formula for gf omega LG}
\gf_{\omega_{\ssLG}}(hx) =
\Delta_{\scriptscriptstyle L\bs H}(h)^{-1} \gf(hx) \left( dr_{hx}([e])^{-1*}\omega_{\scriptscriptstyle L\bs H}  \otimes dr_{x}([e])^{-1*}\omega_\ssHG \right),
\end{equation}
in accordance with the decomposition corresponding to (\ref{e: natural isomorphism density bundle for fibration}).

We
will deduce the result from applying Corollary \ref{c: Fubini for integrable density}
to the fibre bundle given by the canonical projection
 $\pi: F:= L\bs G \to B:= H \bs G$ and to the measurable density
 $\mu := \gf_{\omega_\ssLG}$ on $F.$

 The crucial step is to prove the claim that for $x \in G,$ the integral  for the push-forward $\pi_*(\gf_{\omega_\ssLG})(Hx)$ converges absolutely if and only
if the integral for $I_x(\gf)$ converges absolutely. We will first establish this
claim.

It follows from (\ref{e: defi push forward}) that the  push-forward of $\gf_{\omega_{\ssLG}}$ under $\pi$ is the
 density on $H\bs G$ given by the following fiber integral
\begin{equation}
\label{e: fiber integral}
\pi_*(\gf_{\omega_{\ssLG}})(Hx) = \left( \int_{\pi^{-1}(Hx)} \nu_x \right) \otimes  dr_{x}([e])^{-1*}\omega_\ssHG,
\end{equation}
where $\pi^{-1}(Hx) = r_x(L\bs H),$ and where $\nu_x$ is the density on $r_x(L\bs H)$ given by
$$
\nu_x(L hx) = \Delta_{\scriptscriptstyle L\bs H}(h)^{-1} \gf(hx)\, dr_{hx}([e])^{-1*}\omega_{\scriptscriptstyle L\bs H}.
$$
The convergence and value of this integral depends on $x$ through its class
$H x.$
We now observe that $r_x$ defines a diffeomorphism from the fiber $\pi^{-1}(He)$ onto the fiber $\pi^{-1}(Hx).$ Moreover,
$$
[r_x^* \nu_x](Lh) = dr_x([h])^{-1*} \nu_x(Lhx) =
\Delta_{\scriptscriptstyle L\bs H}(h)^{-1} \gf(hx) \, dr_{h}([e])^{-1*}\omega_{\scriptscriptstyle L\bs H}.
$$
Thus,  $I_x(\gf)$ equals the integral of $r_x^*\nu_x$ over $L\bs H,$
and by invariance of integration we see that it converges absolutely
if and only if the integral for $\pi_*(\gf_{\omega_\ssLG})(Hx)$
converges absolutely. Moreover, in case of convergence we have
$$
I_x(\gf) = \int_{L\bs H} r_x^*(\nu_x) = \int_{\pi^{-1}(Hx)} \nu_x,
$$
so that
\begin{equation}
\label{e: I x and pi star}
\pi_*(\gf_{\omega_\ssLG})(Hx) = I_x(\gf) \,dr_x([e])^{-1*}\omega_\ssHG = I(\gf)_{\omega_\ssHG}(Hx).
\end{equation}
This establishes the claim.

The equivalence of (a) and (b) now readily follows from the similar equivalence in 
Corollary \ref{c: Fubini for integrable density}. Finally, if any of these conditions is fulifilled, both are, and in view of
(\ref{e: I x and pi star}), the identity (\ref{e: fubini identity G L H}) follows from the final assertion of Corollary \ref{c: Fubini for integrable density}.
\end{proof}

\bibliographystyle{plain}
\newcommand{\SortNoop}[1]{}

\def\adritem#1{\hbox{\small #1}}
\def\distance{\hbox{\hspace{3.5cm}}}
\def\apetail{@}
\def\addVdBan{\vbox{
\adritem{E.~P.~van den Ban}
\adritem{Mathematical Institute}
\adritem{Utrecht University}
\adritem{PO Box 80 010}
\adritem{3508 TA Utrecht}
\adritem{The Netherlands}
\adritem{E-mail: E.P.vandenBan{\apetail}uu.nl}
}
}
\def\addKuit{\vbox{
\adritem{J.~J.~Kuit}
\adritem{Department of Mathematical Sciences}
\adritem{University of Copenhagen}
\adritem{Universitetsparken 5}
\adritem{2100 K\o benhavn \O}
\adritem{Denmark}
\adritem{E-mail: j.j.kuit{\apetail}gmail.com}
}
}
\mbox{}
\vfill
\hbox{\vbox{\addVdBan}\vbox{\distance}\vbox{\addKuit}}

\end{document}